\documentclass[12pt,reqno,oneside]{amsart}
\usepackage{fullpage}
\usepackage{amsmath,amssymb,amsthm,enumitem,hyperref,cite,bbm}

\newtheorem{thm}{Theorem}
\newtheorem{lemma}[thm]{Lemma}
\newtheorem{prop}[thm]{Proposition}
\newtheorem{cor}[thm]{Corollary}
\newtheorem{obs}[thm]{Observation}

\theoremstyle{definition}
\newtheorem{defn}[thm]{Definition}
\newtheorem{alg}[thm]{Algorithm}
\newtheorem{rmk}[thm]{Remark}

\numberwithin{thm}{section}
\setlist[itemize]{leftmargin=1cm,itemsep=4pt}
\setlist[enumerate]{leftmargin=1.5cm}

\everydisplay{\thickmuskip=7mu}
\newcommand\til{\kern -.1em\lower .8ex\hbox{\~{}}}

\DeclareRobustCommand{\gobblefive}[5]{}
\newcommand*{\SkipTocEntry}{\addtocontents{toc}{\gobblefive}}

\renewcommand\le{\leqslant}
\renewcommand\ge{\geqslant}

\newcommand\N{\mathbb{N}}

\newcommand\F{\mathbb{F}}

\newcommand\cA{\mathcal{A}}
\newcommand\cB{\mathcal{B}}
\newcommand\cC{\mathcal{C}}
\newcommand\cD{\mathcal{D}}
\newcommand\cE{\mathcal{E}}
\newcommand\cF{\mathcal{F}}
\newcommand\cG{\mathcal{G}}
\newcommand\cH{\mathcal{H}}
\newcommand\cK{\mathcal{K}}
\newcommand\cL{\mathcal{L}}
\newcommand\cM{\mathcal{M}}
\newcommand\cQ{\mathcal{Q}}
\newcommand\cR{\mathcal{R}}
\newcommand\cS{\mathcal{S}}
\newcommand\cT{\mathcal{T}}
\newcommand\cU{\mathcal{U}}
\newcommand\cW{\mathcal{W}}
\newcommand\cX{\mathcal{X}}
\newcommand\tA{\tilde A}
\newcommand\tV{\tilde V}
\newcommand\tE{\tilde E}
\newcommand\tX{\tilde X}
\newcommand\ts{\hat s}
\newcommand\te{\tilde e}
\newcommand\tf{\tilde f}
\newcommand\tj{\tilde\jmath}

\newcommand\tD{\tilde D}
\newcommand\tS{\tilde S}
\newcommand\tR{\tilde R}
\newcommand\tW{\tilde W}
\newcommand\tU{\tilde U}
\newcommand\tY{\tilde Y}
\newcommand\tZ{\tilde Z}

\newcommand\tcF{\tilde\cF}
\newcommand\tcS{\tilde\cS}
\newcommand\tcX{\tilde\cX}
\newcommand\tpsi{\widetilde\Psi}
\newcommand\tdelta{\tilde\delta}
\newcommand\tDelta{\tilde\Delta}
\newcommand\tp{\tilde p}
\newcommand\ds{\displaystyle}

\newcommand\Prb{\mathbb{P}}
\newcommand\E{\mathbb{E}}
\newcommand\eps{\varepsilon}
\newcommand\Ein{\textup{Ein}}
\newcommand\Bin{\textup{Bin}}

\newcommand\Var{\textup{Var}}
\newcommand\bmid{\,\big|\,}
\newcommand\Bmid{\;\Big|\;}
\newcommand\id{\mathbbm{1}}
\newcommand\dispqed{\vskip-10pt\qed}
\newcommand\Th[1]{\text{Theorem~\ref{thm:#1}}}
\newcommand\Co[1]{\text{Corollary~\ref{cor:#1}}}
\newcommand\Sc[1]{\text{Section~\ref{sec:#1}}}
\newcommand\Lm[1]{\text{Lemma~\ref{lem:#1}}}
\newcommand\Pp[1]{\text{Proposition~\ref{prop:#1}}}
\newcommand\Df[1]{\text{Definition~\ref{def:#1}}}
\newcommand\Ob[1]{\text{Observation~\ref{obs:#1}}}
\newcommand\Rk[1]{\text{Remark~\ref{rem:#1}}}
\newcommand\Al[1]{\text{Algorithm~\ref{alg:#1}}}
\newcommand\claim[1]{\medskip\noindent\textbf{#1}}

\title{The sharp threshold for making squares}

\author[P. Balister\and B. Bollob\'as\and R. Morris]
{Paul Balister\and B\'ela Bollob\'as\and Robert Morris}

\address{Department of Pure Mathematics and Mathematical Statistics,
 Wilberforce Road, Cambridge, CB3 0WA, UK,
 and Department of Mathematical Sciences, University of Memphis, Memphis, TN 38152, USA,
 and London Institute for Mathematical Sciences, 35a South Street, London, W1K 2XF, UK}
\email{b.bollobas@dpmms.cam.ac.uk}
\address{Department of Mathematical Sciences, University of Memphis, Memphis, TN 38152, USA}
\email{pbalistr@memphis.edu}
\address{IMPA, Estrada Dona Castorina 110, Jardim Bot\^anico, Rio de Janeiro, 22460-320, Brazil}
\email{rob@impa.br}

\thanks{The first two authors were partially supported by NSF grant DMS~1600742.
The second author was also supported by MULTIPLEX grant no.~317532.
The third author was partially supported by CNPq (Proc.~479032/2012-2 and Proc.~303275/2013-8).}

\date{\today}

 
\begin{document}

\begin{abstract}
Consider a random sequence of $N$ integers, each chosen uniformly and independently from the set $\{1,\dots,x\}$. Motivated by applications to factorisation algorithms such as Dixon's algorithm, the quadratic sieve, and the number field sieve, Pomerance in 1994 posed the following problem: how large should $N$ be so that, with high probability, this sequence contains a subsequence, the product of whose elements is a perfect square? Pomerance determined asymptotically the logarithm of the threshold for this event, and conjectured that it in fact exhibits a \emph{sharp threshold} in $N$. More recently, Croot, Granville, Pemantle and Tetali determined the threshold up to a factor of $4/\pi + o(1)$ as $x \to \infty$, and made a conjecture regarding the location of the sharp threshold.  

In this paper we prove both of these conjectures, by determining the sharp threshold for making squares. Our proof combines techniques from combinatorics, probability and analytic number theory; in particular, we use the so-called method of self-correcting martingales in order to control the size of the 2-core of the random hypergraph that encodes the prime factors of our random numbers. Our method also gives a new (and completely different) proof of the upper bound in the main theorem of Croot, Granville, Pemantle and Tetali.
\end{abstract}

\maketitle

\tableofcontents

\section{Introduction}

Many of the fastest known algorithms for factoring large integers rely on finding subsequences of
randomly generated sequences of integers whose product is a perfect square. Examples include
Dixon's algorithm~\cite{Dixon}, the quadratic sieve~\cite{P82}, and the number field sieve (see,
e.g.,~\cite{LL}); an excellent elementary introduction to the area is given by
Pomerance~\cite{Pnotices}. In each of these algorithms one generates a sequence of congruences of
the form
\[
 a_i\equiv b_i^2\pmod n,\qquad i=1,2,\dots
\]
and then one aims to find subsets of the $a_i$ whose product is a perfect square, say
$\prod_{i\in I}a_i=X^2$, so then one has $X^2\equiv Y^2\pmod n$ with $Y =\prod_{i\in I} b_i$.
If one is lucky then $X\not\equiv\pm Y\pmod n$, in which case one can generate non-trivial
factors of $n$ as $\gcd(X\pm Y,n)$.

A useful heuristic, suggested by Schroeppel in the 1970s (see~\cite{Pnotices}), is to imagine that the
numbers $a_i$ are chosen independently and uniformly at random from the set $\{1,\dots,x\}$, for
some suitably chosen integer $x$. Motivated by this idea, Pomerance~\cite{Picm} posed in 1994 the
problem of determining the \emph{threshold} for the event that such a collection of random numbers
contains a subset whose product is a square. To be precise, given $x\in\N$, let us define a
probability space $\Omega(x)$ by choosing $a_1,a_2,\dots$ independently and uniformly at random
from $\{1,\dots,x\}$, and a random variable $T(x)$ by setting
\[
 T(x):=\min\bigg\{N\in\N:\prod_{i\in I}a_i\textup{ is a perfect square for some }
 I\subseteq\big\{1,\dots,N\big\},\,I\ne\emptyset\bigg\}.
\]
Pomerance~\cite{P96} proved that for all $\eps>0$,
\begin{equation}\label{eq:Pbounds}
 \exp\Big(\big(1-\eps\big)\sqrt{2\log x\log\log x}\Big)\le T(x)\le
 \exp\Big(\big(1+\eps\big)\sqrt{2\log x\log\log x}\Big)
\end{equation}
with high probability\footnote{We use the term \emph{with high probability} to mean with
probability tending to 1 as $x\to\infty$.}, and conjectured that $T(x)$ in fact exhibits
a \emph{sharp threshold}, i.e., that there exists a function $f(x)$ such that
$(1-\eps)f(x)\le T(x)\le (1+\eps)f(x)$ with high probability for all $\eps>0$.
Croot, Granville, Pemantle and Tetali~\cite{CGPT} significantly improved these bounds
(see~\eqref{eq:cgpt}, below), and stated a conjecture as to the location of the threshold, i.e.,
the value of the function~$f(x)$. In this paper we shall prove these two conjectures.

In order to state the theorem and conjecture of Croot, Granville, Pemantle and Tetali,
we will need to recall some standard notation. Let $\pi(y)$ denote the number of primes less than
or equal to $y$, let $\Psi(x,y)$ denote the number of $y$-\emph{smooth} integers in
$\{1,\dots,x\}$, that is, the number of integers with no prime factor strictly greater than~$y$,
and define
\begin{equation}\label{def:J} 
 J(x)=\min_{2\le y\le x}\frac{\pi(y)x}{\Psi(x,y)}.
\end{equation}
It can be shown (see Section~\ref{sec:NTfacts}) that the minimum in \eqref{def:J} occurs at
\[
 y_0=y_0(x)=\exp\Big(\big(1+o(1)\big)\sqrt{\tfrac{1}{2}\log x\log\log x}\Big)
\]
and that
\[
 J(x)=y_0^{2+o(1)}=\exp\Big(\big(1+o(1)\big)\sqrt{2\log x\log\log x}\Big),
\]
and an asymptotic formula for $J(x)$ was obtained by McNew~\cite{McNew}. 
We remark that a relatively straightforward argument due to Schroeppel (see~\cite{P96}) shows that, for all $\eps>0$,
\[
 T(x)\le\big(1+\eps\big)J(x)
\]
with high probability, which implies the upper bound in~\eqref{eq:Pbounds}. Indeed, if
$N\ge (1+\eps)J(x)$ then with high probability at least $\pi(y_0)+1$ of the numbers
$a_1,\dots,a_N$ will be $y_0$-smooth, since each $a_i$ is $y_0$-smooth with probability
$\Psi(x,y_0)/x=\pi(y_0)/J(x)$. Now, by simple linear algebra, it follows that the vectors
encoding the primes that divide $a_i$ an odd number of times are linearly dependent over~$\F_2$,
and hence there exists a subset whose product is a square, as required.

Pomerance's conjecture remained wide open for over ten years, until a fundamental breakthrough
was obtained by Croot, Granville, Pemantle and Tetali~\cite{CGPT}, who used a combination of
techniques from number theory, probability theory and combinatorics to dramatically improve both
the upper bound of Schroeppel and the lower bound of Pomerance~\cite{P96}, determining the
location of the threshold to within a factor of $4/\pi$. To be precise, they proved that
\begin{equation}\label{eq:cgpt}
 \frac{\pi}{4}\big(e^{-\gamma}-\eps\big)J(x)\le T(x)\le\big(e^{-\gamma}+\eps\big)J(x)
\end{equation}
with high probability, where $\gamma\approx 0.5772$ is the Euler--Mascheroni constant. Recall
that $e^{-\gamma}$ is (amongst other things) the limit as $y \to \infty$ of the ratio of the
density of integers with no prime divisor smaller than~$y$, to the proportion of elements
of $\{1,\dots,y\}$ that are prime. 

Croot, Granville, Pemantle and Tetali~\cite{CGPT} conjectured that the upper bound in~\eqref{eq:cgpt} is sharp. Our main
theorem confirms their conjecture.

\begin{thm}\label{thm:squares:sharp}
 For all\/ $\eps>0$ we have with high probability
 \[
  \big(e^{-\gamma}-\eps\big)J(x)\le T(x)\le \big(e^{-\gamma}+\eps\big)J(x).
 \]
\end{thm}

As a simple corollary, we also deduce the following asymptotic expression for the expected
value of $T(x)$.
\begin{cor}\label{cor:expectation}
 $\E\big[T(x)\big]=\big(e^{-\gamma}+o(1)\big)J(x)$ as\/ $x\to\infty$.
\end{cor}

Since the upper bound in \Th{squares:sharp} was proved in~\cite{CGPT}, we are only required to
prove the lower bound. However, we will also obtain a new proof of the upper
bound, quite different from that given in~\cite{CGPT}, as a simple consequence of our method, see
Section~\ref{sec:squares:proof}. We would like to thank Jonathan Lee for pointing out to us a
particularly simple and natural way of deducing this from our proof.

Another significant advantage of our proof, which is outlined in \Sc{outline},
is that it gives detailed structural information about the typical properties of the set of numbers
that are left over after sieving and ``singleton removal'' (see, e.g.,~\cite{768bit}).
We plan to study this structure in a more general setting, and in greater detail,
in a follow-up paper together with Lee~\cite{BBLM}.

Croot, Granville, Pemantle and Tetali~\cite{CGPT} proved their lower bound in \eqref{eq:cgpt} via the first 
moment method, by counting the expected number of non-empty subsets $I\subseteq\{1,\dots,N\}$
such that $\prod_{i\in I} a_i$ is a square. Unfortunately, there exists a constant $c>0$ such
that this expected number blows up when $N\ge(e^{-\gamma}-c)J(x)$, which implies that a
sharp lower bound cannot be obtained by this method (see the comments after the proof
of \Th{squares:sharp} in \Sc{squares:proof}).

Instead, we shall use the method of self-correcting martingales\footnote{This technique is based
on the so-called `differential equations method' (see e.g.~\cite{K,W}), and involves the use of martingales
to control a collection of interacting random variables that exhibit `self-correction' in a certain natural sense
(see Sections~\ref{sec:big:z} and~\ref{sec:critical}). It was introduced in~\cite{BFL1,BP,TWZ}, and has
more recently been further developed in~\cite{BFL2,BK,FGM}; our approach is in particular based on
that used in~\cite{FGM}.} to follow a random process which removes numbers from the set\footnote{This
is, strictly speaking, a multi-set, since the numbers $a_i$ are chosen independently with replacement.
However, since we are very unlikely to choose the same number twice
(indeed, if we do so we immediately have a square), we shall ignore this possibility for
the sake of this discussion.} $\{ a_1,\dots, a_N\}$ as soon as we can guarantee that they are
not contained in a subset whose product is a square. This is in one sense very simple: a number
$a_i$ can be discarded if there exists a prime for which $a_i$ is the only remaining number that
it divides an odd number of times. However, this apparent simplicity is deceiving, and the
technical challenges involved in tracking the process are substantial. For example, we will
need to reveal the random numbers $\{a_1,\dots,a_N\}$ gradually (roughly speaking, prime by prime,
in decreasing order), and the amount of information we are allowed to reveal at each step is
rather delicate. Moreover, the removal of a number can trigger an avalanche, causing many other
numbers to be removed in the same step. Fortunately, however, self-correction (which is partly
a result of these avalanches) will allow us to show that the process remains subcritical (in a
certain natural sense), which will in turn allow us to control the upper tail of the size of the
avalanches, see Section~\ref{sec:branching}. In order to do so, we will need good control over the
dependence between the prime factors of the numbers $\{a_1,\dots, a_N\}$ conditioned on the
information we have observed so far. This is achieved in Section~\ref{sec:compare}, where we
obtain strong bounds on the ratio between the (conditional) probability of certain `basic' events,
and the corresponding probabilities in a simpler independent model. These bounds require some
number-theoretic estimates (stated in Section~\ref{sec:NTfacts}), most of which follow easily
from the fundamental work of Hildebrand and Tenenbaum~\cite{HT} on smooth numbers.

Using the method described above, we shall be able to show that with high probability the number of
`active' numbers (i.e., elements of $\{a_1,\dots,a_N\}$ that we have not yet discarded) tracks a
deterministic function (see \Th{track:m}, below) until there are very few numbers
remaining (roughly $e^{-C\sqrt{\log y_0}}y_0$ for some large constant~$C$), at which point we
can apply the first moment calculation from~\cite{CGPT}. The heuristic reason for the
appearance of the formula in \Th{track:m} is that the number of $y$-smooth numbers
is concentrated (e.g., by Chernoff's inequality) for all reasonably large values of~$y$, since the $a_i$
are chosen independently, and is equal to the number of isolated vertices in a certain natural
(random) hypergraph (see Definition~\ref{def:cS}). We will control the average degree of
this hypergraph (see \Th{track:s}), and show (using \Th{compare}) that
its edges are chosen \emph{almost} independently, so in particular its degree distribution is
close to Poisson. Equating these two estimates for the number of isolated vertices
gives~\eqref{eq:track:m}. The Euler--Mascheroni constant $\gamma$ appears in our proof at this 
point, since when we reveal the $z$th smallest prime, the (typical) average degree of the hypergraph is close to $\Ein(m(z)/z)$,  where $m(z)$ is the number of active numbers at this point, and
$\Ein$ is the exponential integral
\[
 \Ein(w):=\int_0^w\frac{1-e^{-t}}{t}\,dt.
\]
Finally, in order to prove the upper bound in \Th{squares:sharp}, we observe
(\Lm{number:of:columns}) that the ratio of the number of active numbers and active
primes (that is, primes which could still appear in some square) approaches 1 when $z=\pi(y_0)$
and $N / J(x)$ approaches $e^{-\gamma}$. Thus, by adding just a few extra $y_0$-smooth numbers,
we can apply the linear algebra approach of Schroeppel to obtain a subset whose product is a
square, as required.

The rest of the paper is organised as follows. In \Sc{outline} we give a detailed outline of
the proof, state our main auxiliary results, and define precisely the random process mentioned
above. In \Sc{NTfacts} we deduce the number-theoretic estimates we need from known results
on smooth numbers, and in \Sc{events} we recall some basic results from probability theory,
define some useful events, and use the results of \Sc{NTfacts} to prove various useful properties of
these events. In \Sc{compare} we shall again apply the results of \Sc{NTfacts}, this time to
control the dependence between the prime divisors of our random numbers in the probability space
obtained by conditioning on the information revealed in the random process so far, and then in
\Sc{branching} we apply the main theorem of \Sc{compare} to control the size of avalanches
in the process. In \Sc{big:z} we use these results and the method of self-correcting
martingales to control the process for large primes, and in \Sc{critical} we do the same in
the critical range $z=e^{O(\sqrt{\log y_0})}y_0$. Finally, in Sections~\ref{sec:proof:tracking}
and~\ref{sec:squares:proof}, we will deduce the main auxiliary theorems stated in \Sc{outline},
as well as \Th{squares:sharp} and \Co{expectation}.

\subsection{Notation and basic definitions}

Let us conclude this introduction by collecting together for convenience some of the basic
definitions and notation that we shall use throughout the paper. We shall denote the primes by
$q_1,q_2,\dots$, so $q_z$ is the $z$th prime (we use $q$ here to avoid overusing the letter~$p$,
which will often denote a probability). We shall write $[n] =\{1,\dots,n\}$ for the first $n$
positive integers, and $[m,n]$ for the set $\{m,\dots,n\}$. We shall also use the notation $a\in b\pm c$ to mean that
\[
 b-c\le a\le b+c.
\]

In this paper, a \emph{hypergraph} $\cH$ will consist of a set $V(\cH)$ of \emph{vertices} and a
multi-set $E(\cH)$ of \emph{hyperedges} (which we will usually refer to simply as \emph{edges}). 
A hyperedge is just a subset of $V(\cH)$, a $k$-edge is a hyperedge of size $k$. Note that 
we allow multiple copies of the same edge; all edge counts are taken
with multiplicity. A hypergraph $\cH'=(V',E')$ is a sub-hypergraph of $\cH=(V,E)$ if
$V'\subseteq V(\cH)$ and $E'\subseteq E(\cH)$ (so that each $e\in E'$ is a subset of $V'$). 
The \emph{degree} of a vertex $v\in V(\cH)$ in $\cH$ is the number of hyperedges containing it,
counted with multiplicity. An \emph{isolated vertex} is a vertex of degree~0. An \emph{even}
hypergraph is one in which all vertices have even degree.

The 2-\emph{core} of a hypergraph $\cH$ is the hypergraph obtained by repeatedly removing any
vertex of degree at most~1, along with the corresponding edge when the degree is exactly one. 
Clearly, the 2-core is the union of all sub-hypergraphs of minimum vertex degree at least 2.

Finally, let us recall the standard Landau notation, which we shall use throughout the paper. Given
functions $f(x)$ and $g(x)$, we write $f(x)=O(g(x))$ if $|f(x)|\le C |g(x)|$ for some constant
$C$ and all sufficiently large~$x$; and $f(x)=\Theta(g(x))$ if $f(x)=O(g(x))$,
$g(x)=O(f(x))$, and $f(x)/g(x)$ is positive for all sufficiently large~$x$.
We write $f(x)=o(g(x))$ if $f(x)/g(x)\to 0$ as $x\to\infty$, and
$f(x)=\omega(g(x))$ if $g(x)=o(f(x))$. Unless stated otherwise, all limits are as
$x\to\infty$, where $\{1,\dots,x\}$ is the set from which the random numbers~$a_i$ are chosen.
We shall avoid the notations $\Omega(f(x))$, $\ll$, and $\gg$, as these may mean different things
to different mathematical communities.

\section{An outline of the proof}\label{sec:outline}

In this section we shall define precisely the random process that we shall use to prove
\Th{squares:sharp}, and state our key results about this process, Theorems~\ref{thm:track:m}
and~\ref{thm:track:s}. Throughout the proof we fix a constant $\eta>0$ and a sufficiently large
integer $x$.\footnote{In the definitions below, we shall suppress the dependence on $x$ and $\eta$.} 
We set $N=\eta J(x)$ and define an $N$-tuple $(a_1,\dots,a_N)$ by choosing $N$
elements of $[x]$ independently and uniformly at random (with replacement), and form an
$N\times\pi(x)$ 0-1 matrix $A$ by setting $A_{ij}=1$ if and only if the $j$th prime $q_j$ occurs
an odd number of times in the prime factorisation of~$a_i$. Thus, to find a subset
$I\subseteq [N]$ such that $\prod_{i\in I} a_i$ is a square, it is enough to find a set of rows
of $A$ such that all column sums within these rows are even.

Note that the rows of $A$ are chosen independently, but the columns are not. For example, the
condition $a_i\le x$ puts a limit on the number of times a 1 can occur in row $i$ of~$A$.
More precisely, let $\tpsi(x,y)$ be the number of integers in $[x]$ all of whose prime factors that
are strictly greater than $y$ occur to an even power. Thus
\begin{equation}\label{def:tpsi}
 \tpsi(x,y)=\sum_{t\in P(y)}\Psi\big(x/t^2,y\big)
\end{equation}
where $P(y)$ is the set of all $t\ge 1$ that have no prime factor less than or equal to~$y$.
Define
\begin{equation}\label{def:pzx}
 p_j(x):=\frac{\tpsi(x,q_j)-\tpsi(x,q_{j-1})}{\tpsi(x,q_j)},
\end{equation}
for each $j\in [\pi(x)]$, and observe that $p_j(x)$ is equal to the conditional probability that
$A_{ij}=1$ if $A_{ij'}=0$ for every $j'>j$. Indeed, more generally we have
\begin{equation}\label{eq:probs}
 \Prb\Big(A_{ij}=1\bmid (A_{ij'})_{j'=j+1}^{\pi(x)}\Big)
 =p_j\bigg(x\prod_{j'>j,\,A_{ij'}=1}\frac{1}{q_{j'}}\bigg).
\end{equation}
Typically, $p_j(x)$ will be only slowly varying with~$x$, and so the entries in a row of $A$
will depend only mildly on one another. Nevertheless, this dependency is a major technicality
that we shall need to overcome.

We can also think of the matrix $A$ as a hypergraph whose vertices are the primes and whose edges
correspond to the set of primes dividing $a_i$ an odd number of times. We shall often wish to
ignore small primes here, so a precise definition is as follows.

\begin{defn}\label{def:cH}
 For each $z\in [0,\pi(x)]$, define $\cH_A(z)$ to be the hypergraph with vertex set
 $V(\cH_A(z))=[z+1,\pi(x)]$ and hyperedge set $E(\cH_A(z))=\{e'_i:i\in [N]\}$, where
 \[
  e'_i:=\big\{j\in [z+1,\pi(x)]:A_{ij}=1\big\}.
 \]
 In particular, when $z=\pi(x)$ all of the edges of $\cH_A(z)$ are empty.
\end{defn}

Croot, Granville, Pemantle and Tetali~\cite{CGPT} proved the upper bound in \Th{squares:sharp}
by counting the number of  \emph{acyclic} (also called \emph{Berge-acyclic}) even sub-hypergraphs of this
hypergraph. An acyclic hypergraph is one in which there does not exist, for any $k\ge 2$, 
a cycle of $k$ distinct hyperedges $e_0,e_1,\dots,e_k=e_0$ and distinct vertices $v_1,\dots,v_k$ with
each $v_i\in e_{i-1}\cap e_i$, $i=1,\dots,k$.  They showed that if $\eta>e^{-\gamma}$ then, for a suitable~$z$,
$\cH_A(z)$ contains more than $z$ edge-disjoint acyclic even 
sub-hypergraphs\footnote{Note that an empty hyperedge is an acyclic even sub-hypergraph
corresponding to a $q_z$-smooth integer~$a_i$; however, in order to find sufficiently many
relations for all $\eta>e^{-\gamma}$, the authors of~\cite{CGPT} needed to consider even
sub-hypergraphs with an arbitrarily large (but bounded) number of hyperedges.}
with high probability. This guarantees more than $z$ disjoint sets of rows of~$A$, each of which
has a sum in the subspace (taken over $\F_2$) of vectors that are supported on the first $z$
columns. As any set of more than $z$ vectors in this $z$-dimensional subspace is linearly
dependent, this guarantees a linear relation between the rows of~$A$. As noted in the
introduction, the authors of \cite{CGPT} used the first moment method to prove their lower bound,
counting the expected number of even sub-hypergraphs of $\cH_A(0)$ or, equivalently, the
number of sets of rows of $A$ that sum to zero over~$\F_2$. However, as mentioned in the
introduction, this method does not yield a sharp lower bound as this expected number blows up
before the threshold for the existence of a single such set given by \Th{squares:sharp}.

\subsection{The 2-core of $\cH_A(z)$}

Instead of counting the even sub-hypergraphs of $\cH_A(0)$, we shall instead study the 2-core
$\cC_A(z)$ of the hypergraph $\cH_A(z)$ for each $z_-\le z\le\pi(x)$, where
$z_-=e^{-\Theta(\sqrt{\log y_0})}y_0$ is defined below, see~\eqref{def:zpm}. Since all even
sub-hypergraphs (after removing isolated vertices) are sub-hypergraphs of the 2-core, it is
enough to restrict attention to $\cC_A(z)$. As noted in the introduction, this is equivalent to
iteratively removing any $a_i$ for which there exists a prime $q>q_z$ that occurs to an odd power
in~$a_i$, but to an even power in all other remaining~$a_j$. We shall show
that if $\eta<e^{-\gamma}$ then the 2-core $\cC_A(z_-)$ of $\cH_A(z_-)$ is (with high probability)
small by tracking the size of $\cC_A(z)$ throughout the range $z\in [z_-,\pi(x)]$. In particular
we shall show that $\cC_A(z_-)$ has fewer than $z_-$ edges with high probability. As a
consequence, any linear relation between the rows of $A$ must involve fewer than $z_-$ rows.
This however is ruled out by a result in \cite{CGPT} which shows (via a first moment calculation)
that any linear relation between the rows of $A$ must involve many rows.
We remark that this approach was partly inspired by the work of Pittel and Sorkin~\cite{PS} in a
closely related setting, where again a direct first moment calculation fails to find the correct
threshold for the appearance of linear relations in the rows of a random matrix, but succeeds
once restricted to the 2-core.

The following theorem tracks the size of the 2-core of $\cH_A(z)$ for all $z\in [z_-,\pi(x)]$,
and is the key technical statement we will need in order to prove \Th{squares:sharp}.
Let $M(z)$ be the set of rows of $A$ corresponding to the set $E(\cC_A(z))=\{e'_i:i\in M(z)\}$
of hyperedges of the 2-core, and let $m(z)=|M(z)|=|E(\cC_A(z))|$.

\begin{thm}\label{thm:track:m}
 If\/ $\eta<e^{-\gamma}$ and\/ $\eps_0>0$, then with high probability,
 \begin{equation}\label{eq:track:m}
  \frac{m(z)}{z}\exp\Big(-\Ein\Big(\frac{m(z)}{z}\Big)\Big)
  \in (1\pm\eps_0)\eta J(x)\frac{\Psi(x,q_z)}{xz}
 \end{equation}
 for every\/ $z\in [z_-,\pi(x)]$, where\/ $z_-$ is defined in\/~\eqref{def:zpm} below.
\end{thm}

Recall that the exponential integral $\Ein(w)$ is an entire function,
and is related to the incomplete gamma function via the relation
\begin{equation}\label{def:Ein}
 \Ein(w):=\int_0^w\frac{1-e^{-t}}{t}\,dt=\Gamma(0,w)+\log w+\gamma.
\end{equation}
Since $\Gamma(0,w)=\int_w^\infty e^{-t}\frac{dt}{t}$ decreases to 0 as $w\to\infty$, we see that
$w e^{-\Ein(w)}$ is a strictly increasing function of $w$ that converges to $e^{-\gamma}$ as
$w\to\infty$. Thus we can define $\alpha(\eta)\in [0,\infty)$ uniquely by the equation
\begin{equation}\label{def:alpha}
 \alpha(\eta)e^{-\Ein(\alpha(\eta))}=\eta
\end{equation}
for any $\eta\in [0,e^{-\gamma})$. Note that $\tfrac{d}{dw} w e^{-\Ein(w)}=e^{-w-\Ein(w)}$
and hence
\[
 \alpha'(\eta)=e^{\alpha(\eta)+\Ein(\alpha(\eta))}
\]
is an increasing function of~$\eta$. Thus $\alpha(\eta)$ is a convex function that strictly
increases from 0 to $\infty$ as $\eta$ increases from 0 to $e^{-\gamma}$.

Let us assume from now on that $0<\eta<e^{-\gamma}$, so that $\alpha(\eta)\in (0,\infty)$ is
well-defined by~\eqref{def:alpha}. We shall fix sufficiently small positive constants $\eps_0$,
$\eps_1$ and $\delta$ satisfying the following inequalities:
\begin{equation}\label{def:small}
 0<\eps_0<e^{-\gamma}-\eta,\qquad 0<\eps_1<\frac{\eps_0}{16}e^{-C_0}
 \qquad\textup{and}\qquad 0<\delta<\eps_1e^{-3/\eps_1},
\end{equation}
where
\begin{equation}\label{def:C0}
 C_0:=\alpha\big((1+\eps_0)\eta\big).
\end{equation}
Note that the upper bound on $\eps_0$ implies that $(1+\eps_0)\eta<e^{-\gamma}$ and
hence $C_0<\infty$. For convenience we shall also assume that $1/\eps_1$ is an integer.

The constant $\eps_0$ appears in \Th{track:m}, and determines the accuracy with which we
track~$m(z)$, while the constant $\eps_1$ will appear (via \Df{eps}) in
\Th{track:s} below, and will determine the accuracy to which we track various other parameters
of the process. The constant $\delta$ plays a different role: it determines the `critical'
range $[z_-,z_+]$, above which we shall have to use a slightly different approach,
and below which we will lose control of the process. To be precise, set
\begin{equation}\label{def:zpm}
 z_-:=\min\big\{z:\Lambda(z)\ge\delta\big\}\qquad\textup{and}\qquad
 z_+:=\max\big\{z:\Lambda(z)\ge\delta\big\},
\end{equation}
where
\begin{equation}\label{def:Lambda}
 \Lambda(z)=\Lambda_x(z):=J(x)\frac{\Psi(x,q_z)}{xz}
\end{equation}
for each $z\in [\pi(x)]$. As in the introduction, let $\pi(y)x/\Psi(x,y)$ be minimized at $y=y_0$, and define
$z_0=\pi(y_0)$. Observe that we can take $y_0$ to be prime, since $\pi(y)$ and $\Psi(x,y)$ only
change at prime~$y$. It follows from \eqref{def:J} that $\Lambda(z)\le 1$, and that
$\Lambda(z_0)=1$. \Th{track:m} and \eqref{def:alpha} together imply that
\begin{equation}\label{eq:mab}
 m(z)\in\alpha\big((1\pm\eps_0)\eta\Lambda(z)\big)z,
\end{equation}
so in particular, $m(z)\le C_0 z$ for all $z\in [z_-,\pi(x)]$ (see \Ob{m:bounded} below), and
moreover $m(z)\approx\alpha(\eta)z$ when $z\approx z_0$. We will show later
(see~\eqref{eq:zpm:precise}) that $z_\pm=e^{\pm\Theta(\sqrt{\log z_0})} z_0$.

\subsection{The hypergraph $\cS_A(z)$}

As mentioned in the introduction, the equation~\eqref{eq:track:m} comes from counting the number
of isolated vertices in a certain hypergraph in two different ways. This hypergraph is not
$\cC_A(z)$, but (in a certain sense) its `dual', defined as follows.

\begin{defn}\label{def:cS}
 For each $k\ge2$ and $z\in [0,\pi(x)]$, let $S_k(z)$ denote the collection of vertices of
 degree $k$ in $\cC_A(z)$,
 \[
  S_k(z):=\Big\{j\in [z+1,\pi(x)]:\big|\big\{i\in M(z):A_{ij}=1\big\}\big|=k\Big\}.
 \]
 In other words, $S_k(z)$ is the set of all columns of $A$ after column~$z$ that have column
 sum $k$ when restricted to the set of rows $M(z)$. Note for $z=\pi(x)$ we have 
 $S_k(z)=\emptyset$. Also, these column sums are zero for 
 $j\notin V(\cC_A(z))$, so $S_k(z)\subseteq V(\cC_A(z))$. Set $s_k(z):=|S_k(z)|$ and
 define $S(z):=\bigcup_{k \ge 2} S_k(z)$. 
\end{defn}

We shall think of $S_k(z)$ as labeling the $k$-edges $e_j:=\{i\in M(z):A_{ij}=1\}$
of a hypergraph $\cS_A(z)$ with vertex set $M(z)$ and edge set $\{e_j:j\in S(z)\}$.
Note that we are now thinking of the rows (corresponding to numbers~$a_i$) of
$A$ as being the vertices and the columns (primes) as hyperedges, where each prime $q$ corresponds
to the set of $i\in M(z)$ such that $q$ divides $a_i$ an odd number of times.

Later we shall show that $p_j(x)=(1+o(1))/j$ when $j$ is in the critical range
$z_-\le j\le z_+$. Thus heuristically one would expect that
\[
 \Prb\big(j\in S_k(z)\big)\approx\Prb\big(\Bin(m(z),1/j)=k\big)
\]
for $k\ge2$, where $\Bin(n,p)$ denotes a binomial random variable with $n$ trials and
success probability~$p$. Indeed, if a column has at least two 1s in active rows (i.e., rows in $M(z)$) then
this column has no effect on the construction of the 2-core. Also, one would expect that
the events $\big\{j\in S_k(z):j\in[z+1,\pi(x)]\big\}$ are `approximately independent'.
This leads one (after a short calculation) to predict that $s_k(z)$ tracks the following function.

\begin{defn}\label{def:tsk}
 For each $k\ge 2$, and every $z\in [\pi(x)]$, set
 \begin{equation}\label{eq:tsk}
  \ts_k(z):=\frac{m(z)}{k(k-1)}e^{-m(z)/z}
  \sum_{\ell=k-1}^\infty\frac{1}{\ell!}\Big(\frac{m(z)}{z}\Big)^\ell.
\end{equation}
\end{defn}

Note that $\ts_k(z)$ is a decreasing function of $k$ and that
\begin{equation}\label{eq:ts2}
 \ts_2(z)=\frac{m(z)}{2}\big(1-e^{-m(z)/z}\big).
\end{equation}
We next define a function that we shall use to bound the error in $s_k(z)$.

\begin{defn}\label{def:eps}
 For each $z\in [z_-,z_+]$ and each $k\ge 2$, define
 \[
  \eps(k,z):=\frac{\eps_1^k\cdot k!}{\Lambda(z)}.
 \]
\end{defn}

The function $\eps(k,z)$ decreases exponentially fast in $k$ while $k$ is relatively small, and then increases super-exponentially fast when $k$ is large. We will need the former
property in order to obtain the self-correction (see below) that will play a crucial role in our
proof, and the latter property in order to show that the bound~\eqref{eq:track:s} holds when $k$
is reasonably large.\footnote{More precisely, it will be important that $\eps(k,z)\ts_k(z)$ decreases only exponentially fast in $k$, see Observation~\ref{obs:eps:tsk:rough} and its applications in Section~\ref{sec:critical}.} We shall prove the following theorem.

\begin{thm}\label{thm:track:s}
 Suppose\/ $\eta<e^{-\gamma}$. Then, with high probability,
 \begin{equation}\label{eq:track:s}
  s_k(z)\in\big(1\pm\eps(k,z)\big)\ts_k(z)
 \end{equation}
 holds for every\/ $k\ge 2$ and every\/ $z\in [z_-,z_+]$.
\end{thm}

Note that $\eps(k,z)>1$ for all sufficiently large $k$, so to prove that~\eqref{eq:track:s}
holds for these values of $k$ it will suffice to show that $s_k(z)=0$. We shall in fact show
that, with high probability, $s_k(z)=0$ for all $k\ge 5\log z_0/\log\log z_0$ and for all
$z\in [z_-,\pi(x)]$ (see \Lm{skeasy} below).

As mentioned in the introduction, we shall prove \Th{track:s} using the method of self-correcting
martingales (see, e.g.,~\cite[Section~3]{FGM}). Roughly speaking, we shall show (see
Lemmas~\ref{lem:deltaSstar:bigz} and~\ref{lem:deltaSstar}) that if nothing has yet gone wrong,
then the (expected) drift in the \emph{relative} error of $s_k(z)$ depends mainly on the current
error, and (unless it is already quite small) tends to push the error towards zero. We emphasize
that the function $\eps(k,z)$ was chosen with exactly these lemmas in mind; in particular, it will
be important that $\eps(k,z)$ decreases rapidly for small values of~$k$, since we shall use this
fact to bound the influence of the error in $s_{k+1}(z)$ on the drift of $s_k(z)$. Combining these
lemmas with bounds on the probability of a large jump in the relative error (see
Lemmas~\ref{lem:maxstep:bigz} and~\ref{lem:maxstep}), it will be relatively easy to deduce
sufficiently strong bounds on the probability that $s_k(z)$ is the first variable to `go~astray'.

\Th{track:m} will be proved simultaneously with \Th{track:s}, but we will not show that $m(z)$ is self-correcting;
instead we shall show that $m(z)$ is unlikely to go off track before any of the $s_k(z)$. More
precisely, we shall use a martingale approach to show that $m(z)$ does not drift off course too
quickly, together with an occasional application of \Lm{empty} to put it back on track.
Since the probability of failure in \Lm{empty} is relatively large, we
can only apply it a small number of times; however, this will be sufficient to prove \Th{track:m}
over the `critical' interval $[z_-,z_+]$, while larger values of $z$ are easier to deal with.

\subsection{The random process}

Let us finish this section by defining the random process we shall use to reveal the 2-core
$\cC_A(z)$. In each step we reveal just enough information to proceed; in particular, and
crucially, we shall not reveal the exact locations of the $1$s in a column until it has only a
single non-zero element in an \emph{active} row, that is, a row of~$M(z)$.

\begin{alg}\label{alg:algorithm}
 We start with $z:=\pi(x)$, $M(z):=[N]$ and $S_k(z):=\emptyset$ for each $k\ge2$.
 Now repeat the following steps until $z=0$:
 \begin{itemize}
  \item[$1.$] Set $M:=M(z)$, $S_k:=S_k(z)$ for each $k\ge 2$ and $S_1:=\emptyset$.
  \item[$2.$] Reveal the (random) quantity $d(z):=\big|\big\{i\in M(z):A_{iz}=1\big\}\big|$,
   that is, the number of active non-zero entries of column~$z$.
  \item[$3.$] If $d(z)=d>0$, set $S_d:=S_d\cup\{z\}$.
  \item[$4.$] While $S_1\ne\emptyset$ do:\smallskip
  \begin{itemize}
   \item[$(a)$] Pick the smallest\footnote{An arbitrary $z'\in S_1$ would do here, but we shall
    later wish to ensure that the order in which rows are removed by the algorithm is uniquely
    specified.} $z'\in S_1$, observe which row $i$ is such that $i\in M$, $A_{iz'}=1$.
   \item[$(b)$] Set $M:=M\setminus\{i\}$, $S_1:=S_1\setminus\{z'\}$.
   \item[$(c)$] For each $k\ge1$ and each $j\in S_k$, reveal whether column $j$ has a 1 in
    row~$i$; if it does, remove $j$ from $S_k$ and (if $k>1$) add it to $S_{k-1}$.
  \end{itemize}
  \item[$5.$] Set $M(z-1):=M$ and $S_k(z-1):=S_k$ for each $k\ge 2$.
  \item[$6.$] Set $z:=z-1$; if $z>0$ then return to Step~1, otherwise stop.
\end{itemize}
\end{alg}

It is easy to see that this algorithm tracks the 2-core $\cC_A(z)$ as $z$ decreases from $\pi(x)$
to~0. Define a filtration $\cF_{\pi(x)}\subseteq\cF_{\pi(x)-1}\subseteq\dots$ by taking $\cF_y$ to
be the information observed at the moment the index $z$ is set equal to~$y$. More precisely,
$\cF_{y}$ reveals which rows and columns of $A$ correspond to the edges and vertices of the
2-core $\cC_A(y)$, as well as the degrees (column sums) of all the vertices in the 2-core. The 
only other information revealed by $\cF_y$ concerns rows of $A$ outside of~$M(y)$ (as a result 
of earlier steps of the algorithm), which will be irrelevant for our purposes.

Define the $\sigma$-algebra $\cF^+_y$ to include $\cF_y$ and also the information observed at
Step~2 when $z=y$, namely the value of $d(y)$. Thus $\cF^+_y$ specifies the column sums of $A$ of
all columns in $[y,\pi(x)]$, summing only over rows in $M(y)$. The matrix $A$ conditioned on
$\cF^+_y$ and restricted to $M(y)\times [y,\pi(x)]$ can be constructed with the correct
probability distribution by taking a uniform distribution on all choices of the multi-set of
numbers $\{a_i:i\in M(y)\}$ whose column sums are compatible with $\cF^+_y$. Indeed, any such
multiset, combined with the original $a_i$ for all $i \notin M(y)$, would result in the algorithm
constructing the same 2-core, and all such choices of the $a_i$ are equally likely.
Understanding this probability space will be the main aim of the next three sections,
and will be key to the proof of \Th{squares:sharp}.

\section{Number-theoretic facts}\label{sec:NTfacts}

In order to understand the distribution of the numbers $\{a_i:i\in M(z)\}$ conditioned on the
information observed in $\cF_z$ or $\cF_z^+$, we shall need some detailed information about the
smooth number counting function $\Psi(x,y)$ and its close relative $\tpsi(x,y)$. The first result
in this direction was obtained by Dickman~\cite{Dickman} in 1930, who proved that if $u$ is fixed
then
\[
 \lim_{x\to\infty}\frac{\Psi(x,x^{1/u})}{x}=\rho(u),
\]
where $\rho$ is the (unique) continuous solution to the delay differential equation
\begin{equation}\label{eq:rho:ddiff}
 u\rho'(u)+\rho(u-1)=0
\end{equation}
for $u>1$, with the boundary condition $\rho(u)=1$ for all $0\le u\le 1$.
This function is now known as the Dickman--de Bruijn function, and is asymptotically of the form
\begin{equation}\label{eq:rho:rough}
 \rho(u)=u^{-(1+o(1))u}
\end{equation}
as $u\to\infty$. Further important breakthroughs
were made in 1938, by Rankin~\cite{Rankin}, and in 1951, by de Bruijn~\cite{Bru51}, who
determined $\Psi(x,y)$ asymptotically when $y\ge\exp\big((\log x)^{5/8+\eps}\big)$ for some
$\eps>0$. Upper and lower bounds in a much wider range were later proved by de Bruijn~\cite{Bru66}
and by Canfield, Erd\H{o}s and Pomerance~\cite{CEP}, respectively. We will use the following
asymptotic result, due to Hildebrand~\cite{H86}.

\begin{thm}[Hildebrand, 1986]\label{thm:Psi}
 Let\/ $\exp\big((\log\log x)^2\big)\le y\le x$, and set\/ $u=\frac{\log x}{\log y}$. Then
 \[
  \frac{\Psi(x,y)}{x}=\rho(u)\bigg(1+O\bigg(\frac{\log(u+1)}{\log y}\bigg)\bigg)
 \]
 uniformly in\/ $x$ and\/ $y$.
\end{thm}

We remark that the main result of~\cite{H86} is even more general than \Th{Psi}, but the
version above is a little simpler to state, and more than sufficient for our purposes. Indeed,
it follows from \Th{Psi} and \eqref{eq:rho:rough}
(see, for example,~\cite[Section~2.1]{CGPT}) that\footnote{In fact, to prove~\eqref{eq:J:rough}
one only needs the results of de Bruijn~\cite{Bru66} and Canfield,
Erd\H{o}s and Pomerance~\cite{CEP}, which imply that $\Psi(x,y)=xu^{-(1+o(1))u}$ as $u\to\infty$
for all $y \ge (\log x)^{1+\eps}$.}
\begin{equation}\label{eq:J:rough}
 J(x)=y_0^{2+o(1)}=\exp\Big(\big(1+o(1)\big)\sqrt{2\log x\log\log x}\Big)
\end{equation}
as $x\to\infty$, as claimed in the introduction. It also follows that
\begin{equation}\label{eq:Psi:rough}
 \Psi(x,y_0^\beta)=xy_0^{-1/\beta+o(1)}
\end{equation}
for every $\beta=\beta(x)$ bounded away from~0, which
implies the following crude estimate for $\Lambda(z)$.

\begin{cor}\label{cor:Lambda:rough}
 Let\/ $\beta=\beta(x)$ be bounded away from~$0$,
 and set\/ $z=z_0^\beta$, where\/ $z_0=\pi(y_0)$. Then
 \[
  \Lambda(z)=z_0^{2-\beta-1/\beta+o(1)}
 \]
 as\/ $x\to\infty$.
\end{cor}

Note that \Co{Lambda:rough} also follows from~\cite[Lemma~2.1]{CGPT}, and implies that
\begin{equation}\label{eq:zpm:rough}
 z_\pm=z_0^{1+o(1)}.
\end{equation}
Let us define $u_0 := \frac{\log x}{\log y_0}$, so that $y_0=x^{1/u_0}$, and note here for future
reference the following immediate and useful consequence of~\eqref{eq:J:rough}:
\begin{equation}\label{eq:ulogu}
 \log z_0=(1+o(1))u_0\log u_0.
\end{equation}
In order to obtain more detailed information about the function $\Lambda(z)$, we shall need some
fundamental results of Hildebrand and Tenenbaum~\cite[Theorem 3]{HT}, which control the
`local' dependence of $\Psi(x,y)$ on the variable~$x$. (We remark that the idea of using these to
understand the matrix $A$ was one of the key innovations of~\cite{CGPT}.)
Instead of quoting these results directly, we shall prove a form (\Th{rho:ratio}) that
will be more convenient for our purposes. We will need the following two results
on the Dickman--de Bruijn function $\rho(u)$.

\begin{thm}[Hildebrand~{\cite[proof of Lemma~1]{H86}}]\label{thm:rho:log:concave}
 The function\/ $\rho(u)$ is log-concave for $u\ge1$. Equivalently,
 $\frac{\rho(u-1)}{u\rho(u)}$ is increasing in~$u$.
\end{thm}

Define the function $\xi=\xi(u)$ for $u>1$ to be the unique positive solution of the equation
\begin{equation}\label{def:xi}
 e^{\xi(u)}=1+u\xi(u).
\end{equation}

\begin{thm}[Hildebrand and Tenenbaum~{\cite[equation (2.1)$'$]{HTS}}]\label{thm:rho}
 For\/ $u>1$ we have
 \[
  \rho(u)=\bigg(1+\frac{O(1)}{u}\bigg)\sqrt{\frac{\xi'(u)}{2\pi}}
  \exp\bigg(\gamma-\int_1^u \xi(t)\,dt\bigg).
 \]
\end{thm}

We now state our key estimate on the rate of change of the function $\rho(u)$. This essentially
follows from~\cite[Corollary~2.4]{HTS}, but since the precise version we need is not an immediate
consequence of the results stated in~\cite{HTS}, we shall give the proof for completeness.

\begin{thm}\label{thm:rho:ratio}
 Let\/ $u>1$ and\/ $a,b\ge 0$ satisfy\/ $a+b\le u$. Then
 \begin{equation}\label{eq:rho:ineq}
  \frac{\rho(u-a-b)}{\rho(u-a)}\le\exp\bigg(b\xi(u)+\frac{O(1)}{u}\bigg).
 \end{equation}
 If in addition\/ $a+b\le u/2$ then
 \begin{equation}\label{eq:rho:eq}
  \frac{\rho(u-a-b)}{\rho(u-a)}=\exp\bigg(b\xi(u)+\frac{O(b^2 + ab + 1)}{u}\bigg).
 \end{equation}
\end{thm}

We shall use the following simple facts about the function $\xi(u)$, which we collect
here for convenience:
\begin{equation}\label{eq:xi:facts}
 \xi(u) = \big(1+o(1) \big) \log u,\qquad \xi'(u)=\frac{\Theta(1)}{u}\qquad\textup{and}\qquad
 \xi''(u)=-\frac{\Theta(1)}{u^2}
\end{equation}
for all $u>1$. We remark that here the $o(1)$ is as $u \to \infty$. 

\begin{proof}
Suppose first that $a=0$ and $u-b>1$, and apply \Th{rho} to $u-b$ and $u$ to obtain
\[
 \log\frac{\rho(u-b)}{\rho(u)}
 =\frac{1}{2}\log\frac{\xi'(u-b)}{\xi'(u)}+\int_{u-b}^{u}\xi(t)\,dt+\frac{O(1)}{u-b}.
\]
We shall bound each of the terms on the right in turn. First, observe that
\[
 \log\frac{\xi'(u-b)}{\xi'(u)}
 =-\int_{u-b}^u\frac{\xi''(t)}{\xi'(t)}\,dt
 =O(1)\int_{u-b}^u\frac{dt}{t}=O\Big(\log\frac{u}{u-b}\Big),
\]
where the first step follows by differentiating $\log\xi'(t)$, and the second follows
from~\eqref{eq:xi:facts}. Next, note that integration by parts gives
\begin{align*}
 b\xi(u)-\int_{u-b}^u\xi(t)\,dt
 &=\int_{u-b}^u\xi'(t)(t-u+b)\,dt\\
 &=\Theta(1)\int_{u-b}^u\frac{t-u+b}{t}\,dt=\frac{\Theta(b^2)}{u},
\end{align*}
where the second step follows from~\eqref{eq:xi:facts}, and the last equality holds
for $u-b\ge\eps u$ as $\int_{u-b}^u(t-u-b)\,dt=b^2/2$ and $t=\Theta(u)$
for $t\in[\eps u,u]$. On the other hand, the last integral increases to $b$ as $u-b\to 0^+$,
so it also holds for $1<u-b\le\eps u$.

Combining the three equations above, we obtain
\[
 \log\frac{\rho(u-b)}{\rho(u)}=b\xi(u)+O\Big(\log\frac{u}{u-b}\Big)
 -\frac{\Theta(b^2)}{u}+\frac{O(1)}{u-b}.
\]
Now for $0\le b\le u/2$ we have $\log\frac{u}{u-b}=O(b/u)$ and $1/(u-b)=O(1/u)$, so
\begin{equation}\label{eq:dlogrho}
 \frac{\rho(u-b)}{\rho(u)}=\exp\bigg(b\xi(u)+\frac{O(b+1)-\Theta(b^2)}{u}\bigg).
\end{equation}
This clearly also holds for $0\le u-b\le 1$ as then $u\le 2$ is bounded.
Thus \eqref{eq:rho:eq} holds for $a=0$.

Now suppose $1<u-b\le u/2$. Then $b^2/u=\Theta(u)$, while $\log\frac{u}{u-b}\le\log u$
and $1/(u-b)=O(1)$. Thus we obtain
\begin{equation}\label{eq:rho:ub}
 \frac{\rho(u-b)}{\rho(u)}\le\exp\bigg(b\xi(u)+\frac{O(1)}{u}\bigg).
\end{equation}
However, \eqref{eq:rho:ub} also follows from \eqref{eq:dlogrho} when $u-b\ge u/2$ as
the $O(b+1)-\Theta(b^2)$ term is bounded above by a constant for all $b\ge0$.
Finally, note that $\xi(u)\ge 0$ and $\rho(u-b)=1$ for every $0\le u-b\le 1$, so it
follows that~\eqref{eq:rho:ub} in fact holds for all $0\le b\le u$.
In particular, \eqref{eq:rho:ineq} holds for $a=0$.

In order to deduce~\eqref{eq:rho:ineq} and~\eqref{eq:rho:eq} in the case when $a>0$,
we substitute $u-a$ for $u$ in~\eqref{eq:dlogrho} and~\eqref{eq:rho:ub}. In the former
case, this gives
\[
 \frac{\rho(u-a-b)}{\rho(u-a)}=\exp\bigg(b\xi(u-a)+\frac{O(b+1)-\Theta(b^2)}{u-a}\bigg),
\]
for all $a+b\le u/2$, which implies~\eqref{eq:rho:eq} since $\xi(u)=\xi(u-a)+O(a/u)$,
by~\eqref{eq:xi:facts}, and $u-a=\Theta(u)$. Similarly, substituting $u-a$ for $u$ in
\eqref{eq:rho:ub} gives
\[
 \frac{\rho(u-a-b)}{\rho(u-a)}\le\exp\bigg(b\xi(u-a)+\frac{O(1)}{u-a}\bigg)
 \le\exp\bigg(b\xi(u)+\frac{O(1)}{u-a}\bigg)
\]
since $\xi$ is an increasing function. This is enough to prove~\eqref{eq:rho:ineq}
when $u-a\ge\eps u$, so let us assume instead that $u-a\le\eps u$. Now, observe that
\[
 -\frac{\rho'(u)}{\rho(u)}=\frac{\rho(u-1)}{u\rho(u)}\le\frac{e^{\xi(u)+O(1/u)}}{u}
 =\xi(u)+o(1)
\]
as $u\to\infty$, where the first step follows by the definition~\eqref{eq:rho:ddiff}
of~$\rho$, the second step follows by~\eqref{eq:rho:ub} applied with $b=1$, and the
third step follows by the definition~\eqref{def:xi} of $\xi$ and \eqref{eq:xi:facts}.
It follows that $-\rho'(u)/\rho(u)\le\xi(Cu)$ for some absolute constant $C>1$
(since $\xi'(u)=\Theta(1/u)$), and hence
\[
 \log\frac{\rho(u-a-b)}{\rho(u-a)}
 =-\int_{u-a-b}^{u-a}\frac{\rho'(t)}{\rho(t)}\,dt
 \le b\xi\big(C(u-a)\big)\le b\xi(u)
\]
if $\eps$ was chosen sufficiently small, since $\xi$ is increasing.
Thus~\eqref{eq:rho:ineq} also holds in this case, and the proof is complete.
\end{proof}

The probabilities in our model are given in terms of the modified smooth number counting function
$\tpsi(x,y)$. We now show that there is little difference between $\Psi(x,y)$ and $\tpsi(x,y)$.

\begin{lemma}\label{lem:tpsi}
 If\/ $y\ge\exp\big((\log\log x)^2\big)$, then
 \[
  \Psi(x,y)\le\tpsi(x,y)\le\big(1+y^{-1+o(1)}\big)\Psi(x,y).
 \]
 uniformly in\/ $y$ as\/ $x\to\infty$.
\end{lemma}
\begin{proof}
The lower bound holds trivially, by definition; we shall prove the upper bound.
We may assume $y\le x$ as otherwise $\Psi(x,y)=\tpsi(x,y)=x$. Recall
from~\eqref{def:tpsi} that
\[
 \tpsi(x,y)=\sum_{t\in P(y)}\Psi(x/t^2,y),
\]
where $P(y)$ is the set of positive integers whose prime factors are all strictly larger
than~$y$. Applying Theorems \ref{thm:Psi} and \ref{thm:rho:ratio} with
$a=0$, $b=2\frac{\log t}{\log y}$, we obtain
\[
 \frac{t^2\Psi(x/t^2,y)}{\Psi(x,y)}
 \le\bigg(1+\frac{O(\log(u+1))}{\log y}\bigg)\frac{\rho(u-b)}{\rho(u)}
 \le\exp\bigg(b\xi(u)+\frac{O(1)}{u}+\frac{O(\log(u+1))}{\log y}\bigg),
\]
for all $t\in[y,\sqrt{x}]$,
where $u=\frac{\log x}{\log y}\ge u-b=\frac{\log(x/t^2)}{\log y}\ge0$.
Note that in the case when $x/t^2<y$ we apply \Th{Psi} only to $\Psi(x,y)$,
as in this case $t^2\Psi(x/t^2,y)\le x$ and $\rho(u-b)=1$.
Now $\xi(u)=O(\log u)=o(\log y)$ for $y\ge\exp\big((\log\log x)^2\big)$. Hence
$t^2\Psi(x/t^2,y)/\Psi(x,y)=t^{o(1)}$, and so
\[
 \tpsi(x,y)-\Psi(x,y)=\sum_{1<t\in P(y)}\Psi(x/t^2,y)
 \le\Psi(x,y)\sum_{t\ge y}t^{-2+o(1)}\le y^{-1+o(1)}\Psi(x,y),
\]
as claimed.
\end{proof}

We can now state our estimates for the rate of change of the functions $\Psi(x,y)$
and $\tpsi(x,y)$. The form of the statement below is designed to facilitate its application
in \Sc{compare}.

\begin{thm}\label{thm:Psi:ratio}
 Let\/ $z \in [z_-,\pi(x)]$, set\/ $y=q_z$ and\/ $u=\frac{\log x}{\log y}$, and assume\/
 $0\le a\le u$ and\/ $b\ge1$. Then
 \begin{equation}\label{eq:Psi:ratio:ineq}
  \log\frac{\Psi(xy^{-(a+b)},y)}{\Psi(xy^{-a},y)} \le b\big(\xi(u)-\log y\big)+\frac{O(1)}{u}.
 \end{equation}
 Moreover, if\/ $a+b\le u/2$, then
  \begin{equation}\label{eq:Psi:ratio:eq}
  \log\frac{\Psi(xy^{-(a+b)},y)}{\Psi(xy^{-a},y)}
  =b\big(\xi(u)-\log y\big)+\frac{O(b^2+ab+1)}{u}.
 \end{equation}
 Also, both statements hold with\/ $\tpsi$ in place of\/ $\Psi$.
\end{thm}

Note that we interpret $\log0=-\infty$ in the case when $a+b>u$ (and so $\Psi(xy^{-(a+b)},y)=0$).

\begin{proof}
The bound $y\ge\exp((\log\log x)^2)$ follows from~\eqref{eq:J:rough}, since $z\ge z_-$,
so $y\ge y_0^{1+o(1)}$, by~\eqref{eq:zpm:rough}. Thus by \Th{Psi}, we have
\[
 \frac{\Psi(xy^{-(a+b)},y)}{\Psi(xy^{-a},y)}
 =y^{-b}\cdot\frac{\rho(u-a-b)}{\rho(u-a)}\bigg(1+O\bigg(\frac{\log(u+1)}{\log y}\bigg)\bigg)
\]
if $a+b\le u-1$ and $b\ge 0$, since these inequalities imply that $y\le xy^{-(a+b)}\le xy^{-a}$.
Both~\eqref{eq:Psi:ratio:ineq} and~\eqref{eq:Psi:ratio:eq} now follow by \Th{rho:ratio},
since $\log(u+1)/\log y=O(1/u)$, by~\eqref{eq:ulogu}, and the corresponding bounds with $\Psi$
replaced by $\tpsi$ follow using \Lm{tpsi}.

The case when $a+b>u-1$ and $a+b\le u/2$ is ruled out by the assumption\footnote{Note that we need some
condition on the parameters in the statement of the theorem in order to rule out the case when, say,
$y^b\approx 2$, but $1\le xy^{-(a+b)}<xy^{-a}<2$, and $u\to\infty$.} that $b\ge1$, so
it only remains to prove that~\eqref{eq:Psi:ratio:ineq} holds when $u-1\le a+b\le u+b$ and $b\ge 1$.
If $a\le u-1$ then we can apply~\eqref{eq:Psi:ratio:ineq} with $b$ replaced by $b':=u-1-a$
(noting that we in fact proved it for all $a+b\le u-1$ and $b\ge 0$) to obtain
\[
 \log\frac{\Psi(xy^{-(a+b')},y)}{\Psi(xy^{-a},y)}=\log\frac{y}{\Psi(xy^{-a},y)}
 \le b'\big(\xi(u)-\log y\big)+\frac{O(1)}{u}.
\]
Noting that $\Psi(xy^{-(a+b)},y)\le xy^{-(a+b)}=y^{1+b'-b}$, it follows that
\[
 \log\frac{\Psi(xy^{-(a+b)},y)}{\Psi(xy^{-a},y)}
 \le\log\frac{y^{1+b'-b}}{\Psi(xy^{-a},y)}\le b'\xi(u)-b\log y+\frac{O(1)}{u},
\]
which implies~\eqref{eq:Psi:ratio:ineq} since $b' \le b$. The proof with $\Psi$ replaced by $\tpsi$
is identical. Finally, if $a>u-1$ and $b\ge 1$ then $xy^{-(a+b)}<1$, and so the claimed bound
holds trivially.
\end{proof}

We are now ready to bound the conditional probability of $A_{ij}=1$ that is given
by~\eqref{eq:probs}.

\begin{cor}\label{cor:pj}
 If\/ $z\in [z_-,z_+]$ and\/ $a=o(u_0)$, then
 \[
  p_z\big(xq_z^{-a}\big)=\frac{1+o(1)}{z}.
 \]
 Moreover, for any\/ $z\in [z_-,\pi(x)]$ and\/ $a\ge 1$, we have
 \[
  p_z\big(xq_z^{-a}\big) \le \big(1+o(1)\big)p_z(x) \le \frac{1+o(1)}{z}.
 \]
Furthermore, $p_z(x)=z^{-1+o(1)}$ for every\/ $z\in[z_-,\pi(x)]$.
\end{cor}

We shall use the following observation in the proof of Corollary~\ref{cor:pj}.

\begin{obs}\label{obs:exi}
 If\/ $z=z_0^{1+o(1)}$ and\/ $x=q_z^u$, then\/ $e^{\xi(u)}=(1+o(1))q_z/z$.
\end{obs}

\begin{proof}
Note that $z=z_0^{1+o(1)}$ implies that $u=(1+o(1))u_0$, and hence
\begin{align*}
 e^{\xi(u)}&=1+u\xi(u)=(1+o(1))u\log u=(1+o(1))u_0\log u_0\\
 &=(1+o(1))\log z_0=(1+o(1))\log q_z=(1+o(1))q_z/z,
\end{align*}
by~\eqref{eq:ulogu},~\eqref{def:xi} and the prime number theorem, as claimed.
\end{proof}

\begin{proof}[Proof of Corollary~\ref{cor:pj}]
Set $x'=xq_z^{-a}$, and observe that
\[
 \tpsi(x',q_z)-\tpsi(x',q_{z-1})=\tpsi(x' q_z^{-1},q_{z-1})=\tpsi(x' q_z^{-1},q_z)-\tpsi(x' q_z^{-2},q_{z-1})
\]
and $0\le\tpsi(x' q_z^{-b},q_{z-1})\le\tpsi(x' q_z^{-b},q_z)$ for $b\in\{1,2\}$. Therefore, recalling
the definition~\eqref{def:pzx} of $p_z(x)$,
\begin{equation}\label{eq:pj:bounds}
 \frac{\tpsi(x' q_z^{-1},q_z)-\tpsi(x' q_z^{-2},q_z)}{\tpsi(x',q_z)}
 \le p_z(x')\le\frac{\tpsi(x' q_z^{-1},q_z)}{\tpsi(x',q_z)}.
\end{equation}
Let $x=q_z^u$, and note that if $z\in[z_-,z_+]$ then $u=(1+o(1))u_0$, since $z_\pm=z_0^{1+o(1)}$,
by~\eqref{eq:zpm:rough}. Thus, if $a=o(u_0)$ then, by \Th{Psi:ratio}, we have
\[
 \log\frac{\tpsi(x'q_z^{-b},q_z)}{\tpsi(x',q_z)}=b\big(\xi(u)-\log q_z\big)+o(1)
\]
for $b\in\{1,2\}$. Moreover, $e^{\xi(u)}=(1+o(1))q_z/z$ by \Ob{exi}, since $z_\pm=z_0^{1+o(1)}$.
It follows that
\[
 \frac{\tpsi(x'q_z^{-b},q_z)}{\tpsi(x',q_z)}=\bigg(\frac{1+o(1)}{z}\bigg)^b,
\]
which together with~\eqref{eq:pj:bounds} implies that $p_z(x')=(1+o(1))/z$
when $a=o(u_0)$, as claimed.

The bounds $p_z(x)\le(1+o(1))/z$ and $p_z(x)=z^{-1+o(1)}$ follow by a similar argument.
Indeed, by \Th{Psi:ratio} (applied with $a=0$), we have
\begin{equation}\label{eq:pj:azero:upper}
 \log\frac{\tpsi(x q_z^{-b},q_z)}{\tpsi(x,q_z)}\le b\big(\xi(u)-\log q_z\big)+\frac{O(1)}{u}
\end{equation}
for $b\in\{1,2\}$, and moreover a matching lower bound holds when $u\ge 2b$.
Note also that, since $1\le u\le(1+o(1))u_0$, we have
\[
 e^{\xi(u)+O(1/u)}\le\big(1+o(1)\big)e^{\xi(u_0)}=\big(1+o(1)\big)\frac{q_{z_0}}{z_0}
 \le\big(1+o(1)\big)\frac{q_z}{z}
\]
by~\eqref{eq:xi:facts}, \Ob{exi} and the prime number theorem. Thus by~\eqref{eq:pj:azero:upper},
\begin{equation}\label{eq:tpsi:ratio:bound}
 \frac{\tpsi(x q_z^{-b},q_z)}{\tpsi(x,q_z)}\le\bigg(\frac{1+o(1)}{z}\bigg)^b,
\end{equation}
and \eqref{eq:pj:bounds} then implies $p_z(x)\le (1+o(1))/z$. Moreover, the lower bound
$p_z(x)\ge e^{O(1)}/q_z$ holds for $u\ge 2$ by~\eqref{eq:pj:bounds}, using the matching lower bound
in~\eqref{eq:pj:azero:upper} when $b=1$. For $1\le u<2$,
$\tpsi(x/q_z,q_z)=\lfloor x/q_z\rfloor=\Theta(x/q_z)$ and $\tpsi(x,q_z) \le x$, so again
$p_z(x)\ge c/q_z$ for some $c>0$. In particular, $p_z(x)\ge z^{-1+o(1)}$.

Finally, the bound $p_z(xq_z^{-a})\le(1+o(1))p_z(x)$ is also similar, though since it does not
always hold when $a<1$, we shall need to be a little careful. Note first that if $u < a+1$
then the claimed bounds hold trivially, since $xq_z^{-a}<q_z$, and hence
$p_z(xq_z^{-a})=0$. We may therefore assume that $u \ge a+1$, which implies that $xq_z^{-1} \ge q_z$ and $x' \ge q_z$. We claim that
\begin{equation}\label{eq:cor:pj:finalclaim}
 p_z(x')\le\frac{\tpsi(x'q_z^{-1},q_z)}{\tpsi(x',q_z)}
 \le \big( 1+o(1) \big)\frac{\tpsi(xq_z^{-1},q_z)}{\tpsi(x,q_z)}\le \big( 1 + o(1) \big) p_z(x).
\end{equation}
Indeed, the first inequality follows by~\eqref{eq:pj:bounds}, and the second follows since \Th{rho:log:concave} implies that $\frac{\rho(u-1)}{u\rho(u)}$, and hence $\frac{\rho(u-1)}{\rho(u)}$, is an increasing function of~$u$, so if $u \ge a + 2$ then
\begin{align*}
 \frac{\tpsi(x'q_z^{-1},q_z)}{\tpsi(x',q_z)} & \le \big( 1 + o(1) \big) q_z^{-1} \cdot\frac{\rho(u-a-1)}{\rho(u-a)} \\
 & \le \big( 1 + o(1) \big) q_z^{-1} \cdot \frac{\rho(u-1)}{\rho(u)} = \big( 1 + o(1) \big) \frac{\tpsi(xq_z^{-1},q_z)}{\tpsi(x,q_z)}
\end{align*}
by Theorem~\ref{thm:Psi} and \Lm{tpsi}. On the other hand, if $a+1 \le u \le a+2$ then $x'q_z^{-1} < q_z$, so in this case we may replace the bound on $\tpsi(x'q_z^{-1},q_z)$ given by Theorem~\ref{thm:Psi} by the equality $\tpsi(x'q_z^{-1},q_z) = x'q_z^{-1} = \rho(u - a - 1) x'q_z^{-1}$, since by definition $\rho(u) = 1$ for all $0 \le u \le 1$.

To deduce~\eqref{eq:cor:pj:finalclaim} from~\eqref{eq:pj:bounds}, it therefore only remains to observe that, by~\eqref{eq:tpsi:ratio:bound} and two applications of \Th{Psi:ratio},
\[
 \frac{\tpsi(xq_z^{-2},q_z)}{\tpsi(x,q_z)} \le O(1) \cdot \bigg(\frac{\tpsi(xq_z^{-1},q_z)}{\tpsi(x,q_z)}\bigg)^2
 = o(1) \cdot \frac{\tpsi(xq_z^{-1},q_z)}{\tpsi(x,q_z)}.
\]
where we again used the fact that $u \ge a + 1 \ge 2$ when applying \Th{Psi:ratio} with $b = 1$. This proves~\eqref{eq:cor:pj:finalclaim}, and hence completes the proof of the corollary.
\end{proof}

We next deduce some more refined estimates concerning the function $\Lambda(z)$.

\begin{lemma}\label{lem:Lambda:onestep}
 If\/ $z=z_0^{1+o(1)}$, then
 \[
  \frac{\Lambda(z-1)}{\Lambda(z)}=1+\frac{o(1)}{z}.
 \]
\end{lemma}
\begin{proof}
Noting that $\Psi(x,q_z)=\Psi(x,q_{z-1})+\Psi(x q_z^{-1},q_z)$, and recalling the
definition~\eqref{def:Lambda} of $\Lambda(z)$, we have
\[
 \frac{\Lambda(z)-\Lambda(z-1)}{\Lambda(z)}
 =\frac{(z-1)\Psi(x,q_z)-z\Psi(x,q_{z-1})}{(z-1)\Psi(x,q_z)}
 =\frac{z\Psi(x q_z^{-1},q_z)-\Psi(x,q_z)}{(z-1)\Psi(x,q_z)}.
\]
Now, applying \Th{Psi:ratio} with $y=q_z$, $a=0$ and $b=1$, we obtain
\[
 \frac{\Psi(x q_z^{-1},q_z)}{\Psi(x,q_z)}=\exp\big(\xi(u)-\log q_z+o(1)\big),
\]
where $u=\frac{\log x}{\log q_z}=(1+o(1))u_0\to\infty$. Hence
\begin{equation}\label{eq:onestep}
 \frac{\Lambda(z)-\Lambda(z-1)}{\Lambda(z)}
 =\frac{1}{z-1}\bigg(\frac{ze^{\xi(u)}}{q_z}\big(1+o(1)\big)-1\bigg)=\frac{o(1)}{z},
\end{equation}
by \Ob{exi}.
\end{proof}

\begin{lemma}\label{lem:Lambda:precise}
 Write\/ $z=z_0\exp\big(c\sqrt{\log z_0}\big)$ for some\/ $c=c(x)$. Then
 \[
  \Lambda(z)=e^{-c^2}+o(1).
 \]
\end{lemma}
\begin{proof}
We may assume without loss of generality that $z=z_0^{1+o(1)}$, and hence
$c=o(\sqrt{\log z_0})$, as otherwise the result follows immediately from \Co{Lambda:rough} with
$\Lambda(z)=o(1)$. Set $y=q_z$ and $u=\frac{\log x}{\log y}$. Since
$\Lambda(z)=J(x)\Psi(x,q_z)/xz$ and $\Lambda(z_0)=1$, it follows by \Th{Psi} that
\[
 \Lambda(z)=\frac{\Lambda(z)}{\Lambda(z_0)}=\frac{z_0}{z}\cdot\frac{\Psi(x,y)}{\Psi(x,y_0)}
 =\big(1+o(1)\big)\frac{z_0}{z}\cdot\frac{\rho(u)}{\rho(u_0)}.
\]
Set $b=u_0-u$ and note that $b=o(u_0)$, which by \Th{rho:ratio} implies that
\[
 \frac{\rho(u)}{\rho(u_0)} = \exp\bigg(b\xi(u_0)+\frac{O(b^2+1)}{u_0}\bigg).
\]
Thus, defining $\kappa_0$ by $u_0\xi(u_0) = ( 1 + \kappa_0 ) \log y_0$, we obtain
\[
 \Lambda(z) = \big( 1+o(1) \big) \frac{z_0}{z}\cdot y_0^{b(1 + \kappa_0)/u_0+o(b^2/u_0^2)},
\]
where we have used~\eqref{eq:ulogu} (in particular, the fact that $u_0=o(\log y_0)$) to replace
the error factor $e^{O(b^2+1)/u_0}$ by $(1+o(1))y_0^{o(b^2/u_0^2)}$.

Now, by the prime number theorem we have
\[
 \frac{y_0}{y}=\big(1+o(1)\big)\frac{z_0\log y_0}{z\log y}=\big(1+o(1)\big)\frac{z_0}{z},
\]
and by the definitions of $u$, $u_0$ and $b$, and the fact that $b=o(u_0)$,
\[
 \frac{y_0}{y}=x^{1/u_0-1/(u_0-b)}=y_0^{-b/(u_0-b)}=y_0^{-b/u_0-(1+o(1))b^2/u_0^2}.
\]
Thus, we obtain
\[
 \Lambda(z)=\big(1+o(1)\big)y_0^{b\kappa_0/u_0-(1+o(1))b^2/u_0^2}
 =\exp\bigg(\frac{b}{u_0}\kappa_0\log y_0-\big(1+o(1)\big)\frac{b^2}{u_0^2}\log y_0+o(1)\bigg),
\]
and hence, noting that
\[
 c\sqrt{\log z_0}=\log \frac{z}{z_0}=\log\frac{y}{y_0}+o(1)
 =\big(1+o(1)\big)\frac{b}{u_0}\log y_0+o(1),
\]
and that $\kappa_0=o(1)$, by~\eqref{eq:ulogu} and~\eqref{eq:xi:facts}, it follows that
\[
 \Lambda(z)=\exp\Big(\big(1+o(1)\big)\kappa_0 c\sqrt{\log y_0}-\big(1+o(1)\big)c^2+o(1)\Big).
\]
But $\Lambda(z)$ is maximized at $c=0$, so this implies that $\kappa_0\sqrt{\log y_0}=o(1)$, which
in turn implies that $\Lambda(z) = e^{-c^2} + o(1)$, as required.
\end{proof}

Note that as an immediate corollary of \Lm{Lambda:precise} we have
\begin{equation}\label{eq:zpm:precise}
 z_\pm=z_0\exp\Big(\pm\big(\sqrt{\log(1/\delta)}+o(1)\big)\sqrt{\log z_0}\Big)
\end{equation}
as well as
\begin{equation}\label{eq:Lambda:delta:plus:littleoone}
 \Lambda(z)\ge\delta+o(1)\qquad\text{for all }z\in[z_-,z_+],
\end{equation}
and
\begin{equation}\label{eq:Lambda:delta:plusminus:littleoone}
 \Lambda(z_\pm)=\delta+o(1).
\end{equation}

Finally, let us make a trivial observation.

\begin{obs}\label{obs:onesinarow}
 $\ds\sum_{j\ge z_-}A_{ij}\le 2u_0$ for every\/ $i\in [N]$.
\end{obs}
\begin{proof}
If some number $a_i$ is divisible by $k$ distinct primes, each larger than $z_-$, then $z_-^k\le x$.
Since $z_-=z_0^{1+o(1)}=y_0^{1+o(1)}=x^{(1+o(1))/u_0}$, this implies that $k\le(1+o(1))u_0$.
\end{proof}

\section{Probabilistic facts and preliminary results}\label{sec:events}

In this section we shall recall some standard probabilistic tools, define some events that will be
important in later sections, and prove some basic facts about these events. Let us begin by
defining the events that encode our induction hypothesis. Recall that $\eta<e^{-\gamma}$ was
fixed in \Sc{outline}, and that $\eps_0$ and $\eps(k,z)$ were defined in~\eqref{def:small}
and \Df{eps}.

\begin{defn}\label{def:MandTk}
 For each $z\in [z_-,\pi(x)]$, let $\cM(z)$ denote the event that
 \begin{equation}\label{eq:Mevent}
  \frac{m(z)}{z}\exp\bigg(-\Ein\bigg(\frac{m(z)}{z}\bigg)\bigg)\in (1\pm\eps_0)\eta\Lambda(z)
\end{equation}
holds, and let $\cM^*(z)$ denote the event that $\cM(w)$ holds for every $w\in [z,\pi(x)]$.
For each $z\in [z_-,z_+]$ and $k\ge 2$, let $\cT_k(z)$ denote the event that
\[
 s_k(z)\in\big(1\pm\eps(k,z)\big)\ts_k(z).
\]
\end{defn}

Note that \Th{track:m} states that the event $\cM^*(z_-)$ holds with high probability, and
\Th{track:s} states that with high probability the event $\cT_k(z)$ holds for every $k\ge 2$ and
every $z\in [z_-,z_+]$. We shall also need the following slightly more technical events.

\begin{defn}\label{def:Q}
For each $z\in [z_-,\pi(x)]$, let $\cQ(z)$ denote the event that
 \begin{itemize}
  \item $\cM(z)$ holds;
  \item $S(z) =\bigcup_{k\ge 2} S_k(z)$ contains no element $w$ with $w>z_0^5$;
  \item $s_k(z)=0$ for all $k\ge 4u_0$.
 \end{itemize}
Let $\cK(z)$ denote the event that $\cQ(z)$ holds and also
 \begin{equation}\label{def:K}
  \sum_{k\ge 2} 2^k s_k(z)\le 2 e^{C_0} m(z),
 \end{equation}
where $C_0$ was defined in~\eqref{def:C0}.
\end{defn}

We shall prove the following two lemmas, which allow us to deduce the technical events (which we
shall need in the sections below) from our induction hypothesis.

\begin{lemma}\label{lem:easy}
 With high probability, $\cQ(z)\cup\cM^*(z)^c$ holds for every\/ $z\in [z_-,\pi(x)]$.
\end{lemma}

In other words, the lemma above says that the probability that there exists $z\in [z_-,\pi(x)]$
such that $\cM^*(z)$ holds but $\cQ(z)$ does not is $o(1)$ as $x\to\infty$.

\begin{lemma}\label{lem:track:K}
 Let\/ $z\in [z_-, z_+]$. If\/ $\cQ(z)$ holds and\/ $\cT_k(z)$ holds for all\/ $k\ge 2$,
 then\/ $\cK(z)$ holds.
\end{lemma}

Before proceeding to the proofs of Lemmas~\ref{lem:easy} and~\ref{lem:track:K}, let us give a
simple but important application of the event $\cK(z)$. Recall from \Al{algorithm}
that $d(z)$ denotes the number of rows of $M(z)$ that contain a 1 in column~$z$.
The following lemma shows that the distribution of $d(z)$ is close to that of a Poisson
random variable with mean $m(z)/z$.

\begin{lemma}\label{lem:dz}
 Let\/ $z\in [z_-,z_+]$. If\/ $\cK(z)$ holds, then
 \begin{equation}\label{eq:lem:dz}
  \Prb\big(d(z)=k\mid\cF_z\big)
  =\frac{\big(1+o(1)\big)^k}{k!}e^{-m(z)/z}\bigg(\frac{m(z)}{z}\bigg)^k+\frac{O(1)}{z}
 \end{equation}
 for every\/ $k\ge 0$.
\end{lemma}

In the proof of \Lm{dz} we shall use the following bound on the sum of independent Bernoulli random
variables due to Le Cam~\cite{Lc}.

\begin{lemma}[Le Cam, 1960]\label{lem:LeCam}
 Let\/ $X_1,\dots,X_n$ be independent Bernoulli random variables, and let\/ $X:=\sum_{i=1}^n X_i$.
 Then
 \[
  \sum_{k\ge 0}\bigg|\Prb(X=k)-\frac{e^{-\mu}\mu^k}{k!}\bigg|\le 2\sum_{i=1}^n p_i^2,
 \]
 where\/ $p_i:=\Prb(X_i=1)$ and\/ $\mu:=\E[X]=\sum_{i=1}^n p_i$.
\end{lemma}

We shall also use the following fact on numerous occasions throughout the paper.

\begin{obs}\label{obs:m:bounded}
 If\/ $\cM(z)$ holds then\/ $m(z)/z\le C_0\Lambda(z)\le C_0$.
\end{obs}
\begin{proof}
Recall from \eqref{def:Ein} and~\eqref{def:alpha} that $\alpha(w)$ and $w e^{-\Ein(w)}$ are
strictly increasing functions, that $\alpha(w)$ is convex, and that
$\alpha(w) e^{-\Ein(\alpha(w))}=w$. It follows that, if $\cM(z)$ holds, then
\[
 \frac{m(z)}{z}\le\alpha\big((1+\eps_0)\eta\Lambda(z)\big)
 \le\alpha\big((1+\eps_0)\eta\big)\Lambda(z)=C_0\Lambda(z)\le C_0,
\]
as claimed, since $C_0=\alpha\big((1+\eps_0)\eta\big)$ and $\Lambda(z)\le 1$.
\end{proof}

\begin{proof}[Proof of \Lm{dz}]
Recall that
\[
 p_z=p_z(x)=\frac{\tpsi(x,q_z)-\tpsi(x,q_{z-1})}{\tpsi(x,q_z)}
\]
denotes the probability that a uniformly chosen random number in $[x]$ is divisible by $q_z$
to an odd power, conditioned on the event that all larger prime factors occur to an even power.
We shall prove that the lemma holds even when, instead of conditioning on~$\cF_z$, we in fact
condition on the entire matrix to the right of column~$z$. In this case,
$d(z)$ is a sum of independent Bernoulli random variables $\{X_i:i\in M(z)\}$ with
$\Prb(X_i=1)=p_z(xq_z^{-\alpha_i})$, where
\[
 q_z^{\alpha_i} =\prod_{w>z,\,A_{iw}=1}q_w
\]
is the product of the primes $q_w$ greater than $q_z$ that divide $a_i$ an odd number of times.

However, the event $\cK(z)$ implies that
\[
 \sum_{i\in M(z)}\alpha_i\le\sum_{i\in M(z)\ }\sum_{w>z,\,A_{iw}=1}6=\sum_{k\ge 2} 6k s_k(z)=O(m(z)),
\]
where the first step follows since $S(z)$ contains no prime $q_w$ with $q_w\ge q_z^6$ (since
$q_z^6\ge z_-^6\ge q_{z_0^5}=z_0^{5+o(1)}$), and the last follows from~\eqref{def:K}.
By the pigeonhole principle, it
follows that $\alpha_i=o(u_0)$ for all but a $o(1)$-proportion
of the $i\in M(z)$. Now by \Co{pj} we have $\Prb(X_i=1)=(1+o(1))/z$ whenever
$\alpha_i=o(u_0)$, and $\Prb(X_i=1)\le(1+o(1))/z$ for every $i\in M(z)$. Thus
\[
 \sum_{i\in M(z)}\Prb(X_i=1)=\big(1+o(1)\big)\frac{m(z)}{z}\quad\text{and}\quad
 \sum_{i\in M(z)}\Prb(X_i=1)^2=\big(1+o(1)\big)\frac{m(z)}{z^2}
\]
whenever $\cK(z)$ holds. As $m(z)=O(z)$ by \Ob{m:bounded} (since $\cK(z)$ implies $\cM(z)$),
it follows by \Lm{LeCam} that
\[
 \sum_{k\ge 0}\bigg|\Prb\big(d(z)= k\mid\cF_z\big)-\frac{e^{-\mu}\mu^k}{k!}\bigg|=\frac{O(1)}{z},
\]
where $\mu:=\E[d(z)]=(1+o(1))m(z)/z$, as required.
\end{proof}

For the next result it will be useful to have the following simple estimate
on the function $\alpha(w)$ defined in \eqref{def:alpha}.
\begin{obs}
For $0\le w\le 0.2$ we have
\begin{equation}\label{eq:alpha:bounds}
 w\le\alpha(w)\le w+2w^2.
\end{equation}
\end{obs}
\begin{proof}
The lower bound follows from \eqref{def:alpha} as $we^{-\Ein(w)}\le w$ and hence $\alpha(w)\ge w$.
The upper bound also follows from \eqref{def:alpha} as $\Ein(w)\le w$, and
\begin{align*}
 (w+2w^2)e^{-\Ein(w+2w^2)}&\ge (w+2w^2)e^{-(w+2w^2)}\ge (w+2w^2)\big(1-(w+2w^2)\big)\\
 &=w+w^2(1-4w-4w^2)\ge w,
\end{align*}
so $\alpha(w)\le w+2w^2$ for $0\le w\le 0.2$.
\end{proof}

In \Sc{big:z} we shall also need the following weaker bounds on the distribution of $d(z)$
in the range $[z_+,\pi(x)]$.

\begin{lemma}\label{lem:dkvsd1:bigz}
 Let\/ $z\in [z_+,\pi(x)]$. If\/ $\cK(z)$ holds, then
 \[
  \Prb\big(d(z)\ge 1\mid\cF_z\big)\le\big(1+o(1)\big)m(z)p_z(x),
 \]
 and moreover
 \[
  \Prb\big(d(z)=k\mid\cF_z\big)\le\frac{(2\delta\eta)^{k-1}}{k!}\ \Prb\big(d(z)=1\mid\cF_z\big)
 \]
 for every\/ $k\ge 2$.
\end{lemma}
\begin{proof}
Let the independent Bernoulli random variables $\{X_i:i\in M(z)\}$ be defined as in the proof
of \Lm{dz}, and observe that, since $\alpha_i \in \{0\} \cup [1,\infty)$ for every $i\in M(z)$, we have
\[
 \Prb\big(X_i=1\big)\le\big(1+o(1)\big)p_z(x)\le\frac{1+o(1)}{z}
\]
for every $i\in M(z)$, by \Co{pj}. The first claimed inequality now follows by the union bound.
For the second bound we note that in general if $p_i:=\Prb(X_i=1)$ then
\begin{align*}
 \frac{k\,\Prb\big(\sum_iX_i=k\big)}{\Prb\big(\sum_i X_i=1)}
 &=\frac{\sum_{i_0\in[m(z)]}\sum_{i_0\in S\subseteq[m(z)],\,|S|=k}\prod_{i\in S}p_i\prod_{i\notin S}(1-p_i)}
  {\sum_{i_0\in[m(z)]}p_{i_0}\prod_{i\ne i_0}(1-p_i)}\\
 &\le\max_{i_0}\sum_{S\subseteq[m(z)]\setminus\{i_0\},\,|S|=k-1}\prod_{i\in S}\frac{p_i}{1-p_i}\\
 &\le \binom{m(z)-1}{k-1}\max_i\bigg(\frac{p_i}{1-p_i}\bigg)^{k-1}\\
 &\le \frac{m(z)^{k-1}}{(k-1)!}\max_i\bigg(\frac{p_i}{1-p_i}\bigg)^{k-1}.
\end{align*}
Thus
\[
 \frac{\Prb\big(d(z)=k\mid\cF_z\big)}{\Prb\big(d(z)=1\mid\cF_z\big)}
 \le\frac{1}{k!}\bigg((1+o(1))\frac{m(z)}{z}\bigg)^{k-1}.
\]
The result follows as $\cM(z)$ and $z\in [z_+,\pi(x)]$ imply that
\[
 (1+o(1))\frac{m(z)}{z}\le(1+o(1))\alpha\big((1+\eps_0)\eta\Lambda(z)\big)
 \le \alpha(1.6\delta\eta)\le 2\delta\eta
\]
by \eqref{eq:alpha:bounds} as $\Lambda(z)\le\delta+o(1)$ and $\eps_0<e^{-\gamma}<0.6$.
\end{proof}

\subsection{The proof of \Lm{easy}}

The first step is to control $m_0(z)$, the number of isolated vertices in the hypergraph
$\cS_A(z)$. As noted above, this simple fact lies at the heart of our proof, and will be used
several times in later sections.

\begin{lemma}\label{lem:m0}
 For each\/ $z\in [z_-,\pi(x)]$,
 \[
  \Prb\Big(m_0(z)\notin\big(1\pm\eps_1\big)\eta\Lambda(z)z\Big)\le\frac{1}{x^2}.
 \]
 In particular, with high probability, $m_0(z)\in(1\pm\eps_1)\eta\Lambda(z)z$ for every\/
 $z\in [z_-,\pi(x)]$.
\end{lemma}

Note that since $m(z)\ge m_0(z)$, this implies in particular that with high probability we have
$m(z)\ge (1-\eps_1)\eta\Lambda(z)z$ for every $z\in [z_-,\pi(x)]$. We remark (and also note for
future reference) that the event $\cM(z)$ implies deterministically that
\begin{equation}\label{eq:lowerbound:m:using:M}
 m(z)\ge\alpha\big((1-\eps_0)\eta\Lambda(z)\big)z
 \ge (1-\eps_0)\eta\Lambda(z)z\ge (1-\eps_0)\eta\Lambda(z_-)z_-\ge z_0^{1+o(1)}
\end{equation}
for every $z\in[z_-,\pi(x)]$, by \eqref{eq:mab}, \eqref{eq:alpha:bounds}, \eqref{eq:zpm:rough},
and the fact that $\Lambda(z)z=J(x)\Psi(x,q_z)/x$ is increasing in~$z$.

\Lm{m0} is a straightforward consequence of the following special case of the well-known
inequality of Chernoff~\cite{Chernoff}.

\begin{lemma}[Chernoff's inequality]\label{lem:Chernoff}
 Let\/ $X_1,\dots,X_n$ be independent Bernoulli random variables, and let\/ $X:=\sum_{i=1}^n X_i$.
 Then for any\/ $\eps>0$,
 \[
  \Prb\big(|X-\mu|\ge\eps\mu\big)\le 2e^{-\eps^2\mu/(2+\eps)},
 \]
 where\/ $\mu:=\E[X]$.
\end{lemma}
\begin{proof}[Proof of \Lm{m0}]
The number of isolated vertices in $\cS_A(z)$ is precisely the number of rows of $A$ with no
non-zero entry to the right of column~$z$ (all of which will lie in $M(z)$).
This is also the same as the number of integers
$a_i$, $i\in [N]$ such that every prime $q>q_z$ divides $a_i$ an even number of times.
Let $\mu$ denote the expected number of $a_i$ with this property, and observe that
\[
 \mu=\eta J(x)\cdot\frac{\tpsi(x,q_z)}{x}
 =\big(1+o(1)\big)\eta J(x)\cdot\frac{\Psi(x,q_z)}{x}
 =\big(1+o(1)\big)\eta\Lambda(z)z
\]
by \Lm{tpsi} and the definition~\eqref{def:Lambda} of $\Lambda(z)$. Since the numbers $a_i$ are
independent, it follows by \Lm{Chernoff} that
\[
 \Prb\big(m_0(z)\notin\big(1\pm\eps_1\big)\eta\Lambda(z)z\big)
 \le\Prb\big(m_0(z)\notin\big(1\pm\eps_1/2\big)\mu\big)
 \le 2e^{-\eps_1^2\mu/(8+2\eps_1)}.
\]
Now simply note that $\mu$ is increasing in $z$, so (approximating very crudely) we have
$\mu\ge (1+o(1))\eta\Lambda(z_-)z_-\ge\eta\delta z_-/2\ge (\log x)^2$, and
hence the right-hand side of the above
inequality is at most $1/x^2$. The last part follows by the union bound over
$z\in [z_-,\pi(x)]$.
\end{proof}

We shall next use the lower bound on $m(z)$ given by \Lm{m0} to prove the following upper
bounds on the random variables $s_k(z)$.

\begin{lemma}\label{lem:skzero}
 With high probability, the following all hold:
 \begin{itemize}
  \item[$(a)$] $s_k(z)=0$ for every\/ $k\ge 2$ and every\/ $z\ge z_0^5$.
  \item[$(b)$] $s_2(z)+s_3(z)\le z_0^{-1/2}m(z)$ for every\/ $z\ge z_0^3$.
  \item[$(c)$] $s_k(z)=0$ for every\/ $k\ge 4$ and every\/ $z\ge z_0^3$.
 \end{itemize}
\end{lemma}
\begin{proof}
Note that a prime $q$ divides a uniformly chosen random element of $[x]$ with probability
$\lfloor x/q\rfloor/x\le 1/q$. Recall that $A$ has $N$ rows, and that $N=\eta J(x)=z_0^{2+o(1)}$,
by~\eqref{eq:J:rough}. It follows that the expected number of primes $q\ge w$ that divide at least
$k$ of the integers $a_i$, $i\in [N]$, is at most
\begin{equation}\label{eq:largeprimes}
 \sum_{q\ge w}\frac{N^k}{q^k}\le\frac{z_0^{(2+o(1))k}}{w^{k-1}}.
\end{equation}
Now if $s_k(z)>0$ then there must be a prime $q>q_z$ that divides at least $k$ of the~$a_i$.
Hence applying \eqref{eq:largeprimes} with $k\ge2$ and $w=q_{z_0^5}\ge z_0^5$ gives part~$(a)$,
and applying it with $k\ge4$ and $w=q_{z_0^3}\ge z_0^3$ gives part~$(c)$.
To prove~$(b)$, observe first that with high probability
\[
 m(z)\ge m_0(z)\ge (1-\eps_1)\eta\Lambda(z)z
 \ge(1-\eps_1)\eta J(x)\cdot\frac{\Psi(x,q_{z_0^3})}{x}=z_0^{5/3+o(1)}
\]
for every $z\ge z_0^3$. Indeed, the first inequality is trivial, the second follows by \Lm{m0},
the third since $\Psi(x,y)$ is increasing in~$y$, and the fourth by \eqref{eq:J:rough}
and~\eqref{eq:Psi:rough}. Hence, applying \eqref{eq:largeprimes} with $k=2$ and
$w=q_{z_0^3}\ge z_0^3$, it follows that the expected number of primes that can contribute to the
value of $s_2(z)+s_3(z)$ for any $z\ge z_0^3$ is at most $z_0^{1+o(1)}\le z_0^{-2/3+o(1)}m(z)$.
Part~$(b)$ then follows by Markov's inequality.
\end{proof}

We similarly obtain the following bound on $s_k(z)$ for large~$k$. Note that
$u_0=(1+o(1))\log z_0/\log\log z_0$ by~\eqref{eq:ulogu}.

\begin{lemma}\label{lem:skeasy}
 With high probability, for every\/ $z\in [z_-,\pi(x)]$ either\/ $\cM^*(z)$ fails to hold,
 or\/ $s_k(z)=0$ for every\/ $k\ge 4u_0$.
\end{lemma}
\begin{proof}
Suppose that $\cM^*(z)$ holds, and recall that this implies $m(w)\le C_0 w$ for all $w\ge z$, by
\Ob{m:bounded}. Any element $w\in\bigcup_{k\ge 4u_0}S_k(z)$ must have been `born' with
$d(w)\ge 4u_0$. But, by \Lm{skzero}, with high probability no such $w$ exists in $[z_0^3+1,\pi(x)]$
as it would contribute to $s_k(w-1)$ for some $k\ge 4u_0$. Thus it is enough to show that with
high probability no $w$ exists in $[z_-,z_0^3]$ with $d(w)\ge 4u_0$ and $m(w)\le C_0 w$.

The probability that $A_{iw}=1$, conditioned on all entries of $A$ to the right of column~$w$,
is always at most $(1+o(1))/w$, by \Co{pj}, and is conditionally independent for each~$i$.
It follows that
\begin{equation}\label{eq:nottoomanyinacolumn}
 \Prb\big(d(w)\ge k\mid\cF_w\big)\le\binom{m(w)}{k}\bigg(\frac{1+o(1)}{w}\bigg)^k
 \le\frac{(2C_0)^k}{k!}
\end{equation}
when $m(w)\le C_0w$. Thus the expected number of $w$ in $[z_-,z_0^3]$ with $d(w)\ge4u_0$
and $m(w)\le C_0w$ is at most
\[
 \frac{(2C_0)^{4u_0}}{(4u_0)!}\cdot z_0^3\le
 \exp\big(-4u_0\log u_0+O(u_0)+3\log z_0\big)=o(1)
\]
by Stirling's formula and the fact that $\log z_0=(1+o(1))u_0\log u_0$, by~\eqref{eq:ulogu}.
Hence, with high probability, no such $w$ exists.
\end{proof}

\Lm{easy} follows immediately from Lemmas~\ref{lem:skzero} and~\ref{lem:skeasy}. Note that
if $w\in S_k(z)$ then $w\in S_{k'}(w-1)$ for some $k'\ge k$, so if $s_k(w)=0$
for all $k\ge2$, $w\ge z_0^5$, then $S_k(z)$ cannot contain any element $w>z_0^5$.

\subsection{The proof of \Lm{track:K}}

We shall next prove that the event $\cQ(z)\cap\bigcap_{k\ge 2}\cT_k(z)$ implies
(deterministically) that the event $\cK(z)$ holds. To do so, we need to prove that
$\sum_{k\ge 2}2^ks_k(z)\le 2e^{C_0}m(z)$. We begin with a simple but useful observation.

\begin{obs}\label{obs:stildes:ineq}
 $(k+1)\ts_{k+1}(z)\le\ds\frac{m(z)}{z}\cdot\ts_k(z)$.
\end{obs}
\begin{proof}
This follows easily from \Df{tsk}. Indeed, writing $\lambda:=m(z)/z$, we have
\[
 (k+1)\ts_{k+1}(z)
 =\frac{m(z)}{k}e^{-\lambda}\sum_{\ell=k-1}^\infty\frac{\lambda^{\ell+1}}{(\ell+1)!}
 \le\lambda\cdot\frac{m(z)}{k(k-1)}e^{-\lambda}\sum_{\ell=k-1}^\infty\frac{\lambda^{\ell}}{\ell!}
 =\lambda\cdot\ts_k(z),
\]
as claimed.
\end{proof}

We shall first prove the following bound on the sum over $k$ of $2^k\eps(k,z)\ts_k(z)$.
We shall need this bound again in Sections~\ref{sec:proof:tracking} and~\ref{sec:squares:proof}.

\begin{lemma}\label{lem:sum:eps:tsk}
 Let\/ $z\in [z_-,z_+]$. If\/ $\cM(z)$ holds, then
 \[
  \sum_{k=2}^\infty 2^k\eps(k,z)\ts_k(z)\le\eps_1 m(z).
 \]
\end{lemma}
\begin{proof}
Note that
\begin{equation}\label{eq:sum:eps:tsk}
 \sum_{k\ge 2}2^k\eps(k,z)\ts_k(z)
 =\sum_{k\ge 2}\frac{2^k\eps_1^k}{\Lambda(z)}k!\ts_k(z)
 \le\sum_{k\ge 2}\frac{2^k\eps_1^k}{\Lambda(z)}\bigg(\frac{m(z)}{z}\bigg)^{k-2}2\ts_2(z)
\end{equation}
by \Df{eps} and \Ob{stildes:ineq}. Now
\[
 2\ts_2(z)=m(z)\big(1-e^{-m(z)/z}\big)\le\frac{m(z)^2}{z}\le C_0\Lambda(z)m(z),
\]
and $m(z)\le C_0z$, by \eqref{eq:ts2} and \Ob{m:bounded}. Noting from~\eqref{def:small} that
$\eps_1 C_0<\eps_1 e^{C_0}\le 1/16$,
it follows that the right-hand side of~\eqref{eq:sum:eps:tsk} is at most
\[
 \sum_{k\ge 2}2^k\eps_1^k C_0^{k-1}m(z)=\frac{4\eps_1C_0}{1-2\eps_1 C_0}\eps_1 m(z)\le\eps_1 m(z)
\]
as required.
\end{proof}

\begin{proof}[Proof of \Lm{track:K}]
Since $\cT_k(z)$ holds for all $k\ge 2$, we have
\[
 \sum_{k\ge 2}2^k s_k(z)\le\sum_{k=2}^\infty 2^k\ts_k(z)+\sum_{k=2}^\infty 2^k\eps(k,z)\ts_k(z).
\]
The second term is at most $\eps_1 m(z)$, by \Lm{sum:eps:tsk}, and since $\cQ(z)$ implies that the
event $\cM(z)$ holds. To bound the first term, observe that, writing $\lambda=m(z)/z$, we have
\begin{align}
 \sum_{k\ge 2}2^k\ts_k(z)
 &=m(z)e^{-\lambda}\sum_{k\ge 2}\frac{2^k}{k(k-1)}\sum_{\ell=k-1}^\infty\frac{\lambda^\ell}{\ell!}
 =m(z)e^{-\lambda}\sum_{\ell=1}^\infty\frac{\lambda^\ell}{\ell!}\sum_{k=2}^{\ell+1}\frac{2^k}{k(k-1)}\notag\\
 &\le m(z)e^{-\lambda}\sum_{\ell=1}^\infty\frac{\lambda^\ell}{\ell!}\cdot 2^\ell
 \le m(z)e^{-\lambda}e^{2\lambda}=e^{m(z)/z}m(z)\le e^{C_0}m(z),\label{eq:sumsk}
\end{align}
as required, where in the final step we used \Ob{m:bounded}.
\end{proof}

Finally, let us make a simple observation which will play an important role in \Sc{critical}.

\begin{obs}\label{obs:eps:tsk:rough}
 If\/ $\cM(z)$ holds then\/ $\eps(k,z)\ts_k(z)=z_0^{1+o(1)}$ uniformly
 for every\/ $z\in [z_-,z_+]$ and\/ $2\le k\le 4u_0$.
\end{obs}
\begin{proof}
Let $z\in [z_-,z_+]$, and observe that, since $\cM(z)$ implies $m(z)/z=\Theta(1)$, we have
\begin{equation}\label{eq:eps:tsk:rough}
 \eps(k,z)\ts_k(z)=\frac{\eps_1^k\cdot k!}{\Lambda(z)}\cdot\frac{m(z)}{k(k-1)}e^{-m(z)/z}
 \sum_{\ell=k-1}^\infty\frac{1}{\ell!}\bigg(\frac{m(z)}{z}\bigg)^\ell=e^{O(k)}z.
\end{equation}
Since $u_0=o(\log z_0)$, by~\eqref{eq:ulogu}, and $z=z_0^{1+o(1)}$ for every
$z\in [z_-,z_+]$, by~\eqref{eq:zpm:rough}, it follows that $\eps(k,z)\ts_k(z)=z_0^{1+o(1)}$
uniformly for every $z\in[z_-,z_+]$ and $k\le 4u_0$, as claimed.
\end{proof}

\subsection{Martingales}

We finish this section by recalling some standard results about martingales that we shall use in
later sections. Recall that a \emph{super-martingale} with respect to a filtration
$(\cF_t)_{t \ge 0}$ of $\sigma$-algebras, is a sequence of random variables $(X_t)_{t\ge 0}$ such
that the following hold for each $t\ge 0$: $X_t$ is $\cF_t$-measurable, $\E[|X_t|]<\infty$,
and $\E\big[X_{t+1}\mid\cF_t\big]\le X_t$.

The following inequality was proved by Azuma~\cite{Azuma} and Hoeffding~\cite{H} (see, e.g.,~\cite{BelaRG}).

\begin{lemma}[The Azuma--Hoeffding inequality]
 Let\/ $(X_t)_{t=0}^{\ell}$ be a super-martingale with respect to a filtration\/
 $(\cF_t)_{t=0}^{\ell}$, and assume\/ $\Prb(|X_t-X_{t-1}|>c_t)=0$
 for\/ $t=1,\dots,\ell$. Then
 \[
  \Prb\big(X_\ell-X_0\ge a\big)\le\exp\bigg(\frac{-a^2}{2\sum_{t=1}^\ell c_t^2}\bigg)
 \]
 for every\/ $a>0$.
\end{lemma}

Recall that a \emph{stopping time} $T$ with respect to the filtration $(\cF_t)_{t\ge0}$
is a non-negative integer valued random variable such that the event $\{T\le t\}$ is
$\cF_t$-measurable. A stopping time $T$ is called \emph{bounded} if there exists a
deterministic $C>0$ such that $\Prb(T\le C)=1$. We shall require the following
well known theorem (see, for example,~\cite{DW}).

\begin{lemma}[Optional Stopping Theorem]\label{lem:OST}
 Suppose\/ $(X_t)_{t\ge0}$ is a super-martingale and\/ $T$ is a bounded stopping time.
 Then\/ $\E[X_T]\le\E[X_0]$.
\end{lemma}

\section{Approximation by an independent hypergraph model}\label{sec:compare}

In order to control the evolution of the variables $m(z)$ and $s_k(z)$ as $z$ decreases, we shall
need to understand the structure of the 2-core $\cC_A(z)$ of $\cH_A(z)$, conditioned on~$\cF^+_z$.
More precisely, we shall need to prove good approximations for the probability that certain
substructures occur in~$A$, conditioned on the column sums (over the rows $M(z)$) of the columns
$[z,\pi(x)]$. We shall show that, up to relatively small error, these probabilities are the same
as they would be if the columns were independent.

Note that without any conditioning the rows are independent, and the entries in the rows are
almost independent, meaning that we can estimate the probability of given hypergraph structures
using~\eqref{eq:probs}. The problem is that we wish to condition on an event $\cE\in\cF^+_z$ of
the form
\begin{equation}\label{def:eventE}
 \cE:=\bigg\{M(z)=M\text{ and }\sum_{i\in M(z)}A_{ij}=d_j,\text{ for }j\ge z\bigg\},
\end{equation}
which specifies $M(z)$ and all the degrees of vertices in $\cC_A(z)$ as well as~$d(z)$.
Thus both rows and columns are now dependent. If the entries in each row of $A$ were independent
then we could just forget the probabilities $p_j(x)$ and model the matrix as placing $d_j$ 1s in
column $j$ uniformly at random. Unfortunately this is not the case, so we need to be a bit more
careful.

The substructures that we shall need to consider involve a (typically small) subset
$I\subseteq M(z)$ of rows of~$A$, and we shall need to estimate the probability that the entries
in these rows are of a given form. The most general version of this requirement is that the
submatrix obtained by just considering the set $I$ of rows and a (usually larger)
set $C\subseteq [z,\pi(x)]$ of
columns forms a specific $I\times C$ matrix~$R$. This corresponds to specifying exactly which
vertices of $C$ lie in a fixed set $I$ of edges in~$\cC_A(z)$.

In order to state the main result of this section (\Th{compare}, below), we shall need some
notation. Given any matrix $B$ and subsets $I$ and $C$ of the rows and columns of $B$
respectively, define $B[I\times C]$ to be the submatrix of $B$ given by the set $I$ of rows
and the set $C$ of columns. Given $z\in [z_-,\pi(x)]$ and an event $\cE\in\cF_z^+$ of the
form~\eqref{def:eventE}, which determines the set $M(z)=M$ and the sequence $(d_j)_{j\ge z}$,
let $\tA_\cE$ denote the random $M\times [z,\pi(x)]$ matrix obtained by choosing (for each~$j$)
column $j$ uniformly among the $\binom{|M|}{d_j}$ possible choices of column with column
sum~$d_j$, independently for each column. We shall write $m:=|M|$ and
$A_M:=A\big[M\times [z,\pi(x)]\big]$ for the corresponding submatrix of~$A$, even when $\cE$
does not hold. Of course, if $\cE$ holds then $m(z)=m$ and $M(z)=M$.

Assume $I\subseteq M$ and $C\subseteq [z,\pi(x)]$. If $R$ is an $I\times C$ 0-1 matrix, define
\[
 |R|_1=\sum_{i\in I} \sum_{j \in C}R_{ij}\qquad\text{and}\qquad
 |R|_2=\sum_{i\in I}\bigg(\sum_{j\in C}R_{ij}\bigg)^2.
\]
We shall prove the following theorem, which allows us to control the (conditional) probability of
each `basic event' of the form $\{A[I\times C]=R\}$ in terms of the probability of the
corresponding event $\{\tA_\cE[I\times C]=R\}$ in the independent model~$\tA_\cE$.

\begin{thm}\label{thm:compare}
 Let\/ $z\in [z_-,\pi(x)]$ and let\/ $\cE\in\cF_z^+$ be an event of the form~\eqref{def:eventE}
 such that\/ $\cK(z)$ holds and\/ $d(z)\le 4u_0$. Let\/ $I\subseteq M$, $C\subseteq [z,\pi(x)]$
 with\/ $|I|=e^{O(u_0)}$, and let\/ $R$ be an\/ $I\times C$ $0$-$1$ matrix with\/
 $|R|_1=O(|M|/u_0)$. Then
 \begin{equation}\label{eq:compare:ineq}
  \Prb\big(A[I\times C]=R\mid\cE\big)
  \le\exp\bigg(\frac{O\big(|I|+|R|_1\big)}{u_0}\bigg)\Prb\big(\tA_\cE[I\times C]=R\big).
 \end{equation}
 Moreover, if every row sum of\/ $R$ is at most\/ $u_0/150$, then
 \begin{equation}\label{eq:compare:eq}
  \Prb\big(A[I\times C]=R\mid\cE\big)
  =\exp\bigg(\frac{O\big(|I|+|R|_2\big)}{u_0}\bigg)\Prb\big(\tA_\cE[I\times C]=R\big).
 \end{equation}
\end{thm}

For each $I\subseteq M$ and $C\subseteq [z,\pi(x)]$ let us define $\cR_\cE(I,C)$ to be the
collection of $I\times C$ 0-1 matrices $R$ whose $j$th column sum does not exceed~$d_j$ for each $j\in C$.
Note that matrices that do not have this property are inconsistent with the event~$\cE$, and so
the theorem holds trivially for such matrices, since all the probabilities are then zero.

If $\cK(z)$ holds then $d_j=0$ for $j>z_0^5$, so if $R\in\cR_\cE(I,C)$ then
the probabilities in \eqref{eq:compare:ineq} and \eqref{eq:compare:eq} are unaffected if we
replace $C$ by $C\cap[z,z_0^5]$. Thus we may assume without loss of generality that
$C\subseteq[z,z_0^5]$. Similarly, we may assume without loss of generality that $z\le z_0^5$.
Thus, for convenience, let us fix for the rest of the section an integer $z\in [z_-,z_0^5]$ and
an event $\cE\in\cF_z^+$ of the form~\eqref{def:eventE} such that $\cK(z)$ holds and
$d(z)\le 4u_0$. We note that the conditions $\cK(z)$ and $d(z)\le 4u_0$ also imply that
\begin{equation}\label{eq:sumdj}
 \sum_{j=z}^{\pi(x)} d_j=d(z)+\sum_{k\ge2}ks_k(z)=O(m(z)),
\end{equation}
and that $m(z)\ge z_0^{1+o(1)}$, by~\eqref{eq:lowerbound:m:using:M}.

The idea of the proof is to randomly permute some of the entries of $A$ in $M\times\{j\}$
for each $j\in C$ so as to obtain a random submatrix in $M\times C$ that is
zero in $I\times C$, and show that in most cases the expected probability of obtaining this
permuted matrix is roughly the same as obtaining the original. This will not always be the case,
but such exceptional cases occur rarely, and the probabilities in these cases are smaller than
normal, so they have little effect on the probability of seeing a particular pattern in
the entries $I\times C$. This allows us to reduce the general case to the case $R=0$.
Moreover, by summing over $R$ we can reduce the trivial `true' event to the case $R=0$, thus
indirectly estimating the probability when $R=0$.
Lemmas \ref{lem:compare:ineq} and \ref{lem:compare:eq} below
will prove \eqref{eq:compare:ineq} and \eqref{eq:compare:eq}
respectively, subject to the result holding for $R=0$. \Lm{compare:zero} will then deal with
the case when $R=0$. One reason for splitting the result into three lemmas is that the proof of
the lower bound in \Lm{compare:eq} actually relies heavily on the upper bound from
\Lm{compare:ineq}, and the proof of \Lm{compare:zero} also relies heavily on \Lm{compare:eq}.

Given an $M\times [z,\pi(x)]$ matrix $B$, we say $B$ is \emph{consistent with} $\cE$
if the column sums $\sum_{i\in M}B_{ij}$ are equal to $d_j$ for all $j\ge z$. Let
$\cB$ be the set of all $M\times[z,\pi(x)]$ 0-1 matrices that are consistent with $\cE$
and for $R\in\cR_\cE(I,C)$, define
\[
 \cB_R:=\{B\in\cB:B[I\times C]=R\}\qquad\text{and}\qquad
 \cB_0:=\{B\in\cB:B[I\times C]=0\}.
\]
We will use the following simple observation several times in the proof below.

\begin{obs}\label{obs:condce}
$\Prb(A_M=B\mid\cE)=\Prb(A_M=B\mid A_M\in\cB)$ for every\/ $B \in \cB$, and hence
$$\Prb\big( A[I\times C] = R \mid \cE \big) = \Prb\big( A_M \in \cB_R \mid A_M \in \cB \big)$$
for every $R \in \cR_\cE(I,C)$.
\end{obs}

\begin{proof}
Note that the event $\cE$ is equal to $\{M(z)=M\} \cap \{A_M\in\cB\}$. We claim that,
conditional on the event $A_M \in \cB$, the event $M(z)=M$ depends only on rows outside of $M$,
and so is independent of the choice of $A_M\in\cB$. Indeed, since every $d_j$ (for $j > z$)
is either zero or at least 2, none of the rows of $M$ will be deleted by the algorithm that
constructs the 2-core $\cC_A(z)$ (to see this, consider the first such row to be deleted).
Thus the event $M(z) = M$ holds (conditional on $A_M\in\cB$) if and only if all other rows
are removed by step $z$ of the algorithm, which depends on the sequence $(d_{z+1},\ldots,d_{\pi(x)})$ and on
the rows outside $M$, but not on the choice of $A_M \in \cB$, as claimed. This proves the
first statement, and the second follows immediately since $\{ A[I\times C] = R \} \cap \cB = \cB_R$.
\end{proof}

Given $B\in\cB$, define a random matrix $\phi_{I,C}(B)$ as follows.
For each $j\in C$, remove all 1s in the submatrix $B[I\times\{j\}]$ and place them on a
uniformly chosen random subset of the zero entries in $B[(M\setminus I)\times\{j\}]$, choosing the
random subsets independently for each $j \in C$.
The result is a matrix $\phi_{I,C}(B)$ with the same column sums, and hence $\phi_{I,C}(B)$ is still
consistent with $\cE$, but $\phi_{I,C}(B)[I\times C] = 0$, so $\phi_{I,C}(B)\in\cB_0$.
The choices made in the construction of $\phi_{I,C}(B)$ will always be assumed to be
independent of any random choice of $B$, or the matrix~$A$.
The following observation is then immediate.

\begin{obs}\label{obs:uniform:B}
 For any fixed\/ $R \in\cR_\cE(I,C)$, if\/ $B$ is chosen uniformly at random from\/ $\cB_R$ then
 the distribution of\/ $\phi_{I,C}(B)$ is uniform on\/~$\cB_0$.
\end{obs}
\begin{proof} Indeed, the distribution is invariant under any permutation of the rows $M\setminus I$.
\end{proof}

Now, for any 0-1 matrix $B$ and each row $i$ of $B$, let
\[
 t_i(B)=\prod_{j\in[z,\pi(x)]}q_j^{B_{ij}},
\]
and define the \emph{weight} of the $i$th row, $w_i(B)$, by
$$t_i(B)=q_{z-1}^{w_i(B)}.$$
For completeness, define $w_i(B)=0$ if $i$ is not a row of $B$.

\begin{obs}\label{obs:weight}
 If\/ $B_{ij}=0$ for all\/ $j\notin[z,z_0^5]$, then\/
 $\sum_j B_{ij}\le w_i(B)\le 6\sum_j B_{ij}$.
\end{obs}
\begin{proof}
All primes dividing $t_i(B)$ are in
the range $[q_z,q_{z_0^5}]$, and $q_{z_0^5}=z_0^{5+o(1)}\le z_-^6\le q_{z-1}^6$.
\end{proof}

Assume $B\in\cB_R$. Let $B^-$ be the matrix obtained from $B$
by setting all entries in $I\times C$ to~0. Note that $B^-$ is not in general consistent
with~$\cE$. Write
\[
 \delta_i:=w_i(B)-w_i(B^-),\qquad\text{and}\qquad
 \delta_i^\phi:=w_i(\phi_{I,C}(B))-w_i(B^-).
\]
Note that $\delta_i=w_i(R)$ depends only on $R$ and that $\delta_i=0$ for $i\notin I$ while
$\delta_i^\phi=0$ for $i\in I$. Moreover, by \Ob{weight}, we have $\delta_i \in \{0\} \cup [1,\infty)$ for every $i \in M$, and $\delta_i$ (respectively $\delta_i^\phi$) is, up to a constant factor, equal to the number of 1s removed
from (respectively added to) row $i$ by the map~$\phi_{I,C}$. Define
$$\Delta:=\sum_{i\in I}\delta_i=\sum_{i\in M\setminus I}\delta^\phi_i,$$
and note that $\Delta=\Theta(|R|_1)$ and $\Delta$ depends only on~$R$.
Also observe that, by \eqref{eq:sumdj} and \Ob{weight},
\begin{equation}\label{eq:sumalpha}
 \sum_{i\in M}w_i(B)=O(|M|)
\end{equation}
for any $B \in \cB_R$. We shall also write
\begin{equation}\label{eq:alphaR}
 \alpha_R(B):=\sum_{i\in I}\delta_i w_i(B^-).
\end{equation}
Our first challenge will be to prove the following lemma.

\begin{lemma}\label{lem:comp:phi}
 Let\/ $z\in [z_-,z_0^5]$ and let\/ $\cE\in\cF_z^+$ be an event of the form~\eqref{def:eventE}
 such that\/ $\cK(z)$ holds and\/ $d(z)\le 4u_0$. Let\/ $I\subseteq M$
 and\/ $C\subseteq [z,z_0^5]$ be such that\/ $|I|=e^{O(u_0)}$, and let\/ $R\in\cR_\cE(I,C)$.
 If\/ $|R|_1=O(|M|/u_0)$, then for all\/ $B\in\cB_R$,
 \begin{equation}\label{eq:comp:phi:ineq}
  \Prb\big(A_M=B)\le\exp\bigg(\frac{O\big(|R|_1\big)}{u_0}\bigg)\Prb\big(A_M=\phi_{I,C}(B) \big),
 \end{equation}
 Moreover, if no row\/ $i\in I$ of\/ $B$ contains more than\/ $u/12$ $1$s, where $q_{z-1}=x^{1/u}$, then
 \begin{equation}\label{eq:comp:phi:eq}
  \Prb\big(A_M=B)=\exp\bigg(\frac{O\big(\alpha_R(B)+|R|_2\big)}{u_0}\bigg)\Prb\big(A_M=\phi_{I,C}(B)).
 \end{equation}
\end{lemma}

Note that the probabilities here are \emph{unconditional}, and are both over the choice of $A$
and the (uniform and independent) choice of~$\phi_{I,C}(B)$.

\begin{proof}
For any $B\in\cB_R$ there are exactly $\tpsi(x/t_i(B),q_{z-1})$ choices of integer
$a_i\in [x]$ such that $A\big[\{i\}\times[z,\pi(x)]\big]=B\big[\{i\}\times[z,\pi(x)]\big]$.
Thus, by counting the number of choices for $a_i$, $i\in M$, we have
\begin{equation}\label{eq:comp:phi:basicid}
 \frac{\Prb(A_M=B)}{\Prb(A_M=B^-)}=\prod_{i\in M}
 \frac{\tpsi\big(x/t_i(B),q_{z-1}\big)}{\tpsi\big(x/t_i(B^-),q_{z-1}\big)}
 =\prod_{i\colon\delta_i>0}\frac{\tpsi\big(xq_{z-1}^{-w_i(B^-)-\delta_i},q_{z-1}\big)}
 {\tpsi\big(xq_{z-1}^{-w_i(B^-)},q_{z-1}\big)}.
\end{equation}
Thus, by \Th{Psi:ratio}, and recalling that $\delta_i \in \{0\} \cup [1,\infty)$ for every $i \in M$, we have
\begin{align}
 \log\frac{\Prb(A_M=B)}{\Prb(A_M=B^-)}
 &\le\sum_{i\colon\delta_i>0}\bigg(\delta_i \big( \xi(u)-\log q_{z-1} \big)+\frac{O(1)}{u}\bigg)\notag\\
 &=\Delta \big( \xi(u)-\log q_{z-1} \big)+\frac{O(|R|_1)}{u_0}.\label{eq:comp:B:lb}
\end{align}
The error bound in the last line follows as there are clearly at most $|R|_1$ rows where
$B$ and $B^-$ differ, and since $u=\Theta(u_0)$, which holds because $z \in [z_-,z_0^5]$. 

Similarly, conditioned on the choice of $\phi_{I,C}(B)$,
\begin{align}
 \log\frac{\Prb(A_M=\phi_{I,C}(B)\mid\phi_{I,C}(B))}{\Prb(A_M=B^-)}
 &\le\sum_{i\colon\delta_i^\phi>0}\bigg(\delta_i^\phi \big(\xi(u)-\log q_{z-1} \big)+\frac{O(1)}{u}\bigg)\notag\\
 &=\Delta \big(\xi(u)-\log q_{z-1} \big)+\frac{O(|R|_1)}{u_0}.\label{eq:comp:pB:lb}
\end{align}
Again, the error bound follows as there are at most $|R|_1$ rows where $\phi_{I,C}(B)$
and $B^-$ differ.

We now aim to deduce corresponding lower bounds from \Th{Psi:ratio}. To do so, note first that if no row $i \in I$ of $B$ contains more than $u/12$ $1$s, then, by \Ob{weight},
$$w_i(B^-) + \delta_i = w_i(B) \le \frac{u}{2}.$$
Hence, by~\eqref{eq:comp:phi:basicid} and \Th{Psi:ratio}, we have
\begin{align}
 \log\frac{\Prb(A_M=B)}{\Prb(A_M=B^-)}
 &=\log\prod_{i\in M}\frac{\tpsi\big(x,q_{z-1}^{-w_i(B^-)-\delta_i},q_{z-1}\big)}
 {\tpsi\big(x,q_{z-1}^{-w_i(B^-)},q_{z-1}\big)}\notag\\
 &=\sum_{i\colon\delta_i>0}\bigg(\delta_i(\xi(u)-\log q_{z-1})
 +\frac{O\big( \delta_i^2+\delta_iw_i(B^-)+1 \big)}{u_0}\bigg)\notag\\
 &=\Delta(\xi(u)-\log q_{z-1})+\frac{O(\alpha_R(B)+|R|_2)}{u_0}.\label{eq:comp1:B:ub}
\end{align}
Note that the last equality follows by \Ob{weight}, since
\[
 \sum_{i\colon\delta_i>0}(\delta_i^2+\delta_iw_i(B^-)+1)\le 6^2|R|_2+\alpha_R(B)+|R|_1
\]
and $|R|_1\le |R|_2$. 

Proving a similar bound for $\phi_{I,C}(B)$ (and without assuming that the rows of $B$ have few 1s) is a little more complicated, since $w_i(B^-) + \delta^\phi_i \le u/2$ does not necessarily hold for all $\phi_{I,C}(B)$. To get around this problem, we use the following event (which depends on the choice of $\phi_{I,C}(B)$):
\[
 \cG_B:= \Big\{ w_i(B^-) \le u/3\text{ and }\delta_i^\phi\le 6
 \text{ for each $i\in M$ such that }\delta_i^\phi>0 \Big\}.
\]
The first step is to show that this event occurs for most choices of~$\phi_{I,C}(B)$.

\claim{Claim 1:} For every $B\in\cB_R$, we have
$$\Prb(\cG_B) = \exp\bigg( - \ds\frac{O(|R|_1)}{u_0} \bigg).$$

\begin{proof}[Proof of Claim 1.]
Let $M'=\{i\in M:w_i(B^-)\le u/3\}$, and note that if $\phi_{I,C}(B)$ places 1s only in rows of $M'$,
and places no more than a single 1 in each row, then $\cG_B$ holds. Indeed,
only those rows $i$ where a 1 is inserted have $\delta_i^\phi>0$, and if only a single 1 is inserted
in row $i$ then by \Ob{weight} we have $\delta_i^\phi\le6$.

We can construct $\phi_{I,C}(B)$ with the correct distribution by processing each one of the
$|R|_1$ 1 entries of $R=B[I\times C]$ in turn, removing it from $B$ and then adding a 1 to a uniformly
chosen non-zero entry of the same column of $B$, outside of the rows~$I$. The number of possible
choices in this process is clearly at most $|M|^{|R|_1}$. If we instead consider only
those choices where 1s are placed only in the rows $M'\setminus I$, and no two 1s are placed
in the same row, the number of choices will still be at least $(|M'|-|I|-4u_0-|R|_1)^{|R|_1}$.
Indeed, at each step we have $|M'\setminus I|$ choices of row, but must avoid the at most $4u_0$
original 1s of $B$ in that column, and the at most $|R|_1$ rows where we have already placed
a~1. Since the choice of $\phi_{I,C}(B)$ is uniform, we have
$$\Prb(\cG_B) \ge \bigg( \frac{|M'|-|I|-4u_0-|R|_1}{|M|} \bigg)^{|R|_1} \ge \bigg( 1 - \frac{O(1)}{u_0} \bigg)^{|R|_1} = e^{-O(|R|_1/u_0)},$$
since $|I| + 4u_0 + |R|_1 = O(|M|/u_0)$, by our assumptions and using~\eqref{eq:ulogu} and~\eqref{eq:lowerbound:m:using:M}, and since
$$|M \setminus M'| \le \frac{3}{u} \sum_{i \in M \setminus M'} w_i(B^-) \le \frac{3}{u} \sum_{i \in M} w_i(B) = \frac{O(|M|)}{u_0}$$
by \eqref{eq:sumalpha} and since $u = \Theta(u_0)$.
\end{proof}

We are now ready to prove our final lower bound; the lemma follows easily from the following claim, together with~\eqref{eq:comp:B:lb},~\eqref{eq:comp:pB:lb} and~\eqref{eq:comp1:B:ub}.

\claim{Claim 2:} For every $B\in\cB_R$, we have
\[
 \log\frac{\Prb\big(A_M=\phi_{I,C}(B)\big)}{\Prb\big(A_M=B^-\big)}
 \ge\Delta \big( \xi(u)-\log q_{z-1} \big)+\frac{O(|R|_1)}{u_0}.
\]
\begin{proof}[Proof of Claim 2.]
Note first that
$$\frac{\Prb\big(A_M=\phi_{I,C}(B)\big)}{\Prb\big(A_M=B^-\big)} \ge \Prb\big(\cG_B\big)\,\E\bigg[\frac{\Prb\big(A_M=\phi_{I,C}(B)\mid\phi_{I,C}(B)\big)}{\Prb\big(A_M=B^-\big)}\Bmid\cG_B\bigg],$$
and observe that if $\cG_B$ holds, then $w_i(B^-) + \delta^\phi_i \le u/3 + 6 \le u/2$. Hence, by \Th{Psi:ratio}, it follows (cf.~\eqref{eq:comp1:B:ub}) that
\begin{align*}
 \log\frac{\Prb\big(A_M=\phi_{I,C}(B)\mid\phi_{I,C}(B)\big)}{\Prb\big(A_M=B^-\big)}
 &=\log\prod_{i\in M}\frac{\tpsi\big(x,q_{z-1}^{-w_i(B^-)-\delta_i^\phi},q_{z-1}\big)}
 {\tpsi\big(x,q_{z-1}^{-w_i(B^-)},q_{z-1}\big)}\\
 &=\sum_{i\colon\delta_i^\phi>0}\bigg(\delta_i^\phi(\xi(u)-\log q_{z-1})+\frac{O(w_i(B^-)+1)}{u_0}\bigg)\\
 &=\Delta \big( \xi(u)-\log q_{z-1} \big)+\frac{O(\alpha_\phi+|R|_1)}{u_0},
\end{align*}
whenever $\cG_B$ holds, where
\[
 \alpha_\phi:=\sum_{\delta_i^\phi>0}w_i(B^-).
\]
Thus, by the convexity of the exponential function,
\begin{align}
 \frac{\Prb\big(A_M=\phi_{I,C}(B)\big)}{\Prb\big(A_M=B^-\big)} & \ge \Prb\big(\cG_B\big)\E\bigg[\exp\bigg(\Delta(\xi(u)-\log q_{z-1})+\frac{O(\alpha_\phi+|R|_1)}{u_0}\bigg)\Bmid\cG_B\bigg] \notag\\
 & \ge \Prb\big(\cG_B\big) \exp\bigg(\E\bigg[\Delta(\xi(u)-\log q_{z-1})+\frac{O(\alpha_\phi+|R|_1)}{u_0}\Bmid\cG_B\bigg]\bigg).\label{eq:newclaimtwo:convexity}
\end{align}

We claim that
\begin{equation}\label{eq:old:claimtwo}
\E\big[ \alpha_\phi \mid \cG_B\big] = O(|R|_1).
\end{equation}
Indeed, as in the proof of Claim 1, the probability that a given 1 entry in $R$
is moved to row $i$ is at most $\big( |M'| - |I| - 4u_0 - |R|_1 \big)^{-1} = O(1/|M|)$ for each $i \in M$. Thus the probability
that $\delta_i^\phi > 0$ (i.e., that row $i$ receives some 1) is at most $O(|R|_1/|M|)$, and so
$$\E\big[\alpha_\phi\mid\cG\big] = \frac{O(|R|_1)}{|M|} \sum_{i\in M}w_i(B^-) = O(|R|_1),$$
by~\eqref{eq:sumalpha}, as claimed.

Hence, combining~\eqref{eq:newclaimtwo:convexity} and~\eqref{eq:old:claimtwo}, and using Claim 1, we obtain
$$\frac{\Prb\big(A_M=\phi_{I,C}(B)\big)}{\Prb\big(A_M=B^-\big)} \ge \exp\bigg(\Delta \big(\xi(u)-\log q_{z-1}\big)+\frac{O(|R|_1)}{u_0}\bigg),$$
as required.
\end{proof}

To complete the proof, simply observe that combining Claim 2 with~\eqref{eq:comp:B:lb} gives~\eqref{eq:comp:phi:ineq}, and that combining Claim 2 with~\eqref{eq:comp:pB:lb} and~\eqref{eq:comp1:B:ub} gives~\eqref{eq:comp:phi:eq}.
\end{proof}

We can now easily deduce the following lemma.

\begin{lemma}\label{lem:compare:ineq}
 Let\/ $z\in [z_-,z_0^5]$ and let\/ $\cE\in\cF_z^+$ be an event of the form~\eqref{def:eventE}
 such that\/ $\cK(z)$ holds and\/ $d(z)\le 4u_0$. Let\/ $I\subseteq M$
 and\/ $C\subseteq [z,z_0^5]$ be such that\/ $|I|=e^{O(u_0)}$, and let\/ $R\in\cR_\cE(I,C)$.
 If\/ $|R|_1=O(|M|/u_0)$, then
 \begin{equation}\label{eq:lem:compare:ineq}
  \frac{\Prb\big(A[I\times C]=R\mid\cE\big)}{\Prb\big(A[I\times C]=0\mid\cE\big)}
  \le\exp\bigg(\frac{O(|R|_1)}{u_0}\bigg)
  \frac{\Prb\big(\tA_\cE[I\times C]=R\big)}{\Prb\big(\tA_\cE[I\times C]=0\big)}.
 \end{equation}
\end{lemma}
\begin{proof}
Write $\E_{B\in\cB_R}$ for the expectation over a uniform random choice of $B\in\cB_R$,
and similarly for $\E_{B\in\cB_0}$. Note that, by Observation~\ref{obs:condce},
$$ \Prb\big(A[I\times C] = R\mid\cE\big) = \Prb\big(A_M\in\cB_R\mid A_M\in\cB\big)
= \frac{|\cB_R|\,\E_{B\in\cB_R}\Prb\big(A_M=B\big)}{\Prb\big(A_M\in\cB\big)},$$
and that by \Lm{comp:phi} and Observation~\ref{obs:uniform:B},
\begin{align*}
\E_{B\in\cB_R}\Prb\big(A_M=B\big)
& \le \exp\bigg(\frac{O\big(|R|_1\big)}{u_0}\bigg) \E_{B\in\cB_R} \Prb\big(A_M=\phi_{I,C}(B)) \\
& = \exp\bigg(\frac{O\big(|R|_1\big)}{u_0}\bigg) \E_{B\in\cB_0}\Prb\big(A_M=B).
\end{align*}
Thus, using Observation~\ref{obs:condce} again, it follows that
$$\Prb\big(A[I\times C]=R\mid\cE\big) \le \exp\bigg(\frac{O\big(|R|_1\big)}{u_0}\bigg) \frac{|\cB_R|}{|\cB_0|} \Prb\big(A[I\times C]=0\mid\cE).$$
Now, since $A_\cE$ is distributed uniformly on $\cB$, we have
\[
 \frac{\Prb\big(A_\cE[I\times C]=R\big)}{\Prb\big(A_\cE[I\times C]=0\big)} = \frac{|\cB_R|}{|\cB_0|},
\]
and so the lemma follows.
\end{proof}

Our next task, which will be somewhat harder, is to prove an almost-matching lower bound when the row sums of $R$ are not too large. In order to do so we will use the following simple observation, which will also be useful in Section~\ref{sec:branching}.

\begin{obs}\label{obs:mgf}
 Let\/ $X_1,\dots,X_n$ be independent Bernoulli random variables, and let\/ $X=\sum_{i=1}^n X_i$.
 Then for any\/ $\lambda>0$,
 \[
  \E\big[e^{\lambda X}\big]\le\exp\big((e^\lambda-1)\E[X]\big).
 \]
\end{obs}
\begin{proof}
We have, for each Bernoulli random variable $X_i$,
\[
 \E\big[e^{\lambda X_i}\big]=1+(e^\lambda-1)\Prb(X_i=1)\le\exp\big((e^\lambda-1)\E[X_i]\big).
\]
Thus by independence of the $X_i$,
\[
 \E\big[e^{\lambda X}\big]=\prod_{i=1}^n\E\big[e^{\lambda X_i}\big]
 \le\prod_{i=1}^n\exp\big((e^\lambda-1)\E[X_i]\big)
 =\exp\big((e^\lambda-1)\E[X]\big),
\]
as required.
\end{proof}

The following lemma provides us with the lower bound on $\Prb\big(A[I\times C]=R\mid\cE\big)$ that we require in order to prove the second statement in \Th{compare}.

\begin{lemma}\label{lem:compare:eq}
 Under the same assumptions as \Lm{compare:ineq}, but with the extra condition
 that no row sum of\/ $R$ exceeds\/ $u/24$, where $q_{z-1} = x^{1/u}$, we have
 \begin{equation}\label{eq:lem:compare:eq}
  \frac{\Prb\big(A[I\times C]=R\mid\cE\big)}{\Prb\big(A[I\times C]=0\mid\cE\big)}
  =\exp\bigg(\frac{O\big(|R|_2\big)}{u_0}\bigg)
  \frac{\Prb\big(\tA_\cE[I\times C]=R\big)}{\Prb(\tA_\cE[I\times C]=0\big)}
 \end{equation}
\end{lemma}

To prove Lemma~\ref{lem:compare:eq} we would like to repeat the calculation in the
proof of \Lm{compare:ineq}, except using the second statement in \Lm{comp:phi}.
However, there is a problem: we have no control over the entries
of $B^-$ in the rows of~$I$. Indeed, if $B^-$ contains too many 1s in some row $i\in I$
that is non-zero in~$R$, then we will be unable to find any non-trivial lower bound on
$\Prb(A_M=B)$. Fortunately however, very few matrices $B$ have this property, so we
can simply use the trivial lower bound (i.e., zero) for those matrices. The main
challenge is to show that we also do not have a large contribution to
$\Prb(A[I\times C]=0\mid\cE)$ in this case.

\begin{proof}[Proof of \Lm{compare:eq}]
We are only required to prove the lower bound, since the upper bound follows from Lemma~\ref{lem:compare:ineq}. Define
\[
 \cB' := \Big\{ B\in\cB:B^-\text{ has no more than $u/24$ 1s in row $i$ for all }i\in I \Big\},
\]
and set $\cB'_R=\cB'\cap\cB_R$ and $\cB'_0=\cB'\cap\cB_0$. Note that if $B\in\cB'_R$, then, since no row of $R$
contains more than $u/24$~1s, we have at most $u/12$ 1s in each row $i\in I$ of $B$.
Thus we can apply \Lm{comp:phi} to deduce that
\begin{equation}\label{eq:compare:lba0}
 \Prb\big(A_M=B\big) = \exp\bigg(\frac{O(\alpha_R(B)+|R|_2)}{u_0}\bigg)
 \Prb\big(A_M=\phi_{I,C}(B)\big)
\end{equation}
for every $B\in\cB'_R$. Note that $\alpha_R(B)=\alpha_R(\phi_{I,C}(B))$, since $\delta_i$ only depends on $R$, and $w_i(B^-) = w_i(\phi_{I,C}(B)^-)$ for every $i \in I$, and observe that, as in \Ob{uniform:B}, if $B$ is chosen uniformly from $\cB'_R$ then $\phi_{I,C}(B)$ is uniform on $\cB'_0$, since the extra condition that $B\in\cB'$ does not depend on the columns $C$ of the matrix~$B$. Thus, taking the expectation of \eqref{eq:compare:lba0} over a uniform choice of $B\in\cB'_R$, we obtain
\begin{align}
\frac{\Prb\big(A_M\in\cB'_R\big)}{ |\cB'_R|}
 &=\E_{B\in\cB'_R} \bigg[ \exp\bigg(\frac{O\big( \alpha_R(\phi_{I,C}(B)) +|R|_2 \big)}{u_0}\bigg)\Prb\big(A_M=\phi_{I,C}(B)\big) \bigg] \notag\\
 &=\E_{B\in\cB'_0} \bigg[ \exp\bigg(\frac{O\big( \alpha_R(B)+|R|_2 \big)}{u_0}\bigg)\Prb\big(A_M=B\big) \bigg] \notag\\
 &= \frac{1}{|\cB'_0|} \E\bigg[\exp\bigg(\frac{O\big( \alpha_R(A_M)+|R|_2 \big)}{u_0}\bigg)\Bmid A_M\in\cB'_0\bigg]\,\Prb\big(A_M\in\cB'_0\big). \label{eq:compare:eq:takingexp}
\end{align}
The following claim will allow us to bound the right-hand side of~\eqref{eq:compare:eq:takingexp}.

\claim{Claim 1:}
$$\E\big[\alpha_R(A_M)\mid A_M\in\cB'_0\big]=O(|R|_1).$$
\begin{proof}[Proof of Claim 1]
Since $\delta_i$ depends only on the fixed matrix $R$, the definition \eqref{eq:alphaR}
of $\alpha_R$ and \Ob{weight} imply that
\begin{align}
 \E\big[\alpha_R(A_M)\mid A_M\in\cB'_0\big]
 &=\sum_{i\in I}\delta_i\,\E\big[w_i(A_M^-)\mid A_M\in\cB'_0\big]\notag\\
 &\le 6\sum_{i\in I}\delta_i\sum_{j\in[z,\pi(x)]\setminus C}\Prb\big(A_{ij}=1\mid A_M\in\cB'_0\big).
 \label{eq:expwi}
\end{align}
We claim that
\begin{equation}\label{eq:expaij}
 \Prb\big(A_{ij}=1\mid A_M\in\cB'_0\big)\le e^{O(1/u_0)}\frac{d_j}{|M|}
\end{equation}
for every $i\in I$ and $j\in [z,z_0^5]\setminus C$. Indeed,
if $B$ is chosen uniformly at random from the set $\cC_1 := \{ B\in\cB'_0 : B_{ij}=1\}$,
then $\phi_{\{i\},\{j\}}(B)$ is uniform on $\cC_0 := \{ B\in\cB'_0 : B_{ij}=0\}$.
\Lm{comp:phi} then implies (cf.~\eqref{eq:compare:eq:takingexp}) that
\begin{align*}
 \Prb\big(A_M\in\cC_1\big) & = |\cC_1| \cdot \E_{B \in \cC_1} \Prb\big(A_M=B\big) \le |\cC_1| e^{O(1/u_0)} \E_{B \in \cC_1} \Prb\big(A_M = \phi_{I,C}(B) \big)\\
& = |\cC_1| e^{O(1/u_0)} \E_{B \in \cC_0} \Prb\big(A_M=B\big) = \frac{|\cC_1|}{|\cC_0|}e^{O(1/u_0)}\Prb\big(A_M\in\cC_0\big).
\end{align*}
Now $|\cC_1|/|\cC_0|=d_j/(|M|-d_j)$, so
\[
 \frac{\Prb\big(A_{ij}=1\mid A_M\in\cB'_0\big)}{\Prb\big(A_{ij}=0\mid A_M\in\cB'_0\big)}
 =\frac{\Prb\big(A_M\in\cC_1\big)}{\Prb\big(A_M\in\cC_0\big)}
 \le e^{O(1/u_0)}\frac{d_j}{|M|-d_j},
\]
which implies~\eqref{eq:expaij}, because $d_j \le 4u_0$ (since $\cK(z)$ holds) and $u_0^2 = o(|M|)$, by~\eqref{eq:lowerbound:m:using:M}.

Now, combining \eqref{eq:expwi} with \eqref{eq:expaij} gives
\begin{align*}
 \E\big[\alpha_R(A_M)\mid A_M\in\cB'_0\big]
 &\le O(1)\sum_{i\in I}\delta_i\sum_{j=z}^{\pi(x)}\frac{d_j}{|M|}=O(|R|_1),
\end{align*}
where the final step follows by~\eqref{eq:sumdj}, and since $\sum_{i \in I} \delta_i = O(|R|_1)$.
\end{proof}

By Claim~1 and the convexity of the exponential function, we obtain
\[
 \E\bigg[\exp\bigg(\frac{O\big(\alpha_R(A_M)\big)}{u_0}\bigg)
 \bmid A_M\in\cB'_0\bigg]\ge\exp\bigg(\frac{O(|R|_1)}{u_0}\bigg).
\]
Now, combining this with~\eqref{eq:compare:eq:takingexp}, and noting that $|R|_1 \le |R|_2$, gives
\begin{equation}\label{eq:compare:eq:almostdone}
 \Prb\big(A_M\in\cB_R\big)\ge\Prb\big(A_M\in\cB'_R\big)
 \ge e^{O(|R|_2/u_0)} \frac{|\cB'_R|}{|\cB'_0|}\Prb\big(A_M\in\cB'_0\big)
\end{equation}
Note also that
\begin{equation}\label{eq:compare:ratio:bRbzero}
 \frac{|\cB'_R|}{|\cB'_0|}=\frac{|\cB_R|}{|\cB_0|}
 =\frac{\Prb\big(\tA_\cE[I\times C]=R\big)}{\Prb\big(\tA_\cE[I\times C]=0\big)},
\end{equation}
since the condition that $B\in\cB$ lies in $\cB'$ does not affect the columns in~$C$.
The following claim will therefore be sufficient to complete the proof of the lemma.

\claim{Claim 2:}
$$\Prb\big(A_M \not\in\cB'_0\mid A_M\in\cB_0\big) = O(e^{-u}).$$

\begin{proof}[Proof of Claim 2.]
Recall that if $B \in \cB_0\setminus\cB'_0$, then $B\big[\{i\}\times([z,z_0^5]\setminus C)\big]$
contains at least $u/24$ 1s for some $i\in I$. Our strategy will be to use \Lm{compare:ineq}
to bound the probability that this property is satisfied by $A_M$ in terms of the probability that
it is satisfied by $\tA_\cE$, and then use the independence of the columns of $\tA_\cE$ to
deduce the desired bound.

In order to apply \Lm{compare:ineq}, we will first have to cover the event
$A_M  \in \cB_0\setminus\cB'_0$ with a suitable collection of events. To do so, let $\cC$ be
the collection of subsets $C'\subseteq [z,z_0^5]\setminus C$ of size exactly $\lceil u/24\rceil$,
and let $\cR(C')$ be the set of $R'\in\cR(I,C')$ with some row $i\in I$ of $R'$ consisting entirely of~1s.
Now if $A_M\in\cB_0\setminus\cB'_0$, then there must exist some $C'\in\cC$ and $R'\in\cR(C')$
such that $A[I\times C']=R'$. Thus, by the union bound, we have
\begin{equation}\label{eq:compare:claim2union}
\Prb\big(A_M \not\in\cB'_0\mid A_M\in\cB_0\big) \le \sum_{C'\in\cC}\sum_{R'\in\cR(C')}\Prb\big(A[I\times C']=R'\mid A_M\in\cB_0\big).
\end{equation}
We claim that
\begin{equation}\label{eq:compare:claim2obsap}
\Prb\big(A[I\times C'] = R'\mid A_M\in\cB_0\big)
= \frac{\Prb\big(A[I\times C']=R',\,A[I\times C]=0\mid\cE\big)}{\Prb\big(A[I\times C]=0\mid\cE\big)}
 \end{equation}
for every $C' \in \cC$ and $R' \in \cR(C')$. Indeed, by \Ob{condce} both sides are equal to
$$\frac{\Prb\big(A[I\times C']=R', A[I\times C]
= 0 \mid A_M\in\cB\big)}{\Prb\big(A[I\times C] = 0\mid A_M \in\cB \big)}.$$
We are now ready to use \Lm{compare:ineq} to show that
\begin{equation}\label{eq:compare:claim2lemap}
\frac{\Prb\big(A[I\times C']=R',\,A[I\times C]=0\mid\cE\big)}{\Prb\big(A[I\times (C'\cup C)]=0\mid\cE\big)}
\le e^{O(u_0)}\frac{\Prb\big(\tA_\cE[I\times C']=R', \, \tA_\cE[I\times C]=0\big)}{\Prb\big(\tA_\cE[I\times (C'\cup C)]=0\big)}.
\end{equation}
Indeed, this follows by applying \Lm{compare:ineq} with $R$ an $I \times ( C \cup C' )$ matrix
that is $R'$ on $I \times C'$ and zero on $I\times C$, noting that $|R'|_1\le 4u_0|C'|\le u_0^2$.

Note that the right-hand side of~\eqref{eq:compare:claim2obsap} is at most the left-hand side
of~\eqref{eq:compare:claim2lemap}, since we have added the condition that $A[I\times C']=0$
in the denominator. Observe also that
\begin{equation}\label{eq:compare:claim2indepcols}
\frac{\Prb\big(\tA_\cE[I\times C']=R', \, \tA_\cE[I\times C]=0\big)}{\Prb\big(\tA_\cE[I\times (C'\cup C)]=0\big)}
= \frac{\Prb\big(\tA_\cE[I\times C']=R'\big)}{\Prb\big(\tA_\cE[I\times C']=0\big)},
\end{equation}
since the columns of $\tA_\cE$ are independent, and also that
\begin{align}
 \Prb\big(\tA_\cE[I\times C']=0\big)
 &=\prod_{j\in C'}\binom{|M|-|I|}{d_j}\binom{|M|}{d_j}^{-1}
 =\prod_{j\in C'}\prod_{t=0}^{d_j-1}\frac{|M|-|I|-t}{|M|-t} \notag\\
 &=\prod_{j\in C'}\exp\bigg(-\frac{O(|I|d_j)}{|M|}\bigg)
 \ge\exp\bigg(-\frac{O(|I|u_0^2)}{|M|}\bigg) = 1 - o(1),\label{eq:compare:claim2indepzero}
\end{align}
since $\sum_{j\in C'} d_j \le 4u_0 |C'| \le u_0^2$, and $|I| = e^{O(u_0)} = o(|M| / u_0^2)$, by~\eqref{eq:ulogu} and~\eqref{eq:lowerbound:m:using:M}.

Combining~\eqref{eq:compare:claim2union},~\eqref{eq:compare:claim2obsap},~\eqref{eq:compare:claim2lemap},~\eqref{eq:compare:claim2indepcols} and~\eqref{eq:compare:claim2indepzero}, we obtain
\begin{align}
\Prb\big(A_M \not\in\cB'_0\mid A_M\in\cB_0\big)
& \le e^{O(u_0)} \sum_{C'\in\cC}\sum_{R'\in\cR(C')} \Prb\big(\tA_\cE[I\times C']=R'\big) \notag\\
& \le e^{O(u_0)} \sum_{C'\in\cC}\sum_{i\in I} \Prb\big(\tA_\cE[\{i\}\times C']=1\big),\label{eq:compare:claim2main}
\end{align}
where $1$ indicates the all 1s vector, since the events $\Prb\big(\tA_\cE[I\times C']=R'\big)$ are disjoint.

Now, let $X = \sum_{j=z}^{z_0^5} X_j$, where $X_z,\dots,X_{z_0^5}$ are independent Bernoulli random variables with $\Prb( X_j = 1) = d_j / |M|$. We claim that
\begin{equation}\label{eq:compare:claim2X}
\sum_{C'\in\cC} \Prb\big(\tA_\cE[\{i\}\times C']=1\big) = e^{O(u_0)} \Prb\big(X=\lceil u/24\rceil\big)
\end{equation}
for every $i \in I$. Indeed, simply note that
\begin{align*}
\Prb\big(X=\lceil u/24\rceil\big)
& = \sum_{C' \in \cC} \prod_{j \in C'} \frac{d_j}{|M|} \prod_{j \in [z,z_0^5] \setminus C'} \bigg( 1 - \frac{d_j}{|M|} \bigg)\\
& = e^{O(u_0)} \sum_{C' \in \cC} \prod_{j \in C'} \frac{d_j}{|M|} = e^{O(u_0)} \sum_{C'\in\cC} \Prb\big(\tA_\cE[\{i\}\times C']=1\big)
\end{align*}
for each $i \in I$, since the columns of $\tA_\cE$ are independent, and using~\eqref{eq:sumdj}.

Recalling that $|I| = e^{O(u_0)}$ and $u = \Theta(u_0)$, it follows from~\eqref{eq:compare:claim2main} and~\eqref{eq:compare:claim2X} that
$$\Prb\big(A_M \not\in\cB'_0\mid A_M\in\cB_0\big) \le e^{\lambda u} \cdot \Prb\big(X=\lceil u/24\rceil\big)$$
for some constant $\lambda > 0$. Noting that
\[
\E[X] = \sum_{j = z}^{z_0^5} \frac{d_j}{|M|} = O(1)
\]
by~\eqref{eq:sumdj}, and that $e^{24(\lambda+1)X} \ge e^{(\lambda+1)u}$ for $X=\lceil u/24\rceil$,
it follows by \Ob{mgf} that
\begin{align*}
\Prb\big(A_M \not\in\cB'_0\mid A_M\in\cB_0\big)
& \le \E\big[e^{24(\lambda+1)X}\big]e^{-u}\\
& \le \exp\Big( \big( e^{24(\lambda+1)} - 1 \big) \E[X] \Big) e^{-u}
= O\big(e^{-u}\big),
\end{align*}
as required.
\end{proof}

To complete the proof, we simply combine Claim~2 with~\eqref{eq:compare:eq:almostdone} and~\eqref{eq:compare:ratio:bRbzero}. This gives
\begin{align*}
\Prb\big(A_M\in\cB_R\big) \ge e^{O(|R|_2/u_0)} \frac{\Prb\big(\tA_\cE[I\times C]=R\big)}{\Prb\big(\tA_\cE[I\times C]=0\big)} \Prb\big(A_M \in \cB_0\big).
\end{align*}
But, by \Ob{condce}, we have
$$\frac{\Prb\big(A[I\times C]=R\mid\cE\big)}{\Prb\big(A[I\times C]=0\mid\cE\big)} = \frac{\Prb\big( A_M \in \cB_R \mid A_M \in \cB \big)}{\Prb\big( A_M \in \cB_0 \mid A_M \in \cB \big)} = \frac{\Prb\big( A_M \in \cB_R \big)}{\Prb\big( A_M \in \cB_0 \big)},$$
and so the required lower bound follows.
\end{proof}

To prove \Th{compare}, it just remains to estimate $\Prb(A[I\times C]=0\mid\cE)$.
To do so we prove the following lemma, which follows from $|C|$ applications of \Lm{compare:eq}.

\begin{lemma}\label{lem:compare:zero}
 Under the same assumptions as in \Lm{compare:ineq},
 \begin{equation}\label{eq:cmpz}
  \Prb\big(A[I\times C]=0\mid\cE\big)=\exp\big(O(|I|/u_0)\big)\Prb\big(\tA_\cE[I\times C]=0\big).
 \end{equation}
\end{lemma}

\begin{proof}
Enumerate the elements of $C$ as $\{j_1,\dots,j_t\}$ and write $C_i=\{j_i,j_{i+1},\dots,j_t\}$.
Let $p_i:=\Prb(A[I\times C_i]=0\mid\cE)$ and $\tp_i:=\Prb(\tA_\cE[I\times C_i]=0)$, and
observe that
\begin{equation}\label{eq:compare:breakintosteps}
\log \Prb(A[I\times C]=0\mid\cE) = - \sum_{i=1}^t \log\frac{p_{i+1}}{p_i} = - \sum_{i=1}^t \log\bigg(1+\frac{p_{i+1}-p_i}{p_i}\bigg),
\end{equation}
and similarly for $\tA_\cE$ and $\tp_i$, since $p_{t+1}=\tp_{t+1}=1$. We will use \Lm{compare:eq} to show that
\begin{equation}\label{eq:ratiozero}
 \frac{p_{i+1}-p_i}{p_i}=e^{O(d_{j_i}/u_0)}\cdot\frac{\tp_{i+1}-\tp_i}{\tp_i}
\end{equation}
for each $i \in [t]$. To prove this, define a family
$$\cR(i) := \Big\{ R \in \cR_\cE(I,C_i) : R[ I \times \{j_i\} ] \ne 0 \textup{ and } R[ I \times C_{i+1} ] = 0 \Big\},$$
and observe that we can write
$$p_{i+1} - p_i = \sum_{R \in \cR(i)} \Prb\big( A[I\times C_i]=R \mid\cE \big),$$
and similarly for $\tp_{i+1}-\tp_i$. Note that each row sum of each $R \in \cR(i)$ is at most~1, and so $|R|_2 = |R|_1 \le d_{j_i} \le 4u_0 = o(|M|/u_0)$. Hence, by applying \Lm{compare:eq} (with $C=C_i$), we obtain
\[
\frac{\Prb\big(A[I\times C_i]=R\mid\cE\big)}{p_i}
= e^{O(d_{j_i}/u_0)} \frac{\Prb(\tA_\cE[I\times C_i]=R\big)}{\tp_i}
\]
for each $R \in \cR(i)$. It follows that
\begin{align*}
 \frac{p_{i+1}-p_i}{p_i}
 & = e^{O(d_{j_i}/u_0)}  \sum_{R \in \cR(i)} \frac{\Prb(\tA_\cE[I\times C_i]=R\big)}{\tp_i} =e^{O(d_{j_i}/u_0)}\cdot\frac{\tp_{i+1}-\tp_i}{\tp_i},
 \end{align*}
as claimed.

Next, observe that $\tp_i / \tp_{i+1} = \Prb(\tA_\cE[I\times \{j_i\}]= 0 \big)$, since the columns of $\tA_\cE$ are independent, and so, recalling that $d_{j_i} |I| = e^{O(u_0)} = o(|M|)$, by~\eqref{eq:ulogu} and~\eqref{eq:lowerbound:m:using:M}, we have
\begin{equation}\label{eq:tp:est}
\frac{\tp_{i+1}-\tp_i}{\tp_i} = \binom{|M|}{d_{j_i}}\binom{|M| - |I|}{d_{j_i}}^{-1}-1 = \frac{O(d_{j_i}|I|)}{|M|} = o(1).
\end{equation}
Now $\log(1+e^a b)=\log(1+b+O(ab))=\log(1+b)+O(ab)$ for all
$a=O(1)$ and $b=o(1)$. Applying this with $b=(\tp_{i+1}-\tp_i)/\tp_i$, and using~\eqref{eq:ratiozero} and~\eqref{eq:tp:est}, gives
\begin{align*}
\log\bigg(1+\frac{p_{i+1}-p_i}{p_i}\bigg)
 &=\log\bigg(1+e^{O(d_{j_i}/u)}\cdot\frac{\tp_{i+1}-\tp_i}{\tp_i}\bigg)\\
 &=\log\bigg(1+\frac{\tp_{i+1}-\tp_i}{\tp_i}\bigg)+O\bigg(\frac{d_{j_i}}{u_0}\cdot \frac{\tp_{i+1}-\tp_i}{\tp_i}\bigg)\\
 &=\log\frac{\tp_{i+1}}{\tp_i}+O\bigg(\frac{d_{j_i}^2 |I|}{u_0 |M|}\bigg).
\end{align*}


Finally, recall that $\sum_{i=1}^t d_{j_i}^2\le d(z)^2+\sum_{k\ge 2}k^2s_k(z)=O(|M|)$, since $\cK(z)$ holds and $d(z)\le 4u_0$. Thus, using~\eqref{eq:compare:breakintosteps}, we obtain
\begin{align*}
\log\Prb(A[I\times C]=0\mid\cE)
 &=-\sum_{i=1}^t\log\frac{\tp_{i+1}}{\tp_i}+O\bigg(\frac{|I|}{u_0 |M|}\sum_{i=1}^t d_{j_i}^2\bigg)\\
 &=\log\Prb\big( \tA_\cE[I\times C]=0 \big) + \frac{O(|I|)}{u_0},
\end{align*}
as required.
\end{proof}

\begin{proof}[Proof of \Th{compare}.]
The reduction to the case $z\in[z_-,z_0^5]$, $C\subseteq [z,z_0^5]$ and
$R\in\cR_\cE(I,C)$ was given after the statement of the theorem,
so we may assume these conditions hold.
Multiplying \eqref{eq:cmpz}
by \eqref{eq:lem:compare:ineq} gives \eqref{eq:compare:ineq}. Similarly, multiplying~\eqref{eq:cmpz} 
by~\eqref{eq:lem:compare:eq}, and noting that $u > u_0/6$ for all $z \in [z_-,z_0^5]$, gives~\eqref{eq:compare:eq}.
\end{proof}

\section{The Exploration Process}\label{sec:branching}

When there is exactly one active non-zero entry in column~$z$ (i.e., when $d(z)=1$), a chain
reaction is set off that reduces the number of active rows, and can significantly alter the
hypergraph $\cS_A(z)$. In order to control the evolution of the variables $s_k(z)$ (and hence
$m(z)$) we shall need a very precise understanding of this \emph{deterministic} process. In this
section we shall use techniques from the theory of branching processes to control the expected
change of various key parameters of the hypergraph $\cS_A(z)$, where we average over the possible
matrices $A$ that are consistent with the information observed so far in the filtration, see
\Al{algorithm}. Importantly, we shall also obtain strong bounds on the probability of
large deviations.

We begin by defining the various parameters that we shall need to control.

\begin{defn}\label{def:DandDelta}
 For each $z\in [\pi(x)]$, and each $k\ge 2$, define the following random variables.
 \begin{itemize}
  \item[$(i)$] $D(z):=m(z)-m(z-1)$, the number of rows removed in step $z$.
  \item[$(ii)$] $\Delta_k(z):=|S_k(z)\setminus S_k(z-1)|$, the number of edges of size $k$ that
   contain a vertex removed in step~$z$.
  \item[$(iii)$] $R_1(z):=d(z)+\sum_{k\ge2}k(s_k(z)-s_k(z-1))$, the number of 1s removed from the
   matrix in step~$z$ (including those in column $z$) if $d(z) = 1$.
  \item[$(iv)$] $\Delta'(z)$, the number of edges of size at least three that have at least
   two vertices removed in step~$z$.
 \end{itemize}
\end{defn}

The following theorem will play a key role in the proof of \Th{track:s}.

\begin{thm}\label{thm:branching}
 Suppose that\/ $z\in[z_-,\pi(x)]$, $\cK(z)$ holds, and\/ $2s_2(z)\le (1-\eps_1)m(z)$. Then
 \begin{itemize}
  \item[$(a)$] $\ds\E\big[D(z)\mid\cF_z,\,d(z)=1\big]
   =\bigg(1-\frac{2s_2(z)}{m(z)}+o(1)\bigg)^{-1}$,
  \item[$(b)$] $\ds\E\big[\Delta_k(z)\mid\cF_z,\,d(z)=1\big]
   =\bigg(1-\frac{2s_2(z)}{m(z)}+o(1)\bigg)^{-1}\frac{ks_k(z)}{m(z)}+O\big(m(z)^{-1/3}\big)$
   for all\/ $k\ge 2$,
  \item[$(c)$] $\E\big[D(z)^2\mid\cF_z,\,d(z)=1\big]=O(1)$,
  \item[$(d)$] $\E\big[\Delta'(z)\mid\cF_z,\,d(z)=1\big]=O\big(m(z)^{-1/3}\big)$,
 \end{itemize}
 where the bounds implicit in the\/ $o(\cdot)$ and\/ $O(\cdot)$ notation are uniform in\/ $k$ and\/~$z$.
 Moreover, 
 \begin{equation}\label{eq:Dplus:tailevent}
  \Prb\big(R_1(z)\ge u_0^2\mid\cF_z,\,d(z)=1\big)\le z_0^{-20}.
 \end{equation}
\end{thm}

The idea is to prove a corresponding theorem in the simpler (independent) model~$\tA_\cE$,
and then use \Th{compare} to deduce the statement in $\cS_A(z)$. We now recall this
model and define some additional notation.

\begin{defn}\label{def:tilde:variables}
Fix $z\in[z_-,\pi(x)]$ and an event $\cE\in\cF^+_z$ of the form \eqref{def:eventE}
such that $\cK(z)$ holds and
$d(z)=1$, and define $\tcS_\cE$ to be the random hypergraph with vertex set $M$ and edge set
\[
 E(\tcS_\cE) = \big\{ \te_j : j\in [z+1,\pi(x)],\,d_j>0 \big\},
\]
 where $\te_j$ is a random subset of $M$ chosen uniformly
 over all $\binom{|M|}{d_j}$ choices of subsets of $M$ of size~$d_j$, independently for each~$j$.
 For $k\ge2$, let
 \[
 \tS_k= \big\{ j\in[z+1,\pi(x)]:d_j=k \big\}
 \]
 denote the set of columns corresponding to the $k$-edges in~$\tcS_\cE$, and let $\tS=\bigcup_{k\ge2}\tS_k$.
 We also define $\te_z=\{v\}$ where $v$ is chosen uniformly at random from~$M$, independently of $\te_j$,
 $j>z$.

 We define a deterministic process on the random hypergraph $\tcS_\cE$ by `infecting' vertex $v$
 at time zero, and then at each subsequent step infecting any vertex that is the last non-infected
 vertex in an edge of $\tS$. To be precise, we set $\tD_0=\{v\}$ and, for each $t\ge 1$,
 \[
  \tD_t:=\tD_{t-1}\cup\big\{w\in M:\text{there exists }j\in \tS\text{ with }
  w\in\te_j\subseteq\tD_{t-1}\cup\{w\}\big\}.
 \]
 We remark that such processes are usually referred to as `bootstrap percolation', and have been extensively studied in both deterministic and random settings, see e.g.~\cite{BBDM,JLTV}. 

 We now define $\tD$, $\tR_1$, $\tDelta_k$, and $\tDelta'$ to be the quantities corresponding to
 $D(z)$, $R_1(z)$, $\Delta_k(z)$, and $\Delta'(z)$ respectively. That is,
 \begin{itemize}
  \item[$(i)$] $\tD:=|\tD_\infty|$, where $\tD_\infty:=\bigcup_{t\ge 0}\tD_t$.
  \item[$(ii)$] $\tDelta_k:=\big|\big\{j\in\tS_k:\te_j\cap\tD_\infty\ne\emptyset\big\}\big|$
   for each $k\ge 2$.
  \item[$(iii)$] $\tR_1:=1+\sum_{j\in\tS}|\te_j\cap\tD_\infty|$. (Note that the extra 1 is for the single 1
   in column~$z$.)
  \item[$(iv)$] $\tDelta':=\big|\big\{j\in\bigcup_{k\ge3}\tS_k:
   |\te_j\cap\tD_\infty|\ge 2\big\}\big|$.
 \end{itemize}
  For $k\ge3$ we also define $\tDelta_k^{(1)}$ by
 \[
  \tDelta_k^{(1)}:=\big|\big\{j\in\tS_k:|\te_j\cap\tD_\infty|=1\big\}\big|.
 \]
 Note that for $k\ge3$, $\tDelta_k^{(1)}\le\tDelta_k\le\tDelta_k^{(1)}+\tDelta'$.
\end{defn}

The hypergraph $\tcS_\cE$ is the hypergraph corresponding to $\cS_A(z)$, except that we use the
matrix $\tA_\cE$ from \Sc{compare} in place of~$A$. In particular, if $\cE$ holds then
$\tS_k=S_k(z)$, and $M=M(z)$. We remark that, for emphasis, we shall put
tildes over all random variables that are functions of the random hypergraph~$\tcS_\cE$,
and write $s_k=|\tS_k|$ and $m=|M|$, so that if $\cE$ holds then $m=m(z)$ and $s_k=s_k(z)$.
We label the hyperedges of $\tcS_\cE$ by the column indices $j\in [z+1,\pi(x)]$; the
realizations of these edges as subsets of vertices are then given by the random variables
$\te_j=\{i\in M:(\tA_\cE)_{ij}=1\}$. However, as we analyse the progress of the `avalanche'
we shall reveal information about these random subsets only when necessary.
Although technically not part of the hypergraph~$\tcS_\cE$, we also define $\te_z$ as
the random 1-edge corresponding to column~$z$. We remark also that, since we will always assume 
that $\cK(z)$ (and hence $\cM(z)$) holds, it follows from~\eqref{eq:lowerbound:m:using:M} that $m \ge z_0^{1+o(1)}$. 

Our first main task will be to prove the following bound on the probability (in the independent
random hypergraph model) of large deviations of~$\tR_1$.

\begin{lemma}\label{lem:branching:main}
 Suppose that\/ $\cE$ is such that\/ $\cK(z)$ holds, $d(z)=1$, and\/
 $2s_2\le (1-\eps_1)m$. Then there exists a constant\/ $\lambda>0$, depending only
 on\/~$\eps_1$, such that
 \begin{equation}\label{eq:branching:main:a}
  \Prb\big(\tR_1\ge t\big)=\frac{O(s_2)}{m}e^{-\lambda t}
 \end{equation}
 and
  \begin{equation}\label{eq:branching:main:b}
  \E\Big[\tDelta_k^{(1)}\id_{\{t\le \tR_1<m^{1/2}\}}\Big]=\frac{O(ks_k)}{m}e^{-\lambda t}
 \end{equation}
 for all\/ $2\le t\le m^{1/2}$, uniformly in\/ $t$ and\/ $k\ge3$.
\end{lemma}

We shall next define precisely the process via which we reveal the set of vertices removed in the
independent model. For each integer $t\ge 0$, let $\tA_t$ and $\tV_t$ (the \emph{active} and
\emph{visited} vertices respectively) be subsets of $M$ given by the following algorithm. For
technical reasons, we shall need an upper bound on the number of vertices we visit during the
process.

\pagebreak 

\begin{alg}\label{alg:exploration}
 We start with $t:=1$, $\tA_0=\tV_0:=\{v\}$ and $\tE_0:=\{z\}$, where $v$ is the vertex
 corresponding to the unique 1 in column~$z$. Define $\tj(v)=z$ and repeat the following
 steps until $|\tA_t|=0$.
 \begin{itemize}
  \item[$1.$] Pick $u\in\tA_{t-1}$ with the smallest value of $\tj(u)$
   and list the elements of $M\setminus\tV_{t-1}$ (in increasing order, say) as $w_1,\dots,w_r$.
   Set $\tA^{(0)}:=\tA_{t-1}$, $\tV^{(0)}:=\tV_{t-1}$, $\tE^{(0)}:=\tE_{t-1}$ and $\ell:=0$.
  \item[$2.$] While $|\tV^{(\ell)}|<m^{1/2}$ and $\ell<r$, repeat the following steps.
  \smallskip
  \begin{itemize}
   \item[$(a)$] Set $\ell:=\ell+1$.
   \item[$(b)$] Let $\tj(w_\ell)$ be the smallest $j\in\tS$
    with $\{u,w_\ell\}\subseteq\te_j\subseteq\tV_{t-1}\cup\{w_\ell\}$,
    if such a $j$ exists. Set
    $\tV^{(\ell)}:=\tV^{(\ell-1)}\cup\{w_\ell\}$,
    $\tA^{(\ell)}:=\tA^{(\ell-1)}\cup\{w_\ell\}$, and
    $\tE^{(\ell)}:=\tE^{(\ell-1)}\cup\{\tj(w_\ell)\}$.
    If no such $j$ exists, then set $\tV^{(\ell)}:=\tV^{(\ell-1)}$,
    $\tA^{(\ell)}:=\tA^{(\ell-1)}$ and $\tE^{(\ell)}:=\tE^{(\ell-1)}$.
   \end{itemize}
  \item[$3.$] Set $\tA_t:=\tA^{(\ell)}\setminus\{u\}$, $\tV_t:=\tV^{(\ell)}$ and
   $\tE_t:=\tE^{(\ell)}$.
  \item[$4.$] If $|\tA_t|=0$ then set $\tV_\infty:=\tV_t$ and $\tE_\infty:=\tE_t$;
   otherwise set $t:=t+1$ and return to Step~1.
 \end{itemize}
\end{alg}

Let us begin by making a couple of simple but key observations about this algorithm.

\begin{obs}\label{obs:alg:basicproperties}
 $|\tA_t|=|\tV_t|-t$ for every\/ $0\le t\le |\tV_\infty|$, and\/
 $|\tV_\infty|=\min\big\{\tD,\lceil m^{1/2}\rceil\big\}$.
\end{obs}

\begin{proof}
The equation $|\tA_t|=|\tV_t|-t$ follows since we add the same elements to $\tA_t$ and~$\tV_t$,
but remove one element from $\tA_t$ at each time step. To see that $|\tV_\infty|\le\tD$, observe
that in fact we have $\tV_t\subseteq\tD_t$ for all $t\ge0$ by induction, as we only add
vertices to $\tV_t$ that are added to~$\tD_t$. Moreover, if $|\tV_t|<\lceil m^{1/2}\rceil$
for every $t\ge 0$, then the algorithm discovers all vertices that are included in~$\tD_\infty$,
so in that case $\tV_\infty=\tD_\infty$, and so $|\tV_\infty|=\tD$.
On the other hand, if $|\tV_t|=\lceil m^{1/2}\rceil$ for some $t$ then we visit no new
vertices after that point, and therefore $\tV_{t'}=\tV_t$ for all $t<t'\le |\tV_\infty|$,
and hence $|\tV_\infty|=\lceil m^{1/2}\rceil$, as claimed.
\end{proof}

We think of $(|\tA_t|)_{t\ge 0}$ as a random walk, and of $|\tV_\infty|$ as the hitting time of~0.
However, the steps of this random walk might be large, and are not independent. Therefore, in
order to control the walk we shall need to break the steps up into smaller pieces. Let us define
random variables $\tX_{t,w}\in\{0,1\}$ for each $1\le t\le |\tV_\infty|$ and
$w\in M\setminus\tV_{t-1}$ by setting
\[
 \tX_{t,w}=1\quad\Leftrightarrow\quad w\in\tV_t\setminus\tV_{t-1}.
\]
Abusing notation slightly, let us define a filtration
$\cF_z^+=\tcF_0\subseteq\tcF_1\subseteq \dots\subseteq\tcF_{|\tV_\infty|}=\tcF_\infty$ by
defining $\tcF_t$ to be the information observed (about the independent model) at the moment
$\tV_t$ is defined.\footnote{Note that after we have visited $m^{1/2}$ vertices, the algorithm does not observe any further new information, and so the $\sigma$-algebras of the filtration are all the same from that point on.} For each $1\le t\le |\tV_\infty|$, let us define a further filtration
\[
 \tcF_{t-1}=\tcF_{t,<w_1}\subseteq\tcF_{t,<w_2}\subseteq\dots\subseteq\tcF_{t,<w_r}\subseteq\tcF_t
\]
by defining $\tcF_{t,<w}$ to be the information observed just before we begin Step~$2(b)$ in the
round of \Al{exploration} in which we discover whether or not $w\in\tV_t\setminus\tV_{t-1}$.

\begin{rmk}\label{rem:filtration}
 Note that in Step~$2(b)$ of the algorithm we only need to observe whether or not the hyperedge
 $\te_j$ satisfies $\{u,w_\ell\}\subseteq\te_j\subseteq\tV_{t-1}\cup\{w_\ell\}$ in 
 turn\footnote{Since $\tj(w_\ell)$ is the smallest $j\in\tS$ with this property, we consider the elements of $\tS$ in increasing order.} for 
 each $j$ until we find one that does, or we have exhausted all~$j \in \tS$. Moreover, if we do find such
 an edge then we do not test this condition for larger~$j$. We emphasize that this is the
 only (new) information contained in $\tcF_{t,<w_{\ell+1}}$, and therefore, for edges
 $\te_j\in E(\tcS_\cE)$ that are not used in the process, we only have `negative' information
 (i.e., information of the form ``the hyperedge $\te_j$ does not satisfy
 $\{u,w_\ell\}\subseteq\te_j\subseteq\tV_{t-1}\cup\{w_\ell\}$''). This fact will play an
 important role in the proof, see Lemmas~\ref{lem:prob:kedge},~\ref{lem:R:binomial}
 and~\ref{lem:branching:rest}, below.
\end{rmk}

In order to bound $\tR_1$ and the other variables introduced in \Df{tilde:variables}, we shall
need to consider both edges $\te_j$ with $j\in\tE_\infty$ (i.e., edges of $\tcS_\cE$ that are used
in the algorithm), and edges $\te_j$, $j\in\tS\setminus\tE_\infty$, that nonetheless
intersect~$\tV_\infty$. Before embarking on the proof of \Lm{branching:main}, we shall use
\Rk{filtration} to control the distributions of the number of both types of edges.

We begin with the edges $\te_j$, $j\in\tE_\infty$. Let us define a random variable
$\tX^{(k)}_{t,w}\in\{0,1\}$ for each $k\ge 2$, $1\le t\le |\tV_\infty|$ and
$w\in M\setminus\tV_{t-1}$ by setting
\[
 \tX^{(k)}_{t,w}=1\qquad\Leftrightarrow\qquad w\in\tV_t\setminus\tV_{t-1}
 \quad\text{and}\quad\tj(w)\in\tS_k.
\]
Note that $\tX_{t,w} = \sum_{k\ge 2} \tX^{(k)}_{t,w}$ for every $t$ and~$w$. Define $\tcX_{t,w}$ to
be the event that $\tV^{(\ell)}<m^{1/2}$ just before we test vertex $w$ in time step~$t$.
Thus $\tX_{t,w}=1$ is only possible if $\tcX_{t,w}$ holds.

We write $\id_{\cA}$ to denote the indicator function of an event $\cA$. 

\begin{lemma}\label{lem:prob:kedge}
 Suppose that\/ $\cE$ is such that\/ $\cK(z)$ holds, $d(z)=1$, and\/
 $2s_2\le (1-\eps_1)m$. Then
 \[
  \E\big[\tX^{(2)}_{t,w}\mid\tcF_{t,<w}\big] = \left( \frac{2s_2}{m^2} + O\big( m^{-3/2} \big) \right)\id_{\tcX_{t,w}}
 \]
 and if\/ $k\ge 3$ then
 \[
  \E\big[\tX^{(k)}_{t,w}\mid\tcF_{t,<w}\big]\le\frac{2k^2 s_k}{m^{(k+2)/2}}\id_{\tcX_{t,w}}
 \]
 for every\/ $t\ge1$ and\/ $w\in M\setminus \tV_{t-1}$. As a consequence,
 \[
  \E\big[\tX_{t,w}\mid\tcF_{t,<w}\big]= \bigg( \frac{2s_2}{m^2} + O\big( m^{-3/2} \big) \bigg) \id_{\tcX_{t,w}}.
 \]
 Moreover, the constants implicit in the\/ $O(\cdot)$ notation are uniform in\/ $z$, $t$ and\/ $w$.
\end{lemma}

\begin{proof}
Let $\tV^{(\ell)}$ be the set of visited vertices just before we ask whether or not
$w\in\tV_t\setminus\tV_{t-1}$. If $|\tV^{(\ell)}|\ge m^{1/2}$ then $\tX_{t,w}=\id_{\tcX_{t,w}}=0$,
so we may assume that $|\tV^{(\ell)}|<m^{1/2}$. As $\cK(z)$ holds, we may also assume $k\le 4u_0$,
as otherwise $\tS_k=\emptyset$ and hence $\tX^{(k)}_{t,w}=0$. By~\eqref{eq:lowerbound:m:using:M},
we may also assume that $m \ge z_0^{1+o(1)}$, so $u_0 = o(\log m)$. We shall prove the first two
statements by bounding (for each $k\ge 2$) the probability that $w\in\tV_t\setminus\tV_{t-1}$
and $\tj(w)\in\tS_k$ by the expected number of edges $j\in\tS_k\setminus\tE^{(\ell)}$ of size $k$ with
$\{u,w\}\subseteq\te_j\subseteq\tV_{t-1}\cup\{w\}$.

Indeed, by \Rk{filtration}, conditioned on $\tcF_{t,<w}$, the edges~$\te_j$,
$j\in\tS_k\setminus\tE^{(\ell)}$, are each chosen independently and uniformly from the collection
of sets that fail to satisfy the test in Step~$2(b)$ in any prior time step where the edge
$\te_j$ was actually tested. This (crucially) includes all $k$-subsets of $M$ that contain at
least two vertices of $M\setminus\tV_{t-1}$. Since $|\tV_{t-1}|<m^{1/2}$, the number of such sets is therefore
\[
 \binom{m}{k}-O(m)\binom{|\tV_{t-1}|}{k-1}= \big( 1 + o(1) \big) \binom{m}{k}
 = \big( 1 + o(1) \big) \frac{m^k}{k!},
\]
where we have used the fact that $2\le k\le 4u_0=o(m^{1/2})$, by~\eqref{eq:lowerbound:m:using:M},
and the general fact that $\binom{n}{r}=(1-O(r^2/n))n^r/r!$. It follows that, for each $k \ge 2$ and each
$j\in\tS_k\setminus\tE^{(\ell)}$, the edge $\te_j$ satisfies
$\{u,w\}\subseteq\te_j\subseteq\tV_{t-1}\cup\{w\}$ with (conditional) probability
\begin{equation}\label{eq:prob:edge:w}
\big( 1 + o(1) \big) \frac{k!}{m^k} \cdot {|\tV_{t-1}| \choose k-2}
 \le \big(1+o(1)\big)\frac{k(k-1)|\tV_{t-1}|^{k-2}}{m^k},
\end{equation}
so, in particular, for each $k\ge3$ we have
\[
 \E\big[\tX^{(k)}_{t,w}\mid\tcF_{t,<w}\big]
 \le(1+o(1))\frac{k(k-1)|\tV_{t-1}|^{k-2}}{m^k}\cdot s_k
 \le\frac{2k^2s_k}{m^{(k+2)/2}}
\]
since $|\tV_{t-1}|<m^{1/2}$, as claimed. When $k = 2$, on the other hand, we can replace~\eqref{eq:prob:edge:w} by
$$\bigg( \binom{m}{2} - O\big( m \cdot |\tV_{t-1}| \big) \bigg)^{-1} = \frac{2}{m^2} + O\big( m^{-5/2} \big),$$
since $|\tV_{t-1}|<m^{1/2}$, and that at most $|\tV_{t-1}|$ edges have already been used in the process (since every time a new edge is used, we visit a new vertex). Hence the probability that some 2-edge satisfies $\te_j = \{u,w\}$ is
\[
\bigg( \frac{2}{m^2} + O\big( m^{-5/2} \big) \bigg)\Big(  s_2 + O\big( m^{1/2} \big) \Big) = \frac{2s_2}{m^2} + O\big( m^{-3/2} \big),
\]
as claimed, since $s_2 = O(m)$.

For the last part we note that $\tX_{t,w}=\sum_{k\ge2}\tX^{(k)}_{t,w}$ and so, assuming
$\tcX_{t,w}$ holds,
\begin{align*}
 \E\big[\tX_{t,w} - \tX^{(2)}_{t,w} \mid\tcF_{t,<w}\big] \le \sum_{k\ge3} \frac{2k^2 s_k}{m^{(k+2)/2}} = O\big( m^{-3/2} \big)
\end{align*}
as $\cK(z)$ holds, so $\sum_{k\ge 2}s_k=O(m)$. Uniformity in $z$, $t$ and $w$ follows as
all the $o()$ terms are in fact bounded by functions of~$m$, and since $\cK(z)$ holds we
have $m \ge z_0^{1+o(1)}$, by~\eqref{eq:lowerbound:m:using:M}.
\end{proof}

Next, define
\[
 \tf(t,w) := \begin{cases}
 \; k &\text{if $\,\tX_{t,w}=1\,$ and $\,\tj(w)\in\tS_k$, and}\\[8pt]
 \; 0 &\text{if $\,\tX_{t,w} = 0$.}
 \end{cases}
\]
Using \Lm{prob:kedge}, we can easily deduce the following bounds, which will be needed in the
proof of \Lm{branching:main}, below. 

\begin{lemma}\label{lem:exp:Xtw}
 Suppose \/ $\cE$ is such that\/ $\cK(z)$ holds, $d(z)=1$, and\/ 
 $2s_2\le (1-\eps_1)m$. Then
 for any\/ $0<\lambda\le\eps_1$, we have
 \[
  \E\big[e^{\lambda\tX_{t,w}}\mid\tcF_{t,<w}\big]
  \le\exp\bigg(\frac{\lambda-\lambda^2/2}{m}\bigg)
 \]
 and
 \[
  \E\big[e^{\lambda\tf(t,w)}\mid\tcF_{t,<w}\big]
  \le\exp\bigg(\frac{2\lambda}{m}\bigg)
 \]
 for every\/ $t\ge1$ and\/ $w\in M\setminus\tV_{t-1}$.
\end{lemma}
\begin{proof}
Since $2s_2\le (1-\eps_1)m$, it follows by \Ob{mgf} and \Lm{prob:kedge} that
\begin{align*}
 \E\big[e^{\lambda\tX_{t,w}}\mid\tcF_{t,<w}\big]
 &\le\exp\bigg(\big(e^\lambda-1\big)\left( \frac{2s_2}{m^2} + O\big( m^{-3/2} \big) \right)\bigg)\\
 &\le\exp\bigg(\frac{\big(e^\lambda-1\big)(1-\eps_1+o(1))}{m}\bigg)
  \le\exp\bigg(\frac{\lambda-\lambda^2/2}{m}\bigg),
\end{align*}
since $0<\lambda\le\eps_1<1$ and so $(e^\lambda-1)(1-\eps_1)
\le(\lambda+\frac{\lambda^2}{2}+\dots)(1-\lambda)<\lambda-\frac{\lambda^2}{2}$.
Similarly, recalling that $m \ge z_0^{1+o(1)}$ and that $\sum_{k \ge 2} 2^k s_k = O(m)$, since $\cK(z)$ holds, we have
\begin{align*}
 \E\big[e^{\lambda\tf(t,w)}\mid\tcF_{t,<w}\big]
 &=1+\sum_{k=2}^{4u_0}\big(e^{k\lambda}-1\big)\Prb\big(\tX^{(k)}_{t,w}=1\mid\tcF_{t,<w}\big)\\
 &\le 1+\big(e^{2\lambda}-1\big)\left( \frac{2s_2}{m^2} + O\big( m^{-3/2} \big) \right)
  +\sum_{k\ge3}\frac{2k^2(e^{k\lambda}-1)}{m^{(k+2)/2}}s_k\\
 &\le 1+\big(e^{2\lambda}-1\big)\frac{1-\eps_1}{m}+ O\big( \lambda m^{-3/2} \big)\\
 &\le1+\frac{2\lambda}{m}\le\exp\bigg(\frac{2\lambda}{m}\bigg),
\end{align*}
since $0<\lambda\le\eps_1<1$ and so $(e^{2\lambda}-1)(1-\eps_1)
\le(2\lambda+\frac{4\lambda^2}{2}+\frac{8\lambda^3}{6}+\dots)(1-\lambda)<2\lambda$.
\end{proof}

In the proof of~\eqref{eq:branching:main:b} we shall use the inequality
\begin{equation}\label{eq:branching:main:b:first:step}
 \E\big[\Delta_k^{(1)}\id_{\{t\le\tR_1<m^{1/2}\}}\big]
 \le e^{-\lambda t} \cdot \E\big[\Delta_k^{(1)}e^{\lambda\tR_1}\id_{\{2\le\tR_1<m^{1/2}\}}\big],
\end{equation}
so it will be important that we have some control over the distributions of $\Delta_k^{(1)}$ and $\tR_1$ conditioned on the `positive' information that $\tR_1 > 1$. The next observation provides us with this control.

\begin{obs}\label{obs:step1} 
The random variable \/ $|\tV_1|-1$ is stochastically dominated by the binomial random variable 
$\Bin(s_2,2/m)$. In particular, $\E\big[ e^{|\tV_1|}\mid\tR_1>1 \big]=O(1)$.
\end{obs}
\begin{proof}
Only 2-edges can be included in $\tE_1$ as $j\in\tE_1$ implies $|\te_j\setminus\{v\}|=1$.
Each 2-edge is included precisely when it contains $v$ and is not identical to a previously
encountered 2-edge, and this occurs with probability at most $2/m$. Thus
$|\tV_1|$ is stochastically dominated by a $1+\Bin(s_2,2/m)$ random variable.
The condition that $\tR_1>1$ is precisely the condition that the exploration process
does not immediately die out, so is equivalent to the condition that $|\tV_1|>1$.
Now conditioning on $\tR_1>1$ is equivalent to conditioning on at least one
2-edge containing~$v$. But conditioned on that, $|\tV_1|-2$ is stochastically bounded
by a $\Bin(s_2,2/m)$ random variable, as there are at most $s_2$ remaining edges
to test, and each adds 1 to $|\tV_1|$ with probability at most $2/m$.
By \Ob{mgf}, $\E[e^{\Bin(s_2,2/m)}]\le \exp((e-1)2s_2/m)=O(1)$, so the second result
follows.
\end{proof}

We are ready to bound the contribution to $\tR_1$ of the edges used in \Al{exploration}. Let
\[
 \tW_\infty:=\sum_{w\in\tV_\infty}|\te_{\tj(w)}|=\sum_{j\in\tE_\infty}d_j
\]
denote the sum of the sizes of these edges; as these edges all lie inside~$\tV_\infty$, this is
precisely their contribution to $\tR_1$. The following lemma controls the size of $\tW_\infty$.

\begin{lemma}\label{lem:Winfty}
 There exists some\/ $\lambda>0$, depending only on\/~$\eps_1$, such that
 \[
  \E\big[e^{\lambda\tW_\infty}\mid\tR_1>1\big]=O(1)
 \]
\end{lemma}
\begin{proof}
Note that the condition $\tR_1>1$ is equivalent to $\tW_\infty>0$, and
is $\tcF_1$-measurable, as one discovers whether or not $\tR_1>1$ in the
first round of \Al{exploration}.
We will first need to control the large deviations of $|\tV_\infty|$.

\claim{Claim 1:} For every $t\ge 1$,
\[
 \Prb\big(|\tV_\infty|\ge t\mid\tR_1>1\big)=O\big(e^{-\eps_1^2 t/2}\big).
\]
\begin{proof}[Proof of Claim~1]
Since $|\tV_t|=|\tA_t|+t\ge t$ for $t\le|\tV_\infty|$, we have
\begin{equation}\label{eq:vinfbound}
 \Prb\big(|\tV_\infty|\ge t\mid\tR_1>1\big)=\Prb\big(|\tV_t|\ge t\mid\tR_1>1\big)
 \le e^{-\eps_1 t}\,\E\big[e^{\eps_1|\tV_t|}\mid\tR_1>1\big].
\end{equation}
We shall bound $\E\big[e^{\eps_1|\tV_t|}\mid\tR_1>1\big]$ using \Lm{exp:Xtw} and the law of iterated
expectations. Indeed, setting $\tX_t=\sum_w\tX_{t,w}$, so that $\tX_t=|\tV_t\setminus\tV_{t-1}|$,
we have
\begin{equation}\label{eq:iteratedexpectation:V}
 \E\big[e^{\eps_1|\tV_t|}\mid\tR_1>1\big]=\E\big[e^{\eps_1 (1+\tX_1+\dots+\tX_t)}\mid\tR_1>1\big]
 =\E\Big[e^{\eps_1|\tV_{t-1}|}\cdot\E\big[e^{\eps_1\tX_t}\mid\tcF_{t-1}\big]\bmid\tR_1>1\Big]
\end{equation}
and similarly
\begin{align}
 \E\big[e^{\eps_1\tX_t}\mid\tcF_{t-1}\big]
 &=\E\big[e^{\eps_1 (\tX_{t,w_1}+\dots+\tX_{t,w_r})}\mid\tcF_{t-1}\big]\nonumber\\
 &=\E\Big[e^{\eps_1 (\tX_{t,w_1}+\dots+\tX_{t,w_{r-1}})}\cdot
 \E\big[e^{\eps_1\tX_{t,w_r}}\mid\tcF_{t,<w_r}\big]\mid\tcF_{t-1}\Big].\label{eq:iteratedexpectation}
\end{align}
Now, applying \Lm{exp:Xtw} with $\lambda=\eps_1$, we have
\begin{equation}\label{eq:Bernoulli:inequality}
 \E\big[e^{\eps_1\tX_{t,w}}\mid\tcF_{t,<w}\big]\le\exp\bigg(\frac{\eps_1-\eps_1^2/2}{m}\bigg)
\end{equation}
for every $t\ge 1$ and $w\in M\setminus\tV_{t-1}$. Combining this
with~\eqref{eq:iteratedexpectation}, and iterating the procedure, we obtain
\begin{align*}
 \E\big[e^{\eps_1\tX_t}\mid\tcF_{t-1}\big]
 &\le\exp\bigg(\frac{\eps_1-\eps_1^2/2}{m}\bigg)
 \E\big[e^{\eps_1 (\tX_{t,w_1}+\dots+\tX_{t,w_{r-1}})}\mid\tcF_{t-1}\big]\\
 &\le\dots\le\exp\bigg(\frac{(\eps_1-\eps_1^2/2)r}{m}\bigg)
 \le\exp\big(\eps_1-\eps_1^2/2\big),
\end{align*}
since $r=|M\setminus\tV_{t-1}|\le m$. Hence, using~\eqref{eq:iteratedexpectation:V},
and iterating again, we have
\begin{align*}\label{eq:iteratedexpectation:E}
\E\big[e^{\eps_1|\tV_t|}\mid\tR_1>1\big]
& \le e^{\eps_1-\eps_1^2/2}\E\big[e^{\eps_1|\tV_{t-1}|}\mid\tR_1>1\big]\\
 & \le\dots\le e^{(\eps_1-\eps_1^2/2)(t-1)}\E\big[e^{\eps_1|\tV_1|}\mid\tR_1>1\big].
\end{align*}
Thus by \Ob{step1},
\[
 \E\big[e^{\eps_1|\tV_t|}\mid\tR_1>1\big]=O\big(e^{(\eps_1-\eps_1^2/2)t}\big).
\]
Finally, it follows from~\eqref{eq:vinfbound} that
\[
 \Prb\big(|\tV_\infty|\ge t\mid\tR_1>1\big)
 \le e^{-\eps_1 t}\cdot\E\big[e^{\eps_1|\tV_t|}\mid\tR_1>1\big]
 =O\big(e^{-\eps_1^2 t/2}\big)
\]
as claimed.
\end{proof}

We shall next use a similar argument to control
\[
 \tW_t:=\sum_{w\in\tV_t}|\te_{\tj(w)}|=\sum_{j\in\tE_t}d_j,
\]
the sum of the sizes of the edges used in the first $t$ iterations of \Al{exploration}.

\claim{Claim 2:} If $0<\lambda\le\eps_1$, then
\[
 e^{\lambda\tW_t-2\lambda t}
\]
is a super-martingale with respect to the filtration $(\cF_t)_{t \ge 0}$.

\begin{proof}[Proof of Claim~2]
The proof is similar to that of Claim~1, except the bound~\eqref{eq:Bernoulli:inequality} is
replaced by
\[
 \E\big[e^{\lambda\tf(t,w)}\mid\tcF_{t,<w}\big]\le\exp\bigg(\frac{2\lambda}{m}\bigg),
\]
which also follows from \Lm{exp:Xtw}. Indeed,
\begin{align*}\label{eq:iteratedexpectation:E}
 \E\big[e^{\lambda (\tW_t-\tW_{t-1})}\mid\tcF_{t-1}\big]
 &=\E\Big[e^{\lambda (\tf(t,w_1)+\dots+\tf(t,w_{r-1}))}
 \cdot\E\big[e^{\lambda\tf(t,w_r)}\mid\tcF_{t,<w_r}\big]\mid\tcF_{t-1}\Big]\\
  & \le \exp\bigg(\frac{2\lambda}{m}\bigg) \E\Big[e^{\lambda (\tf(t,w_1)+\dots+\tf(t,w_{r-1}))} \mid\tcF_{t-1}\Big]\\
 &\le\dots\le\exp\bigg(\frac{2\lambda r}{m}\bigg)\le e^{2\lambda},
\end{align*}
since $r\le m$. Since $\tW_{t-1}$ is $\tcF_{t-1}$-measurable, it follows immediately that
\[
 \E\big[e^{\lambda\tW_t-2\lambda t}\mid\tcF_{t-1}\big] 
 \le e^{\lambda\tW_{t-1}-2\lambda(t-1)},
\]
as required.
\end{proof}

We are now ready to bound the expectation of $e^{\lambda\tW_\infty}$. Observe first that
\[
 \E\big[e^{\lambda\tW_\infty/2}\mid\tR_1>1\big]
 \le\frac{1}{2} \Big( \E\big[e^{\lambda\tW_\infty-2\lambda|\tV_\infty|}\mid\tR_1>1\big]+
 \E\big[e^{2\lambda|\tV_\infty|}\mid\tR_1>1\big] \Big),
\]
by the convexity of $e^x$. Now, if $\lambda<\eps_1^2/4$ then
$$\E\big[e^{2\lambda|\tV_\infty|}\mid\tR_1>1]
\le\sum_{t = 0}^\infty e^{2\lambda t}\Prb(|\tV_\infty|\ge t\mid\tR_1>1)=O(1)$$
by Claim~1. Moreover, since the event $\tR_1>1$ is $\tcF_1$-measurable, it follows from Claim~2 by the optional stopping theorem that
\[
 \E\big[e^{\lambda\tW_\infty-2\lambda|\tV_\infty|}\mid\tR_1>1\big]
 \le\E\big[e^{\lambda\tW_1}\mid\tR_1>1\big]=O(1)
\]
for every $\lambda < 1/2$. Indeed, since $\tW_1=2(|\tV_1|-1)$ is twice the number of 2-edges used in the first round
of \Al{exploration}, the last equality follows by \Ob{step1}. Thus
\[
 \E\big[e^{\lambda \tW_\infty}\mid\tR_1>1\big]=O(1)
\]
for every $\lambda < \eps_1^2/8$, as required.
\end{proof}

We shall next use \Rk{filtration} to control the distribution of the number of remaining edges
that nonetheless intersect~$\tV_\infty$. For each $k\ge 2$, define
\[
 \tR(k):=\big|\big\{j\in\tS_k\setminus\tE_\infty:\te_j\cap\tV_\infty\ne\emptyset\big\}\big|
\]
to be the number of edges of $\tcS_\cE$ of size $k$ that have at least one vertex removed but are
not used in \Al{exploration}. For each $k\ge 2$, define binomial random variables
\[
 Z(k)\sim\Bin\bigg(s_k,\frac{k|\tV_\infty|}{m-k}\bigg).
\]

\begin{lemma}\label{lem:R:binomial}
 The random variables\/ $\tR(k)$ are conditionally independent given\/ $\tcF_\infty$ and,
 conditioned on\/ $\tcF_\infty$, are stochastically dominated by\/ $Z(k)$ for each\/~$k\ge2$.
\end{lemma}
\begin{proof}
Since $\cK(z)$ holds, we have $s_k=0$ for all $k>4u_0$, so we may assume that $k\le 4u_0$.
Run the algorithm to reveal $\tV_\infty$ and $\tE_\infty$. By \Rk{filtration}, we have not
revealed any edge of $\tS_k\setminus\tE_\infty$, though we have gained some `negative' information
about the events $\{\te_j\cap\tV_\infty\ne\emptyset\}$. To be precise, conditioned on the
information we have observed during the process (i.e.,~$\tcF_\infty$), the edges $\te_j$,
$j\in\tS_k\setminus\tE_\infty$, are each chosen independently and uniformly from the collection
of sets of size $k$ that would not have resulted in $j$ being picked as some~$\tj(w)$. The crucial
observation in this case is that this collection is $\tcF_\infty$-measurable, and
includes all $k$-subsets of $M\setminus\tV_\infty$,
cf.\ the proof of \Lm{prob:kedge}. Thus, the (conditional) probability that $\te_j$ meets
$\tV_\infty$ is at most the probability that a uniformly chosen $k$-set of $M$
meets $\tV_\infty$.

The events $\{\te_j\cap\tV_\infty\ne\emptyset\}$,  $j\in\tS_k\setminus\tE_\infty$, are therefore
conditionally independent given $\tcF_\infty$, and each has (conditional) probability at most
\[
 \frac{|\tV_\infty|\binom{m}{k-1}}{\binom{m}{k}}=\frac{k|\tV_\infty|}{m-k}.
\]

As $|\tS_k\setminus\tE_\infty|\le s_k$, $\tR(k)$ is stochastically dominated by~$Z(k)$.
The $\tR(k)$ are conditionally independent given $\tcF_\infty$ as they depend on disjoint
sets $\tS_k\setminus\tE_\infty$ of random edges that are themselves conditionally
independent given~$\tcF_\infty$.
\end{proof}

We are finally ready to prove the key lemma of this section.

\begin{proof}[Proof of \Lm{branching:main}]
Recall that
\[
 \tW_\infty:=\sum_{w\in\tV_\infty}|\te_{\tj(w)}|=\sum_{j\in\tE_\infty}d_j
\]
is the sum of the sizes of the edges used in \Al{exploration}. Observe that
\[
 \min\big\{\tR_1,m^{1/2}\big\}\le\tW_\infty+\sum_{k=2}^\infty k\cdot\tR(k),
\]
since each edge counted by $\tR(k)$ can contribute at most $k$ to~$\tR_1$, so
this holds when $|\tV_\infty|<m^{1/2}$, and $\tW_\infty\ge |\tV_\infty|\ge m^{1/2}$
otherwise.
Recall that $\tW_\infty$ is $\tcF_\infty$-measurable, and that $\cK(z)$ implies that $\tR(k)=0$
for all $k\ge 4u_0$. Given~$\tcF_\infty$, the random variables $\tR(k)$ are conditionally independent,
so \Lm{R:binomial} implies that 
\begin{align}
 \E\big[e^{\lambda\min\{\tR_1,m^{1/2}\}}\mid\tR_1>1\big]
 &\le\E\Big[\E\big[e^{\lambda\tW_\infty+\sum_{k\ge2}\lambda k\tR(k)}\mid\tcF_\infty\big]\bmid\tR_1>1\Big]\notag\\
 &=\E\bigg[e^{\lambda\tW_\infty}
 \prod_{k=2}^{4u_0}\E\big[e^{\lambda k\tR(k)}\mid\tcF_\infty\big]\Bmid\tR_1>1\bigg].\label{eq:expR1}
\end{align}
Moreover, it follows from \Lm{R:binomial} and \Ob{mgf} that
\[
 \E\big[e^{\lambda k\tR(k)}\mid\tcF_\infty\big]\le\E\big[e^{\lambda kZ(k)}\mid\tcF_\infty\big]
 \le\exp\bigg(\big(e^{\lambda k}-1\big)\frac{k|\tV_\infty|}{m-k}\cdot s_k\bigg),
\]
since $Z(k)$ is a sum of $s_k$ independent Bernoulli random variables, each of expectation
$k|\tV_\infty|/(m-k)$. Since $e^w-1=(1-e^{-w})e^w\le we^w$ for all $w\ge0$, and $\sum_{k\ge 2}2^k s_k=O(m)$ (since $\cK(z)$ holds), it follows that 
\[
 \sum_{k=2}^{4u_0}\big(e^{\lambda k}-1\big)\frac{k|\tV_\infty|}{m-k}\cdot s_k
 \le\frac{|\tV_\infty|}{m-4u_0}
 \sum_{k=2}^{4u_0}\lambda k^2 e^{\lambda k}s_k=O(\lambda|\tV_\infty|)
\]
for $\lambda\le\eps_1<\log 2$. Thus, recalling that $\tW_\infty\ge |\tV_\infty|$, we have
\begin{equation}\label{eq:expW:calc:other}
 \E\big[e^{\lambda\min\{\tR_1,m^{1/2}\}}\mid\tR_1>1\big]
 \le\E\big[e^{\lambda\tW_\infty+O(\lambda |\tV_\infty|)}\mid\tR_1>1\big]
 =\E\big[e^{O(\lambda)\tW_\infty}\mid\tR_1>1\big]
\end{equation}
for all $0<\lambda\le\eps_1$. Hence, by \Lm{Winfty}, it follows that
\[
 \E\big[e^{\lambda\min\{\tR_1,m^{1/2}\}}\mid\tR_1>1\big]
 \le \E\big[e^{O(\lambda)\tW_\infty}\mid\tR_1>1\big]=O(1),
\]
for sufficiently small $\lambda>0$. Now, by \Ob{step1}, $\Prb(\tR_1>1) = \Prb( |\tV_1| > 1) \le 2s_2/m$. Thus
\[
 \Prb(\tR_1\ge t)\le e^{-\lambda t}\,
 \E\big[e^{\lambda\min\{\tR_1,m^{1/2}\}}\mid\tR_1>1\big]\,\Prb\big(\tR_1>1\big)
 =\frac{O(s_2)}{m}e^{-\lambda t}
\]
for all sufficiently small $\lambda>0$ and all $2\le t\le m^{1/2}$, as required.

For the second part, observe first that, as noted in~\eqref{eq:branching:main:b:first:step}, we have 
\begin{align*}
\E\big[\Delta_k^{(1)}\id_{\{t\le\tR_1<m^{1/2}\}}\big]
& \le e^{-\lambda t} \cdot \E\big[\Delta_k^{(1)}e^{\lambda\tR_1}\id_{\{2\le\tR_1<m^{1/2}\}}\big]\\
& \le e^{-\lambda t} \cdot \E\big[\Delta_k^{(1)}e^{\lambda\tR_1}\id_{\{\tR_1<m^{1/2}\}}\mid\tR_1>1\big].
\end{align*}
Recall that  $\tDelta_k^{(1)} = |\{ j\in\tS_k : |\te_j\cap\tD_\infty|=1 \}|$, and note that therefore
$$\E\Big[\Delta_k^{(1)}e^{\lambda\tR_1}\id_{\{\tR_1<m^{1/2}\}}\mid\tR_1>1\Big] \le \sum_{j \in \tS_k} \E\Big[ \id_{\{|\te_j\cap\tV_\infty|=1\}}e^{\lambda\min\{\tR_1,m^{1/2}\}}\bmid\tR_1>1 \Big],$$
since $\tV_\infty=\tD_\infty$ when $\tR_1 < m^{1/2}$. Now, repeating the argument of~\eqref{eq:expR1}--\eqref{eq:expW:calc:other}, we obtain
$$\E\Big[ \id_{\{|\te_j\cap\tV_\infty|=1\}}e^{\lambda\min\{\tR_1,m^{1/2}\}}\bmid\tR_1>1 \Big] \le \E\Big[\id_{\{|\te_j\cap\tV_\infty|=1\}}e^{\lambda+O(\lambda)\tW_\infty}\bmid\tR_1>1\Big]$$
for every $j \in \tS_k$ and $0 < \lambda \le \eps_1$, since $|\te_j \cap\tV_\infty|=1$ and $k \ge 3$ imply that $j \not\in \tE_\infty$ (and that $\te_j$ contributes exactly one to~$\tR_1$), and the edges not used in the algorithm are conditionally independent given $\tcF_\infty$. Combining the above inequalities, we obtain
\begin{align*}
\E\big[\Delta_k^{(1)}\id_{\{t\le\tR_1<m^{1/2}\}}\big]
& \le e^{-\lambda t} \sum_{j \in \tS_k} \E\Big[\id_{\{|\te_j\cap\tV_\infty|=1\}}e^{\lambda+O(\lambda)\tW_\infty}\bmid\tR_1>1\Big]\\
& \le e^{-\lambda t} \sum_{j \in \tS_k} \E\bigg[\frac{k|\tV_\infty|}{m-k} e^{\lambda+O(\lambda)\tW_\infty}\Bmid\tR_1>1\bigg], 
\end{align*}
cf. the proof of Lemma~\ref{lem:R:binomial}. Since $|\tS_k| = s_k$ and $|\tV_\infty|\le\tW_\infty=O(e^{\lambda\tW_\infty})$, it follows by \Lm{Winfty} that
$$\E\big[\Delta_k^{(1)}\id_{\{t\le\tR_1<m^{1/2}\}}\big] \le \frac{O(k s_k)}{m} e^{-\lambda t} \cdot \E\big[e^{O(\lambda)\tW_\infty}\mid\tR_1>1\big] = \frac{O(ks_k)}{m} e^{-\lambda t}$$
for all sufficiently small $\lambda > 0$, as required. 
\end{proof}

Having done the hard part, it is now relatively straightforward to deduce \Th{branching}, using
\Th{compare}. The next step is to use Lemmas \ref{lem:branching:main}, \ref{lem:prob:kedge}
and~\ref{lem:R:binomial} to deduce the following estimates for the other quantities introduced
in \Df{tilde:variables}.

\begin{lemma}\label{lem:branching:rest}
 Suppose that\/ $\cE$ is such that\/ $\cK(z)$ holds, $d(z)=1$, and\/
 $2s_2\le (1-\eps_1)m$. Then, in the independent random hypergraph model,
 \begin{itemize}
  \item[$(a)$] $\ds\E\big[\tD\big]=\bigg(1-\frac{2s_2}{m}+o(1)\bigg)^{-1}$,
   \item[$(b)$] $\ds\E\big[\tDelta_k\big]=\bigg(1-\frac{2s_2}{m}+o(1)\bigg)^{-1}\frac{k s_k}{m} + O\big( m^{-1/2} \big)$
   for all\/ $k\ge 2$,  
   \item[$(c)$] $\E\big[\tD^2\big]=O(1)$,
  \item[$(d)$] $\E\big[\tDelta'\big]=O(m^{-1/2})$,
 \end{itemize}
 where the bounds implicit in the\/ $o(\cdot)$ and\/ $O(\cdot)$ notation are uniform in\/ $k$ and\/~$z$.
\end{lemma}

\begin{proof}
Parts $(a)$ and~$(c)$ are easier: we shall prove them first. Part $(c)$ follows immediately from \Lm{branching:main}, since $\tD\le\tR_1$, so
\[
 \E\big[\tD^2\big] \le \E\big[\tR_1^2\big] \le \sum_{t = 1}^\infty t^2 \cdot \Prb( \tR_1 \ge t) = O(1),
\]
as required.

To prove~$(a)$, recall that $\tX_t=|\tV_t\setminus\tV_{t-1}|=\sum_w \tX_{t,w}$, and that
$\tcX_{t,w}$ is the event that $\tV^{(\ell)}<m^{1/2}$ just before we test whether or
not $\tX_{t,w}=1$. Also, by \Lm{prob:kedge}, we have 
\begin{equation}\label{eq:Ex:Xtw}
 \E\big[\tX_{t,w}\mid\tcF_{t,<w}\big] = \bigg( \frac{2s_2}{m^2} + O\big( m^{-3/2} \big) \bigg) \id_{\tcX_{t,w}}
\end{equation}
for every $t\in\N$ and $w\in M\setminus\tV_{t-1}$. Let us define, for each $t\ge 1$, an $\tcF_{t-1}$-measurable event
\[
 \tcX_t:=\bigg\{\Prb\Big(|\tV_t|\ge m^{1/2}\mid\tcF_{t-1}\Big)\le\frac{1}{m^2}\bigg\},
\]
and observe that, if $\tcX_t$ holds, then $\E\big[ \id_{\tcX_{t,w}^c} \mid \tcF_{t-1} \big] \le 1/m^2$ for each $w \in M\setminus\tV_{t-1}$, which in turn implies that 
\begin{equation}\label{eq:Ex:Xt}
 \E\big[\tX_t\mid\tcF_{t-1}\big] = \sum_{w \in M\setminus\tV_{t-1}} \E\big[ \tX_{t,w} \mid \tcF_{t-1}\Big] = \frac{2s_2 + o(m)}{m^2} \cdot |M \setminus \tV_{t-1}| = \frac{2s_2}{m}+o(1),
\end{equation}
by~\eqref{eq:Ex:Xtw} and since $|\tV_{t-1}| = o(m)$. 
Observe also that, by \Lm{branching:main}, we have
\begin{equation}\label{eq:mainlemma:application}
 \Prb\big(\tcX_t^c\big)\le m^2\Prb\big(|\tV_\infty|\ge m^{1/2}\big)
 \le m^2\Prb\big(\tR_1\ge m^{1/2}\big)=O\big(m^2e^{-\lambda m^{1/2}}\big),
\end{equation}
since $\tD\le\tR_1$ and $|\tV_\infty|=\min\big\{\tD,\lceil m^{1/2}\rceil\big\}$,
by \Ob{alg:basicproperties}.

We now define sequences $\tY^+_t$ and $\tY^-_t$ as follows. Let $0 < \eps < \eps_1$ 
be an arbitrarily small constant, set $\tY^+_0=1$ and $\tY^-_0=-1$, and for each
$1\le t\le|\tV_\infty|$, define
\[
 \tY_t^+:=\begin{cases}
  \ds|\tA_t|+\bigg(1-\frac{2s_2}{m}-\eps\bigg)t,&\text{if $\tcX_t$ holds;}\\[8pt]
  \tY^+_{t-1}, &\text{otherwise.}
 \end{cases}
\]
Similarly, define $\tY^-_t:=-|\tA_t|-\big(1-\frac{2s_2}{m}+\eps\big)t$ if $\tcX_t$ holds,
and $\tY^-_t:=\tY^-_{t-1}$ otherwise. We claim that $\tY^+_t$ and $\tY^-_t$ are both
super-martingales with respect to the filtration $(\tcF_t)_{t \ge 0}$. Indeed, for each
$1\le t\le |\tV_\infty|$, if $\tcX_t$ holds then
\[
\E\big[\tY^+_t - \tY^+_{t-1} \mid\tcF_{t-1}\big] = \E\big[|\tA_t| - |\tA_{t-1}| \mid\tcF_{t-1}\big] + \bigg(1-\frac{2s_2}{m} - \eps \bigg) \le 0
\]
by~\eqref{eq:Ex:Xt}, since $|\tA_t| - |\tA_{t-1}| = \tX_t-1$, and similarly for $\tY^-_t$. 
Therefore, by the optional stopping
theorem applied to $\tY^+_t$, and writing $\tY^+_\infty:=\tY^+_{|\tV_\infty|}$, it follows that
\[
 \E\big[\tY^+_\infty\big]\le\tY^+_0=1.
\]
Now, recalling that $\tA_{|\tV_\infty|} = 0$, it follows that
\[
 \tY^+_\infty=\bigg(1-\frac{2s_2}{m}-\eps\bigg)|\tV_\infty|
\]
if $\tcX_t$ holds for every $1\le t\le|\tV_\infty|$, and in general $|\tY^+_\infty|\le m^{1/2}$, by \Ob{alg:basicproperties}. Recall also that, by \Ob{alg:basicproperties}, we have $\tD=|\tV_\infty|$ if
$|\tV_\infty|<m^{1/2}$, and otherwise $\tD\le m$. Hence, by~\eqref{eq:mainlemma:application},
it follows that
\[
 \E\big[\tY^+_\infty\big]\ge\bigg(1-\frac{2s_2}{m}-2\eps\bigg)\E\big[\tD\big].
\]
Applying the same argument to $\tY^-_t$, we obtain
$1\le\big(1-\frac{2s_2}{m}+2\eps\big)\E\big[\tD\big]$. Since $\eps>0$ was arbitrary,
and $2s_2\le(1-\eps_1)m$, it follows that
\[
 \E\big[\tD\big]=\bigg(1-\frac{2s_2}{m}+o(1)\bigg)^{-1},
\]
as required.

To prove parts~$(b)$ and~$(d)$, we shall need to consider separately edges of
$\tS_k\cap\tE_\infty$ (i.e., edges that are used in the process), and edges of
$\tS_k\setminus\tE_\infty$ that intersect $\tV_\infty$. We will show that
$\tE_\infty$ is unlikely to contain any edges of size at least three, but
expects to have the required number of edges of size~2.

The proof of part~$(b)$ in the case $k=2$ is very similar to that of part~$(a)$, so we shall be
somewhat brief with the details. Let $\lambda > 0$ be a sufficiently large constant, set $\tZ^+_0=0$, and for each $1\le t\le |\tV_\infty|$, define
\[
 \tZ_t^+:=\begin{cases}
  \ds|\tS_2\cap\tE_t|-\left( \frac{2s_2}{m} + \frac{\lambda}{m^{1/2}} \right) \cdot t,&\text{if $\tcX_t$ holds;}\\[8pt]
  \tZ^+_{t-1},&\text{otherwise.}
 \end{cases}
\]
Similarly, define $\tZ_t^-:=-|\tS_2 \cap \tE_t| + \big( \frac{2s_2}{m} - \lambda m^{-1/2} \big) t$ 
if $\tcX_t$ holds, and $\tZ_t^-:=\tZ^-_{t-1}$ otherwise. Recall that $\tX^{(k)}_{t,w}=1$ if and only if 
$w\in\tV_t\setminus\tV_{t-1}$ and $\tj(w)\in\tS_k$, and that
\[
 \E\big[\tX^{(2)}_{t,w}\mid\tcF_{t,<w}\big]= \left( \frac{2s_2}{m^2} + O\big( m^{-3/2} \big) \right)\id_{\tcX_{t,w}}
\]
for every $t\in\N$ and $w\in M\setminus\tV_{t-1}$, by \Lm{prob:kedge}. It follows that
\[
\E\big[\tZ^+_t - \tZ^+_{t-1} \mid\tcF_{t-1}\big] =  \E\big[|\tS_2\cap\tE_t| - |\tS_2\cap\tE_{t-1}| \mid\tcF_{t-1}\big] -\left( \frac{2s_2}{m} + \frac{\lambda}{m^{1/2}} \right) \le 0
\]
if $\tcX_t$ holds, since $s_2 = O(m)$. 
Thus $\tZ^+_t$ is a super-martingale with
respect to the filtration $(\tcF_t)_{t \ge 0}$, and hence, by the optional stopping theorem,
\[
 \E\big[\tZ^+_\infty\big]\le\tZ^+_0= 0,
\]
where $\tZ^+_\infty:=\tZ^+_{|\tV_\infty|}$. Now,
\begin{equation}\label{eq:Zplusinfty}
 \tZ^+_\infty=|\tS_2\cap\tE_\infty| - \left( \frac{2s_2}{m} + \frac{\lambda}{m^{1/2}} \right) |\tV_\infty|
\end{equation}
if $\tcX_t$ holds for every $1\le t\le|\tV_\infty|$, and otherwise $|\tZ^+_\infty|\le s_2 \le m$.
Thus, by~\eqref{eq:mainlemma:application}, it follows that
\begin{equation}\label{eq:used:twoedges:upper}
 \E\big[|\tS_2\cap\tE_\infty|\big]\le \left( \frac{2s_2}{m} + \frac{\lambda}{m^{1/2}} \right)\E\big[|\tV_\infty|\big]+\frac{1}{m^2}.
\end{equation}
Now, using part~$(a)$, and recalling that $\eps>0$ was arbitrary and $2\le 2s_2\le (1-\eps_1)m$
(since if $s_2=0$ then we trivially have $\E[\tDelta_2]=0$), it follows that
\[
 \E\big[|\tS_2\cap\tE_\infty|\big]\le\frac{2s_2}{m}\bigg(1-\frac{2s_2}{m}+o(1)\bigg)^{-1} + \/ O\big( m^{-1/2} \big).
\]
Repeating the argument for $\tZ_t^-$, we obtain a corresponding lower bound, and hence
\begin{equation}\label{eq:used:twoedges}
 \E\big[|\tS_2\cap\tE_\infty|\big]=\frac{2s_2}{m}\bigg(1-\frac{2s_2}{m}+o(1)\bigg)^{-1}+ \/ O\big( m^{-1/2} \big).
\end{equation}

When $k\ge 3$ we only need a weaker bound, and as a consequence the argument is simpler.
Recall first that
\[
 \E\big[\tX^{(k)}_{t,w}\mid\tcF_{t,<w}\big]\le\frac{2k^2 s_k}{m^{(k+2)/2}}\id_{\tcX_{t,w}}
\]
for every $t\in\N$, $k \ge 3$ and $w\in M\setminus V_{t-1}$, by \Lm{prob:kedge}. Note also that
\[ 
 |\tS_k\cap\tE_t|=|\tS_k\cap\tE_{t-1}|+\sum_{w\in M\setminus\tV_{t-1}}\tX^{(k)}_{t,w},
\]
and therefore $\tU^{(k)}_t := |\tS_k\cap\tE_t| - 2k^2 s_k m^{-k/2} t$ is a super-martingale, since
\[
\E\big[\tU^{(k)}_t - \tU^{(k)}_{t-1} \mid\tcF_{t-1}\big] = \E\big[|\tS_k\cap\tE_t| - |\tS_k\cap\tE_{t-1}| \mid\tcF_{t-1}\big] - \frac{2k^2 s_k}{m^{k/2}} \le 0
\]
for every $t\ge 1$. Hence, by the optional stopping theorem, we have
\begin{equation}\label{eq:used:threeedges:kfixed}
 \E\big[|\tS_k\cap\tE_\infty|\big] \le \frac{2k^2 s_k}{m^{k/2}} \cdot \E\big[|\tV_\infty|\big].
\end{equation}
Since $\E\big[|\tV_\infty|\big]=O(1)$ (by \Lm{branching:main}, or by part $(a)$), it follows,
using the event $\cK(z)$, that
\begin{equation}\label{eq:used:threeedges}
 \sum_{k=3}^\infty\E\big[|\tS_k\cap\tE_\infty|\big] = O\big(m^{-1/2}\big).
\end{equation}

We shall next deal with edges that are not used in the process, but nevertheless intersect
$\tV_\infty$ in at least two vertices. To be precise, we shall show that the expected size of
\[
 \tR_2(k):=\big|\big\{j\in\tS_k\setminus\tE_\infty:|\te_j\cap\tV_\infty|\ge 2\big\}\big|
\]
is small for each $2\le k\le 4u_0$. Indeed, recall (see \Rk{filtration} and the proofs of
Lemmas~\ref{lem:prob:kedge} and~\ref{lem:R:binomial}) that conditional on the information we have
observed during the process, each edge $\te_j$, $j\in\tS_k\setminus\tE_\infty$, is chosen
uniformly from a collection of sets of size $k$ that depends on $j$ but always
includes all $k$-subsets that contain at least two elements of $M\setminus\tV_\infty$.
The (conditional) probability of the event $\{|\te_j\cap\tV_\infty|\ge 2\}$ is therefore at most
\begin{equation}\label{eq:condprob:twoinV}
 \binom{|\tV_\infty|}{2}\binom{m}{k-2}\binom{m-|\tV_\infty|}{k}^{-1}
 \le\frac{k^2\cdot |\tV_\infty|^2}{m^2}
\end{equation}
since $|\tV_\infty|\le m^{1/2}$, so $k\cdot |\tV_\infty|=o(m)$. Note also that if $k\ge 3$ then
the (conditional) probability of the event $\{|\te_j\cap\tV_\infty|=1\}$ is at least
\[
 |\tV_\infty|\binom{m-|\tV_\infty|}{k-1}\binom{m}{k}^{-1}
 =\big(1+o(1)\big)\frac{k\cdot |\tV_\infty|}{m}.
\]
Thus the expected number of edges of size $k$ that intersect $\tV_\infty$ is at least
\begin{equation}\label{eq:condprob:oneinV}
 \big(1+o(1)\big)\frac{k s_k}{m}\E\big[|\tV_\infty|\big],
\end{equation}
which together with part~$(a)$ proves the lower bound of part~$(b)$ when $k \ge 3$.

Now, using part~$(c)$ and $\cK(z)$ to bound $\E\big[|\tV_\infty|^2\big]=O(1)$ and
$s_k=O(m)$, respectively, it follows from~\eqref{eq:condprob:twoinV} that
\begin{equation}\label{eq:expectation:Rtwok}
 \E[\tR_2(k)]\le\frac{k^2 s_k}{m^2}\cdot\E\big[|\tV_\infty|^2\big]
 =O\bigg(\frac{k^2 s_k}{m^2}\bigg).
\end{equation}
Note that if $|\tV_\infty|<m^{1/2}$ then every edge of size 2 is (by definition) either
contained in or disjoint from $\tD_\infty$, and that otherwise $\tDelta_2=O(m)$, by $\cK(z)$.
Hence, combining \eqref{eq:expectation:Rtwok} with~\eqref{eq:used:twoedges}, and
using~\eqref{eq:mainlemma:application}, we obtain the case $k=2$ of part~$(b)$. Moreover,
combining \eqref{eq:expectation:Rtwok} with~\eqref{eq:used:threeedges}, and recalling that
$\sum_{k\ge2}k^2s_k=O(m)$, by $\cK(z)$, it follows that
\[
 \E\big[\tDelta'\big]=O\big(m^{-1}\big)+O\big(m^{-1/2}\big)=O\big(m^{-1/2}\big),
\]
which proves $(d)$.

It only remains to prove the upper bound in part~$(b)$ when $k\ge 3$.
By~\eqref{eq:used:threeedges:kfixed}, it is sufficient to show that the expected number of
edges of $\tS_k\setminus\tE_\infty$ that intersect $\tV_\infty$ in at least one vertex is at most
\[
 \bigg(1-\frac{2s_2}{m}+o(1)\bigg)^{-1}\frac{k s_k}{m}
\]
for each $k\ge 3$. This follows by \Lm{R:binomial} and part~$(a)$, since
\[
 \E\big[Z(k)\big]=\E\Big[\E\big[Z(k)\mid\tcF_\infty\big]\Big]
 =\big(1+o(1)\big)\frac{ks_k}{m}\E\big[|\tV_\infty|\big]
\]
for every $3\le k\le 4u_0$. This completes the proof of part~$(b)$, and therefore of the lemma.
\end{proof}

We are finally ready to deduce \Th{branching}. The key observation, which allows us to apply
\Th{compare}, is that each of the variables $D(z)$, $R_1(z)$, $\Delta_k(z)$ and $\Delta'(z)$
depends only on $\cF_z^+$ and the sub-matrix $R(I):=A\big[I\times[z,\pi(x)]\big]$, where $I$
is the set of rows that are removed in step~$z$. Note also that we can deduce from $\cF_z^+$ and
$R(I)$ whether or not $I$ is the set of removed rows. Indeed, in \Al{algorithm} we
only need to look at the entries of a row once we know we shall remove it. The event $\{D(z)=t\}$
(for example) is therefore a disjoint union of events of the form
$\big\{A\big[I\times [z,\pi(x)]\big]=R\big\}$ with $R$ having exactly $t$ rows.

It follows that we would be able to deduce \Th{branching} from Lemmas \ref{lem:branching:main}
and~\ref{lem:branching:rest} if we could restrict to events $\cE\in\cF_z^+$ and
$\big\{A\big[I\times [z,\pi(x)]\big]=R\big\}$ for which
\[
 \Prb\big(A\big[I\times [z,\pi(x)]\big]=R\mid\cE\big)
 =\big(1+o(1)\big)\Prb\big(\tA_\cE\big[I\times [z,\pi(x)]\big]=R\big).
\]
\Th{compare} provides us with such a bound as long as $|R|_1$ is not too large, since
$|I|\le |R|_1 \le |R|_2\le (|R|_1)^2$ for the matrices we shall be dealing with. 
Moreover, observe that if $I$ is the set of removed rows, then $|R(I)|_1=R_1(z)$, 
so we shall be able to use \Th{compare}
and \Lm{branching:main} to bound from above the probability that $|R|_1$ is large.

\begin{proof}[Proof of \Th{branching}]
We partition the probability space into three pieces, using the events
\[
 \cD_1=\big\{R_1(z)<u_0^{1/3}\big\},\quad
 \cD_2=\big\{u_0^{1/3}\le R_1(z)<u_0^2\big\}\quad\text{and}\quad
 \cD_3=\big\{R_1(z)\ge u_0^2\big\}.
\]
Let us fix $z\in [z_-,\pi(x)]$ and an event $\cE\in\cF_z^+$ of the form~\eqref{def:eventE} such
that $\cK(z)$ holds, $2s_2(z)\le (1-\eps_1)m(z)$ and $d(z)=1$, and say that an
$I\times [z,\pi(x)]$ matrix $R$ is \emph{$1$-acceptable} (with respect to~$\cE$) if it is
consistent with $\cD_1\cap\cE$ and if the event
$\big\{A\big[I\times[z,\pi(x)]\big]=R\big\}\cap\cE$ implies that $I$ is the set of removed rows
in step~$z$.\footnote{To be precise, an $I\times [z,\pi(x)]$ matrix $R$ with this latter property is said to be consistent with $\cD_1\cap\cE$ if and only if the column sums of $R$ satisfy $\sum_{i \in I} R_{ij} \le d_j$ for each $j \in [z,\pi(x)]$, and $|R|_1 < u_0^{1/3}$.} Note that $\cD_1$ is the disjoint union of the events
$\big\{A\big[I\times [z,\pi(x)]\big]=R\big\}$ over the family $\cU_1$ of all $1$-acceptable
matrices~$R$. We claim that, for every 1-acceptable matrix $R\in\cU_1$, we have
\begin{equation}\label{eq:1acceptable}
 \Prb\big(A\big[I\times [z,\pi(x)]\big]=R\mid\cE\big)
 =\big(1+o(1)\big)\Prb\big(\tA_\cE\big[I\times[z,\pi(x)]\big]=R\big).
\end{equation}
Indeed, this follows by applying \Th{compare} since
\[
 \frac{|I|+|R|_2}{u_0}=o(1),
\]
which holds because every row of $R$ is non-empty, so $|I|\le |R|_2\le (|R|_1)^2<u_0^{2/3}$,
and also every row sum of $R$ is at most $|R|_1=o(u_0)$.

We shall next prove that $\Prb(\cD_3\mid\cF_z,\,d(z)=1)\le z_0^{-20}$. To do so, let us
say that an $I\times [z,\pi(x)]$ matrix $R$ is \emph{$3$-acceptable} (with respect to~$\cE$) if
it is consistent with $\cE$ and if $I$ is the set of rows removed in \Al{algorithm} at the first
point at which $\cD_3$ is guaranteed to hold, i.e., the first point at which
$|R|_1\ge u_0^2$. Thus $\cD_3$ is the disjoint union of the events
$\big\{A\big[I\times [z,\pi(x)]\big]=R\big\}$ over the family $\cU_3$ of all $3$-acceptable
matrices~$R$. Now, by \Th{compare}, we have
\[
 \frac{\Prb\big(A\big[I\times [z,\pi(x)]\big]=R\mid\cE\big)}
 {\Prb\big(\tA_\cE[I\times [z,\pi(x)]\big]=R\big)}
 \le\exp\bigg(\frac{O\big(|I|+|R|_1\big)}{u_0}\bigg)=e^{O(u_0)}
\]
for any $3$-acceptable matrix~$R$, since $|I|\le |R|_1=O(u_0^2)$. Indeed, we have $|I|\le |R|_1$
(as before) since $R$ has no empty rows, and $|R|_1=O(u_0^2)$ since $I$ is minimal, and since
each row of $R$ contains at most $2u_0$ non-zero entries (by \Ob{onesinarow}), so each row can
increase $|R|_1$ by at most $O(u_0)$.

Now, by \Lm{branching:main}, we have
\[
 \sum_{R\in\cU_3}\Prb\big(\tA_\cE[I\times [z,\pi(x)]\big]=R\big)
 =\Prb\big(\tR_1\ge u_0^2\big)=O\big(e^{-\lambda u_0^2}\big),
\]
and hence, noting that $\log z_0=o(u_0^2)$, by~\eqref{eq:ulogu}, we obtain
\[
 \Prb\big(\cD_3\mid\cF_z,\,d(z)=1\big)\le\exp\big(O(u_0)-\lambda u_0^2\big)\le z_0^{-20},
\]
as claimed, which proves~\eqref{eq:Dplus:tailevent}. Note that this also implies that
\begin{equation}\label{eq:3acceptable}
 \E\big[R_1(z)^2\id_{\cD_3}\mid\cF_z,\,d(z)=1\big]\le z_0^{-9},
\end{equation}
since the event $\cK(z)$ implies that $R_1(z)^2\le(4u_0 z_0^5)^2\le z_0^{11}$.

Finally, let $\cU_2$ denote the family of (\emph{$2$-acceptable} with respect to~$\cE$)
$I\times [z,\pi(x)]$ matrices $R$ that are consistent with $\cD_2\cap\cE$ and are such that the event
$\big\{A\big[I\times [z,\pi(x)]\big]=R\big\}\cap\cE$ implies that $I$ is the set of removed rows
in step~$z$. By \Th{compare}, we have
\[
 \frac{\Prb\big(A\big[I\times [z,\pi(x)]\big]=R\mid\cE\big)}
 {\Prb\big(\tA_\cE[I\times [z,\pi(x)]\big]=R\big)}\le\exp\bigg(\frac{O(|R|_1)}{u_0}\bigg)
\]
for any $2$-acceptable matrix~$R$, since $|I|\le |R|_1$, and hence
\begin{align*}
 \Prb\big(\cD_2\mid\cF_z,\,d(z)=1\big)
 &\le\sum_{R\in\cU_2}\exp\bigg(\frac{O(|R|_1)}{u_0}\bigg)
  \Prb\big(\tA_\cE[I\times [z,\pi(x)]\big]=R\big)\\
 &\le\sum_{t=u_0^{1/3}}^{u_0^2}e^{O(t/u_0)}\Prb\big(\tR_1\ge t\big)
  \le\exp\bigg(-\frac{\lambda u_0^{1/3}}{2}\bigg),
\end{align*}
by \Lm{branching:main}. Since (by definition) $R_1(z)\le u_0^2$ if $\cD_2$ holds, it follows that
\begin{equation}\label{eq:2acceptable}
 \E\big[R_1(z)^2\cdot\id_{\cD_2}\mid\cF_z,\,d(z)=1\big]
 \le u_0^4\exp\bigg(-\frac{\lambda u_0^{1/3}}{2}\bigg)=o(1).
\end{equation}
Combining~\eqref{eq:2acceptable} with~\eqref{eq:1acceptable},~\eqref{eq:3acceptable} and
\Lm{branching:rest}, this completes the proof of parts $(a)$ and $(c)$ of \Th{branching}. Moreover,
for part $(d)$, we have
\begin{align*}
 \E\big[\Delta'(z)\big]
 &=\E\big[\Delta'(z)\id_{\cD_1}\big]+\E\big[\Delta'(z)\id_{\cD_2}\big]+\E\big[\Delta'(z)\id_{\cD_3}\big]\\
 &\le (1+o(1))\E\big[\tDelta'\big]+e^{O(u_0)}\E\big[\tDelta'\big]+z_0^{-9}=O\big(m(z)^{-1/3}\big),
\end{align*}
by \Lm{branching:rest}, since $u_0=o(\log m(z))$ and $m(z)\le N=z_0^{2+o(1)}$.

For part $(b)$, we need to take more care in the case that $s_k(z)$ is small. For $k\ge3$
we have
\[
 \E\big[\Delta_k(z)\id_{\cD_2}\big] \le \sum_{t = u_0^{1/3}}^{u_0^2} \E\big[\Delta_k^{(1)}(z)\id_{\{R_1(z)=t\}}\big] + \E\big[\Delta'(z)\big],
\]
and applying \Th{compare} to each $2$-acceptable matrix $R$ with $|R|_1 = t$ gives
\[ 
 \E\big[\Delta_k^{(1)}(z)\id_{\{R_1(z)=t\}}\big]\le\E\big[\tDelta_k^{(1)}e^{O(t/u_0)}\id_{\{\tR_1=t\}}\big]
 \le\frac{O(ks_k(z))}{m(z)}e^{-\lambda t/2}
 \]
for each $u_0^{1/3} \le t \le u_0^2$, where the second inequality follows by \Lm{branching:main}. The claimed bound now follows by~\eqref{eq:1acceptable},~\eqref{eq:3acceptable} and \Lm{branching:rest}, and using part~$(d)$ to bound $\E\big[\Delta'(z)\big]$. For $k=2$ we note that $\Delta_2(z)\le R_1(z)$, so by \Th{compare} and \Lm{branching:main},
\[
 \E\big[\Delta_2(z)\id_{\{R_1(z)=t\}}\big]\le \E\big[\tR_1 e^{O(t/u_0)}\id_{\{\tR_1=t\}}\big]
 \le\frac{O(s_2(z))}{m(z)}e^{-\lambda t/2}.
\]
Summing over $t\in[u_0^{1/3},u_0^2]$ gives $\E[\Delta_2(z)\id_{\cD_2}]=o(s_2(z)/m(z))$, so we are done as before.
\end{proof}

\section{Tracking the process when most columns are empty}\label{sec:big:z}

In this section we shall track $s_k(z)$ and $m(z)$ above the `critical' range $[z_-,z_+]$,
by showing that for $z\ge z_+$ there are few edges in $\cS_A(z)$ and as a consequence that
$m(z)\approx\eta\Lambda(z)z$. For $z\ge z_0^3$ these results follow almost immediately from
\Lm{skzero}, so for most of this section we shall be interested in the range $z\in[z_+,z_0^3]$.

Set $\tdelta:=(1-2\eps_1)\eta\delta$ and, recalling
Definition~\ref{def:eps}, define
\begin{equation}\label{def:sigmak}
 \sigma_k:=\frac{\eps(k,z_+)}{k!}\cdot\frac{\tdelta^{k-1}}{2k}
 =\frac{\eps_1^k}{\Lambda(z_+)}\cdot\frac{\tdelta^{k-1}}{2k}.
\end{equation}
We remark that the fact that $\sigma_k$ decreases only exponentially fast (as a function
of~$k$) will play an important role in the proofs of Lemmas \ref{lem:deltaSstar:bigz}
and~\ref{lem:maxstep:bigz}, below. Note that \Lm{m0} implies that
$m(z_+)\ge m_0(z_+)\ge\tdelta z_+$ with high probability, since $\Lambda(z_+)\ge\delta$.

Recall that, by \Df{MandTk}, $\cM^*(z)=\bigcap_{w\ge z}\cM(z)$
implies that $m(w)$ is well controlled for all $w\ge z$, and $\cT_k(z)$ implies
that $s_k(z)$ is well controlled.
Define $\cD(z)$ to be the event that $m(z)-m(z-1)\le u_0^2$ and set
\begin{equation}\label{eq:Devent}
 \cD^*(z):=\bigcap_{w=z+1}^{\pi(x)}\cD(w).
\end{equation}
Note that $\cD^*(z)$ does \emph{not} include the event $\cD(z)$.
We shall prove the following upper bound on $s_k(z)$ when $z\ge z_+$.

\begin{prop}\label{prop:largez}
 With high probability, $\cM^*(z_+)$ and $\cD^*(z_+)$ hold, and moreover
 \begin{equation}\label{eq:largez}
  s_k(z)\le\sigma_k m(z)
 \end{equation}
 for every\/ $k\ge 2$ and every\/ $z\in [z_+,\pi(x)]$.
\end{prop}

We remark that~\eqref{eq:largez} is extremely weak at the beginning of the process, but becomes
progressively stronger as time goes on (i.e., as $z$ decreases). We shall begin by showing that
at $z=z_+$ it implies the events $\cT_k(z_+)$.

\begin{cor}\label{cor:zp}
 If\/ $s_k(z_+)\le\sigma_k m(z_+)$ and\/ $m(z_+)\ge\tdelta z_+$, then
 \[
  s_k(z_+)\in\bigg(1\pm\frac{\eps(k,z_+)}{2}\bigg)\ts_k(z_+).
 \]
 In particular, with high probability\/ $\cT_k(z_+)$ holds for every\/ $k\ge 2$.
\end{cor}
Note that our bound on $s_k(z_+)$ is slightly stronger than necessary for $\cT_k(z_+)$ here.
This is because we shall require the stronger bound in \Sc{proof:tracking}. To prove
\Co{zp}, we first note the following observation, which will also be used later in
the proof of \Lm{deltas:bigz}.

\begin{obs}\label{obs:delta}
 For every\/ $k\ge2$, we have
 \[
  \frac{\eps_1^{k+1}k!}{3^k}\ge\Lambda(z_+),
 \]
 and hence $\eps(k,z_+)\ge 3^k/\eps_1$.
\end{obs}
\begin{proof}
Note that the expression $\eps_1^{k+1}k!/3^k$ is minimized by taking $k=3/\eps_1$
(which we have assumed is an integer), and that $k!\ge 2(k/e)^k$ for all $k\ge2$. Thus
\[
 \frac{\eps_1^{k+1}k!}{3^k}\ge 2\eps_1\bigg(\frac{3}{e\eps_1}\bigg)^{3/\eps_1}
 \bigg(\frac{\eps_1}{3}\bigg)^{3/\eps_1}=2\eps_1 e^{-3/\eps_1}\ge 2\delta
\]
by \eqref{def:small}. The first result follows since $\Lambda(z_+)=\delta+o(1)$, by~\eqref{eq:Lambda:delta:plusminus:littleoone}. It now follows immediately from Definition~\ref{def:eps} that 
\[
 \eps(k,z_+)=\frac{\eps_1^kk!}{\Lambda(z_+)}\ge \frac{3^k}{\eps_1},
\]
 as required.
\end{proof}

\begin{proof}[Proof of \Co{zp}, assuming \Pp{largez}]
The lower bound on $s_k(z_+)$ holds trivially as $\eps(k,z_+)\ge2$ by \Ob{delta}.
To prove the upper bound, observe that
\begin{equation}\label{eq:zp}
 \frac{s_k(z_+)}{\eps(k,z_+)}\le\frac{\sigma_k m(z_+)}{\eps(k,z_+)}
 =\frac{\tdelta^{k-1}m(z_+)}{k!\cdot 2k}
 \le\frac{m(z_+)}{2(k-1)k!}\bigg(\frac{m(z_+)}{z_+}\bigg)^{k-1},
\end{equation}
by~\eqref{def:sigmak} and the assumption that $m(z_+)\ge\tdelta z_+$. Thus,
by considering just the first term in the sum,
\[
 s_k(z_+)\le\frac{\eps(k,z_+)m(z_+)}{2k(k-1)}\sum_{\ell=k-1}^{\infty}\frac{1}{\ell!}
 \bigg(\frac{m(z_+)}{z_+}\bigg)^\ell=\frac{\eps(k,z_+)}{2}\cdot\ts_k(z_+),
\]
as required.

The last part follows as by \Pp{largez}, $s_k(z_+)\le\sigma_k m(z_+)$, and by \Lm{m0},
$m(z_+)\ge\tdelta z_+$, with high probability.
\end{proof}

Note that $m(z)-m(z-1)=D(z)\le R_1(z)$ when $d(z)=1$ and $m(z)-m(z-1)=0$ otherwise. So by
\Th{branching}, if $\cK(z)$ holds and $2s_2(z)\le(1-\eps_1)m(z)$, then
\begin{equation}\label{eq:cDlikely}
 \Prb\big(\cD(z)^c\mid\cF_z\big)\le\Prb\big(\cD(z)^c\mid\cF_z,\,d(z)=1\big)\le z_0^{-20}.
\end{equation}
Let
\[
 s^*_k(z):=\begin{cases}
 \ds\frac{s_k(z)}{\sigma_k m(z)}&\text{if $\cD^*(z)$ holds;}\\[8pt]
 s^*_k(z+1)&\text{otherwise;}
 \end{cases}
\]
for each $k\ge 2$ and $z\in [z_+,\pi(x)]$, and define
\begin{equation}\label{def:L}
 \cL^*(z):=\cD^*(z)\cap\bigcap_{w\in[z,\pi(x)]}
 \Big(\cQ(w)\cap\bigcap_{k\ge 2}\big\{s^*_k(w)\le 1\big\}\Big).
\end{equation}
for each $z\ge z_+$. We shall in fact show that the event $\cL^*(z_+)$ holds with high probability,
which will be sufficient to prove \Pp{largez}. Let us quickly note, for future reference, that
the event $\cL^*(z)$ implies that the conditions of \Th{branching} are satisfied.

\begin{lemma}\label{lem:largezK}
 Let\/ $z\in [z_+,\pi(x)]$. If\/ $\cD^*(z)$ holds and\/ $s^*_k(z)\le 1$ for every\/ $k\ge 2$, then
 \begin{equation}\label{eq:largezK}
  \sum_{k\ge 2}2^k s_k(z)\le\eps_1 m(z).
 \end{equation}
 In particular, if\/ $\cL^*(z)$ holds, then\/ $2s_2(z)\le\eps_1 m(z)$
 and\/ $\cK(z)$ holds.
\end{lemma}
\begin{proof}
Note first that if $\cD^*(z)$ holds and $s^*_k(z)\le 1$ for every $k\ge 2$, then
\[
 \sum_{k\ge 2}2^k s_k(z)\le\sum_{k\ge 2}2^k\sigma_k m(z)
 =\sum_{k\ge 2}\frac{(2\eps_1)^k\tdelta^{k-1}}{2k\Lambda(z_+)}m(z)\le\eps_1 m(z),
\]
since $\tdelta\le\delta\le\Lambda(z_+)\le1$ and $\eps_1\le 1/16$, which proves~\eqref{eq:largezK}. It follows that $\cL^*(z)$ also implies~\eqref{eq:largezK}, which in turn implies that $2s_2(z) \le \eps_1 m(z)$ and that~\eqref{def:K} holds. Since $\cL^*(z)$ also implies $\cQ(z)$, this is sufficient to show that $\cK(z)$ holds, as claimed.
\end{proof}

In order to apply the method of self-correcting martingales, we shall need (for each $k\ge 2$)
an estimate of the expected change in $s^*_k(z)$ as well as a bound on the largest jump in $s^*_k(z)$.
To be precise, we shall prove the following two lemmas.

\begin{lemma}\label{lem:deltaSstar:bigz}
 Let\/ $z\in [z_+,z_0^3]$. If\/ $\cL^*(z)$ holds, then
 \[
  \E\big[s^*_k(z-1)-s^*_k(z)\mid\cF_z\big]\le (k-1)p_z(x)\big(-s^*_k(z)+16\eps_1\big)
 \]
 for every\/ $k\ge 2$.
\end{lemma}

\begin{lemma}\label{lem:maxstep:bigz}
 Let\/ $z=z_0^\beta\in [z_+,z_0^3]$. If\/ $\cL^*(z)$ holds, then
 \[
  |s^*_k(z-1)-s^*_k(z)|\le z_0^{-2+1/\beta+\eps_1}
 \]
 for every\/ $k\ge 2$.
\end{lemma}

We shall first deduce \Lm{maxstep:bigz}, which is relatively straightforward, and then
\Lm{deltaSstar:bigz}, the proof of which will require a little more work. The first step is to
recall the following simple facts about~$m(z)$.

\begin{obs}\label{obs:m:rough}
 Let\/ $z\in [z_+,\pi(x)]$. If\/ $\cM(z)$ holds, then\/ $z_0^{1+o(1)}\le m(z)\le z_0^{2+o(1)}$.
 More precisely, if\/ $z=z_0^\beta$ then\/ $m(z)=z_0^{2-1/\beta+o(1)}$.
\end{obs}
\begin{proof}
Recall first that if $\cM(z)$ holds then $m(z) = \Theta(\Lambda(z)z)$, by Observation~\ref{obs:m:bounded} and~\eqref{eq:lowerbound:m:using:M}. Moreover, $\Lambda(z) = z_0^{2 - \beta - 1/\beta+o(1)}$ for $z = z_0^\beta$, by \Co{Lambda:rough}, so $m(z) = z_0^{2 - 1/\beta+o(1)}$, as claimed. Since $\beta\ge 1$ for every $z_+\ge z_0$, the bounds $z_0^{1+o(1)}\le m(z)\le z_0^{2+o(1)}$ follow.
\end{proof}

\begin{proof}[Proof of \Lm{maxstep:bigz}]
Recall first that $\cL^*(z)$ implies $\cQ(z)$, which implies that $s_k(z)=0$ if $k>4u_0$ (see
\Df{Q}), so we may assume that $k\le 4u_0$.
Also, by definition of $s^*_k(z)$, we may assume $\cD(z)$ holds.
Recall that $u_0=o(\log z_0)$, by~\eqref{eq:ulogu},
and note that therefore
\begin{equation}\label{eq:sigma:small}
 \sigma_k=\frac{\eps_1^k}{\Lambda(z_+)}\cdot\frac{\tdelta^{k-1}}{2k}=e^{O(k)}
 =z_0^{o(1)}=m(z)^{o(1)}.
\end{equation}
By \Ob{onesinarow}, each row of $A$ contains at most $2u_0$ non-zero entries to the right of~$z$,
and hence $|s_k(z-1)-s_k(z)|\le 2u_0|m(z-1)-m(z)|$ if $d(z)=1$. On the other hand,
$|s_k(z-1)-s_k(z)|\le 1$ and $m(z-1)=m(z)$ if $d(z)\ne1$. Thus $\cD(z)$ implies
\begin{equation}\label{eq:schange:vs:mchange}
 |s_k(z-1)-s_k(z)|\le 2u_0^3.
\end{equation}

Finally, note that, by the definition of $s_k^*(z)$, we have
\begin{align*}
 \big|\sigma_k\big(s^*_k(z-1)-s^*_k(z)\big)\big|
 &=\bigg|\frac{s_k(z-1)}{m(z-1)}-\frac{s_k(z)}{m(z)}\bigg|\\
 &\le\frac{|s_k(z-1)-s_k(z)|}{m(z-1)}+\frac{s_k(z)|m(z)-m(z-1)|}{m(z)m(z-1)}\\
 &\le\frac{2u_0^3}{m(z-1)}+\frac{s_k(z)u_0^2}{m(z)m(z-1)}
 \le\frac{3u_0^3}{m(z)-u_0^2}=m(z)^{-1+o(1)},
\end{align*}
since $s_k(z)\le\eps_1 m(z)$, by $\cL^*(z)$ and \Lm{largezK}, and since
$u_0\le\log z_0$ and $m(z)\ge z_0^{1+o(1)}$.
The lemma now follows from \eqref{eq:sigma:small} and \Ob{m:rough} as $\cL^*(z)$ implies~$\cM(z)$.
\end{proof}

We shall now prove \Lm{deltaSstar:bigz}. The first step is to note the following immediate
consequence of \Th{branching} and \Lm{largezK}. The key observation is that since
$s_2(z)$ is small, if $d(z)=1$ then we are likely to only remove a single row.

\begin{lemma}\label{lem:branching:bigz}
 Let\/ $z\in [z_+,\pi(x)]$. If\/ $\cL^*(z)$ holds, then
 \[
  \E\big[ \big( m(z)-m(z-1) \big)\id_{\cD(z)}\mid\cF_z,\,d(z)=1\big]\in 1\pm 2\eps_1.
 \]
\end{lemma}
\begin{proof}
By \Lm{largezK}, $2s_2(z)\le\eps_1 m(z)$. Thus by \Th{branching}, and recalling that
$D(z)=m(z)-m(z-1)$, it follows that
\begin{align*}
 \E\big[(m(z)-m(z-1))\id_{\cD(z)}\mid\cF_z,\,d(z)=1\big]
 &=\bigg(1-\frac{2s_2(z)}{m(z)}+o(1)\bigg)^{-1}+O(m(z)z_0^{-20})\\
 &\in 1\pm 2\eps_1
\end{align*}
by \eqref{eq:cDlikely}, as claimed.
\end{proof}

Next, we shall calculate the expected change in $s_k(z)$.

\begin{lemma}\label{lem:deltas:bigz}
 Let\/ $z\in [z_+,\pi(x)]$. If\/ $\cL^*(z)$ holds, then
 \begin{equation}\label{eq:deltastildek:bigz}
  \E\big[ \big( s_k(z-1)-s_k(z) \big)\id_{\cD(z)}\mid\cF_z\big]
  \le k\sigma_k\big(-s^*_k(z)+5\eps_1\big)\Prb\big(d(z)=1\mid\cF_z\big)
 \end{equation}
 for every\/ $k\ge 2$.
\end{lemma}

Recall from \Df{DandDelta} that $\Delta_k(z)=|S_k(z)\setminus S_k(z-1)|$ denotes the number
of edges of size $k$ that contain a vertex removed in step~$z$, and $\Delta'(z)$ denotes the
number of edges of size at least three that have at least two vertices removed in step~$z$.
We shall use the following simple observation to prove \Lm{deltas:bigz}, and again in
\Sc{critical}.

\begin{obs}\label{obs:change:ins:using:deltas}
 For every\/ $z\in [z_-,\pi(x)]$ and\/ $k\ge 2$, we have
 \[
  s_k(z-1)-s_k(z)\in \id_{\{d(z)=k\}}
  +\big(-\Delta_k(z)+\Delta_{k+1}(z)\pm\Delta'(z)\big)\id_{\{d(z)=1\}}.
 \]
\end{obs}

\begin{proof}
Recall from \Al{algorithm} that if $d(z)\notin\{1,k\}$ then $S_k(z-1)=S_k(z)$,
while if $d(z)=k$ then $S_k(z-1)=S_k(z)\cup\{z\}$. If $d(z)=1$, then exactly $\Delta_k(z)$
edges are removed from $S_k(z)$, and $\Delta_{k+1}(z)\pm\Delta'(z)$ edges are added to $S_k(z)$,
as required. Note that the $\Delta'(z)$ bounds both the number of $(k+1)$-edges that lose more
than one vertex, as well as edges of size at least $k+2$ that lose enough vertices to become
$k$-edges.
\end{proof}

\begin{proof}[Proof of \Lm{deltas:bigz}]
Note first that, since $\cL^*(z)$ holds, we have $s^*_k(z)\le 1$ for each $k\ge 2$,
$2s_2(z)\le\eps_1 m(z)$ and $\cK(z)$ holds (by \Lm{largezK}), and $s_k(z)=\sigma_k m(z)s^*_k(z)$
(since $\cL^*(z)$ implies $\cD^*(z)$). Thus, by \Th{branching},
\begin{align*}
 \E\big[\Delta_k(z)\mid\cF_z,\,d(z)=1\big]
 &=\bigg(1-\frac{2s_2(z)}{m(z)}+o(1)\bigg)^{-1}\frac{ks_k(z)}{m(z)}+O(m(z)^{-1/3})\\
 &\in\big(1\pm 2\eps_1\big)k\sigma_k s^*_k(z)+O(m(z)^{-1/3})
\end{align*}
for each $k\ge 2$, and also
\[
 \E\big[\Delta'(z)\mid\cF_z,\,d(z)=1\big]=O(m(z)^{-1/3}).
\]
Note also that \eqref{def:sigmak} implies
\[
 (k+1)\sigma_{k+1}=\eps_1\tdelta\cdot k\sigma_k.
\]
By \Ob{change:ins:using:deltas}, and using~\eqref{eq:cDlikely} and~\eqref{eq:sigma:small},
it follows that
\begin{align*}
 &\E\big[\big(s_k(z-1)-s_k(z)\big)\id_{\cD(z)}\mid\cF_z,\,d(z)=1\big]\\
 &\hspace{3cm}\le -(1-2\eps_1)k\sigma_k s^*_k(z)
 +(1+2\eps_1)(k+1)\sigma_{k+1}s^*_{k+1}(z)+O\big(m(z)^{-1/3}\big)\\
 &\hspace{3cm}\le k\sigma_k\big(-s^*_k(z)+2\eps_1+(1+2\eps_1)\eps_1\tdelta+o(1)\big)\\
 &\hspace{3cm}\le k\sigma_k\big(-s^*_k(z)+3\eps_1\big).
\end{align*}
Finally, by \Lm{dkvsd1:bigz} and \Ob{delta}, we have
\[
 \frac{\Prb\big(d(z)=k\mid\cF_z\big)}{\Prb\big(d(z)=1\mid\cF_z\big)}
 \le\frac{(2\delta\eta)^{k-1}}{k!}
 \le\frac{(3\tdelta)^{k-1}}{k!}
 \le\frac{\eps_1^{k+1}\tdelta^{k-1}}{\Lambda(z_+)}
 =2\eps_1 k\sigma_k
\]
for every $z\in [z_+,\pi(x)]$ and every $k\ge 2$. Hence
\[
 \E\big[(s_k(z-1)-s_k(z))\id_{\cD(z)}\mid\cF_z\big]
 \le k\sigma_k\big(-s^*_k(z)+5\eps_1\big)\Prb\big(d(z)=1\mid\cF_z\big),
\]
as required.
\end{proof}

\Lm{deltaSstar:bigz} will now follow as a straightforward consequence of
Lemmas \ref{lem:branching:bigz} and~\ref{lem:deltas:bigz}.

\begin{proof}[Proof of \Lm{deltaSstar:bigz}]
As in the proof of \Lm{maxstep:bigz}, we may assume that $k\le 4u_0$, and hence that
$\sigma_k=z_0^{o(1)}$, by~\eqref{eq:sigma:small}. Observe that
\begin{align}
 \sigma_k\big(s^*_k(z-1)-s^*_k(z)\big)
 &=\bigg(\frac{s_k(z-1)}{m(z-1)}-\frac{s_k(z)}{m(z)}\bigg)\id_{\cD(z)}\nonumber\\
 &=\frac{m(z)}{m(z-1)}\bigg(\frac{s_k(z-1)-s_k(z)}{m(z)}
 +\frac{s_k(z)\big(m(z)-m(z-1)\big)}{m(z)^2}\bigg)\id_{\cD(z)}.\label{eq:deltasstar:rewrite}
\end{align}
Since $\cL^*(z)$ implies $\cK(z)$, by \Lm{largezK}, it follows from Lemmas~\ref{lem:dkvsd1:bigz} and~\ref{lem:branching:bigz} that
\[
 \E\big[\big(m(z)-m(z-1)\big)\id_{\cD(z)}\mid\cF_z\big]\le(1+3\eps_1)m(z)p_z(x),
\]
and since $\cL^*(z)$ implies that $s_k(z)=\sigma_k m(z)s^*_k(z)$ and $s^*_k(z)\le 1$, it follows that
\begin{equation}\label{eq:exp:s:times:mchange}
 \E\big[s_k(z)\big(m(z)-m(z-1)\big)\id_{\cD(z)}\mid\cF_z\big]
 \le\big(s^*_k(z)+3\eps_1\big)\sigma_k m(z)^2 p_z(x).
\end{equation}
Similarly, by Lemmas~\ref{lem:dkvsd1:bigz} and~\ref{lem:deltas:bigz}, we have
\begin{equation}\label{eq:exp:schange:rewrite}
 \E\big[ \big( s_k(z-1)-s_k(z) \big)\id_{\cD(z)}\mid\cF_z\big]
 \le\big(-k s^*_k(z)+6\eps_1 k\big)\sigma_k m(z)p_z(x).
\end{equation}
It remains to bound
\begin{equation}\label{eq:sstar:remainder}
 \bigg(\frac{m(z)-m(z-1)}{m(z-1)}\bigg)
 \bigg(\frac{s_k(z-1)-s_k(z)}{m(z)}+\frac{s_k(z)\big(m(z)-m(z-1)\big)}{m(z)^2}\bigg)
\end{equation}
under the assumption that $\cD(z)$ holds. To do so, recall that $s_k(z)\le\eps_1 m(z)$, by \Lm{largezK}, 
that $m(z)=z_0^{2-1/\beta+o(1)}$, by \Ob{m:rough}, where $z=z_0^\beta$, and that $1\le\beta\le 3$, 
since $z\in [z_+,z_0^3]$. It follows from \eqref{eq:schange:vs:mchange} that~\eqref{eq:sstar:remainder} is at most
\[
 \bigg(\frac{u_0^2}{m(z-1)}\bigg)\bigg(\frac{2u_0^3}{m(z)}+\frac{s_k(z)u_0^2}{m(z)^2}\bigg)
 \le\frac{3u_0^5}{m(z)(m(z)-u_0^2)}\le z_0^{-4+2/\beta+\eps_1}.
\]
Now $p_z(x) = z^{-1+o(1)} = z_0^{-\beta+o(1)}$, by \Co{pj}, and $\sigma_k=z_0^{o(1)}$, by~\eqref{eq:sigma:small}. Thus
\[
 z_0^{-4+2/\beta+\eps_1}\le z_0^{-\beta-\eps_1}\le \eps_1\sigma_k p_z(x),
\]
where we have used the fact that $\beta+2/\beta\le\frac{11}{3}<4-2\eps_1$ for $1\le \beta\le 3$.
Hence, using \eqref{eq:deltasstar:rewrite},
\eqref{eq:exp:s:times:mchange} and~\eqref{eq:exp:schange:rewrite}, we have
\begin{align*}
 \E\big[s^*_k(z-1)-s^*_k(z)\mid\cF_z\big]
 &\le\big(-(k-1)s^*_k(z)+6\eps_1 k+4\eps_1\big)p_z(x)\\
 &\le(k-1)p_z(x)\big(-s^*_k(z)+16\eps_1\big)
\end{align*}
for every $k\ge 2$, as required.
\end{proof}

We can now deduce the main result of the section.

\begin{proof}[Proof of \Pp{largez}]
We shall show that $\cL^*(z_+)$ holds with high probability, which implies both the events
$\cM^*(z_+)$ and $\cD^*(z_+)$, and the inequality~\eqref{eq:largez}. The idea is (roughly speaking) to bound,
for each $a\ge z_+$, the probability that $a$ is maximal such that $\cL^*(a)$ does not hold.
The main step is the proof of the following claim.

\claim{Claim:} For each $z_+\le a\le z_0^3$ and $k\ge 2$, we have
\[
 \Prb\Big(\cL^*(a+1)\cap\big\{s^*_k(a)>1\big\}\cap\big\{s^*_k(z_0^3)\le 3/4\big\}\Big)
 \le z_0^{-20}.
\]

\begin{proof}[Proof of claim]
For each $z_+\le a<b\le z_0^3$ and $k\ge 2$, let us define $\cU_k(a,b)$ to be the event that the
following all occur:
\begin{itemize}
\item[$(a)$] $s^*_k(a)>1$,
\item[$(b)$] $s^*_k(z)>3/4$ for every $a<z<b$,
\item[$(c)$] $s^*_k(b)\le 3/4$,
\item[$(d)$] $\cL^*(a+1)$ holds.
\end{itemize}
Note that if $\cL^*(a+1)$ holds, $s^*_k(a)>1$ and $s^*_k(z_0^3)\le 3/4$,
then the event $\cU_k(a,b)$ occurs for some (unique)~$b$, $a<b\le z_0^3$.
By the union bound, it will therefore suffice to prove that
\[
 \Prb\big(\cU_k(a,b)\big)\le z_0^{-23}
\]
for every $z_+\le a<b\le z_0^3$.

By \Lm{maxstep:bigz} we may assume $s^*_k(b)\ge 3/4-\eps_1$, as otherwise $s^*_k(b-1)\le 3/4$
and so $\cU_k(a,b)$ is impossible. For each $t\in\{0,\dots,b-a\}$, define
\[
 X_t:=\begin{cases}
  s^*_k(b-t)-s^*_k(b),&\text{if }X_{t-1}\ge0\text{ or }t=0;\\
  X_{t-1},&\text{otherwise.}
 \end{cases}
\]
We claim that $X_t$ is a super-martingale with respect to the filtration $(\cF_{b-t})_{t = 0}^{b-a}$. Indeed,
if $X_t<0$ then $X_{t+1}=X_t$ and if $X_t\ge0$ then
\[
 \E\big[X_{t+1}-X_t\mid\cF_{b-t}\big]\le(k-1)p_{b-t}(x)\big(-s^*_k(b-t)+16\eps_1\big)\le 0
\]
by \Lm{deltaSstar:bigz}, since $X_t\ge0$ and \eqref{def:small} imply
$s^*_k(b-t)\ge s^*_k(b)\ge 3/4-\eps_1\ge 16\eps_1$. Write $c_t=z_0^{-2+1/\beta+\eps_1}$,
where $b-t=z_0^\beta$. Then
\[
 |X_{t+1}-X_t|\le c_t,
\]
by \Lm{maxstep:bigz}. Also,
\[
 \sum_{t=0}^{b-a-1}c_t^2
 \le\sum_{z=z_+}^{z_0^3}z_0^{-4+2/\beta+2\eps_1}
 \le\sum_{z=z_+}^{z_0^3}z_0^{-4/3+2\eps_1}z^{-2/3}
 \le z_0^{-1/3+3\eps_1},
\]
where $\beta=\beta(z)$ is defined by $z=z_0^\beta$,
and the second inequality holds as $2/\beta\le (8-2\beta)/3$
for $1\le\beta\le 3$. But $\cU_k(a,b)$ implies that $X_{b-a}>1/4$, so by the Azuma--Hoeffding inequality we obtain
\[
 \Prb\big(\cU_k(a,b)\big)\le\Prb\big(X_{b-a}>1/4\big)
 \le\exp\big(-z_0^{1/3-4\eps_1}\big)\le z_0^{-23},
\]
as claimed.
\end{proof}

To complete the proof of the proposition, we will show that with high probability there does not
exist $z\in [z_+,\pi(x)]$ such that $\cL^*(z+1)$ holds but $\cL^*(z)$ does not hold.
(Here $\cL^*(\pi(x)+1)$ holds vacuously.)
Note first that, by \Lm{largezK} and \eqref{eq:cDlikely}, with high probability there does not
exist $z\in [z_+,\pi(x)]$ such that $\cL^*(z+1)$ holds but $\cD(z+1)$ does not; since $\cD^*(z+1)$
and $\cD(z+1)$ imply $\cD^*(z)$, we may assume that $\cD^*(z)$ holds.

Now consider the condition $s^*_k(z)\le 1$. \Lm{skzero} implies that with high probability we have
$s_2(z)+s_3(z)\le z_0^{-1/2} m(z)$ for every $z\ge z_0^3$, and $s_k(z)=0$ for every $k\ge 4$ and
every $z\ge z_0^3$. Since $\sigma_k=\Theta(1)$ for $k\in\{2,3\}$, it follows that with high
probability~\eqref{eq:largez} holds, and moreover $s^*_k(z)\le 3/4$, for all $z\ge z_0^3$ and every
$k\ge 2$. But if $s^*_k(z_0^3)\le 3/4$, then the claim implies that with high probability there does
not exist $z\in[z_+,z_0^3]$ and $k\ge2$ such that $\cL^*(z+1)$ and $s^*_k(z)>1$.

We next consider the event $\cM(z)$. Note first that, by \Lm{largezK}, and using the assumptions that $\cD^*(z)$ holds and that $s^*_k(z)\le 1$ for every $k\ge 2$, we have
\begin{equation}\label{eq:mm0}
 m_0(z)\le m(z)\le m_0(z)+\sum_{k\ge 2}ks_k(z)\le m_0(z)+\eps_1 m(z),
\end{equation}
since the number of non-isolated vertices is at most the sum of the degrees. Recall that with high
probability we have $m_0(z)\in(1\pm\eps_1)\eta\Lambda(z)z$ for every $z\in[z_-,\pi(x)]$, by \Lm{m0}.
Assuming this holds, it follows from~\eqref{eq:mm0} that
\[
 (1-\eps_1)\eta\Lambda(z)
 \le \frac{m(z)}{z}\le \frac{1+\eps_1}{1-\eps_1}\eta\Lambda(z).
\]
Recalling from~\eqref{def:zpm} and~\eqref{eq:Lambda:delta:plusminus:littleoone} that $\Lambda(z) \le \delta+o(1)$ when $z\ge z_+$, and that
\[
w(1-w) \le we^{-w} \le we^{-\Ein(w)} \le w
\]
for $0\le w\le 1$, since $0 \le \Ein(w)\le w$, by~\eqref{def:Ein}, we obtain
\[
 (1-2\delta\eta)(1-\eps_1)\eta\Lambda(z)\le
 \frac{m(z)}{z}e^{-\Ein(m(z)/z)}\le \frac{1+\eps_1}{1-\eps_1}\eta\Lambda(z).
\]
This implies that
\begin{equation}\label{eq:m:largez}
 \frac{m(z)}{z} e^{-\Ein(m(z)/z)} \in \big(1\pm 3\eps_1\big)\eta\Lambda(z),
\end{equation}
which implies that $\cM(z)$ holds. As $\cL^*(z+1)$ implies $\cM^*(z+1)$, this implies $\cM^*(z)$
holds. However, by \Lm{easy}, with high probability there does not exist $z\in [z_+,\pi(x)]$ such
that $\cM^*(z)$ holds but $\cQ(z)$ does not. It follows that with high probability, there does not
exist $z\ge z_+$ such that $\cL^*(z+1)$ but $\cL^*(z)$ fails to hold, and the proof is complete.
\end{proof}

\section{Tracking the process in the critical range}\label{sec:critical}

In the next two sections we shall track $s_k(z)$ and $m(z)$ in the `critical' range $[z_-,z_+]$.
The main aim of this section is to prove two lemmas corresponding to
Lemmas~\ref{lem:deltaSstar:bigz} and~\ref{lem:maxstep:bigz} from the previous section. We shall
need these lemmas in \Sc{proof:tracking} in order to track $s_k(z)$ using the method of
self-correcting martingales.

The first step is to define the event we shall use to prove our key lemmas.
Recall that $\cD(z)$ denotes the event that $m(z)-m(z-1)\le u_0^2$ and
$\cD^*(z)$ denotes the event that $\cD(w)$ holds for all $w>z$.
Now define, for each $z\in [z_-, z_+]$ and each $k\ge 2$,
\begin{equation}\label{def:sstar:critical}
 s^*_k(z):=\begin{cases}
 \ds\frac{s_k(z)-\ts_k(z)}{\eps(k,z)\ts_k(z)},&\text{if $\cD^*(z)$ holds;}\\[8pt]
 s^*_k(z+1),&\text{otherwise.}
 \end{cases}
\end{equation}
Recall that $\cT_k(z)$ denotes the event that $s_k(z)\in(1\pm\eps(k,z))\ts_k(z)$ and define
\[
 \cT^*(z):=\cD^*(z)\cap\bigcap_{w=z}^{z_+}\bigg(\cQ(w)\cap\bigcap_{k=2}^{4u_0}\cT_k(w)\bigg)
\]
for each $z\in [z_-, z_+]$. Note that $\cQ(z)$ implies $\cT_k(z)$ for every $k\ge 4u_0$,
since for such $k$ we have $\eps(k,z)>1$ (since $u_0=\omega(1)$), and $s_k(z)=0$ (see \Df{Q}).

In this section we shall prove the following two lemmas.

\begin{lemma}\label{lem:deltaSstar}
 Let\/ $z\in [z_-, z_+]$. If\/ $\cT^*(z)$ holds, then
 \[
  \E\big[s^*_k(z-1)-s^*_k(z)\mid\cF_z\big]\in - \/ \frac{k-1}{z}\bigg( s^*_k(z)\pm\frac{1}{2}\bigg)
 \]
 for every\/ $2 \le k \le 4u_0$.
\end{lemma}

\begin{lemma}\label{lem:maxstep}
 Let\/ $z\in [z_-, z_+]$. If\/ $\cT^*(z)$ holds, then
 \[
  |s^*_k(z-1)-s^*_k(z)|\le z_0^{-1+\eps_1}
 \]
 for every\/ $2 \le k \le 4u_0$.
\end{lemma}

The proofs of these two lemmas are, in outline, similar to those of
Lemmas~\ref{lem:deltaSstar:bigz} and~\ref{lem:maxstep:bigz}, but the calculation is more delicate
in the critical range, and as a consequence the details are somewhat more complicated. We begin
by noting that $\cT^*(z)$ implies that the conditions of \Th{branching} are satisfied, and so
in particular that \eqref{eq:cDlikely} holds.

\begin{lemma}\label{lem:Kcritical}
 Let\/ $z\in [z_-,z_+]$. If\/ $\cT^*(z)$ holds, then\/ $2s_2(z)\le (1-\eps_1)m(z)$
 and\/ $\cK(z)$ holds.
\end{lemma}
\begin{proof}
Since $\cT^*(z)$ implies that $\cQ(z)$ holds and that $\cT_k(z)$ holds for every $k\ge 2$, it
follows from \Lm{track:K} that $\cK(z)$ holds. To prove the bound on $s_2(z)$, note that
$\cT_2(z)$ implies that
\begin{equation}\label{eq:lems2:firstbound}
 2s_2(z)\le 2\big(1+\eps(2,z)\big)\ts_2(z)
 =\bigg(1+\frac{2\eps_1^2}{\Lambda(z)}\bigg)m(z)\big(1-e^{-m(z)/z}\big),
\end{equation}
by~\eqref{eq:ts2} and \Df{eps}. Since
\[
 1-e^{-m(z)/z}\le\frac{m(z)}{z}\le C_0\Lambda(z),
\]
by \Ob{m:bounded}, it follows that for $\Lambda(z)\le 2\eps_1$,
\[
 2s_2(z)\le\big(C_0\Lambda(z)+2\eps_1^2 C_0\big)m(z)\le 2\eps_1 C_0(1+\eps_1)m(z)\le\frac{m(z)}{2}
\]
since $\eps_1 C_0\le \eps_1e^{C_0}<1/16$ by~\eqref{def:small}. On the other hand,
if $2\eps_1\le\Lambda(z)\le 1$ then
\[
 2s_2(z)\le (1+\eps_1)\big(1-e^{-C_0}\big)m(z)\le (1+\eps_1)(1-2\eps_1)m(z)\le (1-\eps_1)m(z),
\]
as required, as $e^{-C_0}\ge2\eps_1$.
\end{proof}

\subsection{The expected change in $m(z)$}

The next step is to use \Th{branching} to bound the expected number of vertices removed in
each step. Recall that if $\cT^*(z)$ holds, then the average degree in the graph $S_2(z)$ is close
to $2\ts_2(z)/m(z)=1-e^{-m(z)/z}$, and so the expected size of $D(z)=m(z)-m(z-1)$ should be about
$e^{m(z)/z}\cdot\Prb(d(z)=1)\approx m(z)/z$, by \Lm{dz}. The following lemma makes this precise.

\begin{lemma}\label{lem:branching:critical}
 Let\/ $z\in [z_-,z_+]$. If\/ $\cT^*(z)$ holds, then
 \[
  \E\big[m(z)-m(z-1)\mid\cF_z\big]=\big(1+\gamma(z)+o(1)\big)\frac{m(z)}{z}
 \]
 and
 \[
  \E\big[ \big( m(z)-m(z-1) \big)\id_{\cD(z)}\mid\cF_z\big]=\big(1+\gamma(z)+o(1)\big)\frac{m(z)}{z},
 \]
 where\/ $\gamma(z)$ is defined by
 \begin{equation}\label{def:gamma}
  \gamma(z):=\frac{\eps(2,z)s^*_2(z)\big(e^{m(z)/z}-1\big)}
  {1-\eps(2,z)s^*_2(z)\big(e^{m(z)/z}-1\big)}.
 \end{equation}
\end{lemma}

\begin{obs}\label{obs:gamma:small}
 Let\/ $z\in [z_-, z_+]$. If\/ $\cT^*(z)$ holds, then\/ $|\gamma(z)|\le\eps_1$.
\end{obs}
\begin{proof}
Since $m(z)/z\le C_0\Lambda(z)$, by \Ob{m:bounded}, it follows that
\[
 \eps(2,z)\big(e^{m(z)/z}-1\big)\le\frac{2\eps_1^2}{\Lambda(z)}\big(e^{C_0\Lambda(z)}-1\big).
\]
Now $(e^x-1)/x$ is an increasing function of $x$ and $\Lambda(z)\le 1$, so
\[
 \eps(2,z)\big(e^{m(z)/z}-1\big)\le 2\eps_1^2(e^{C_0}-1)\le\eps_1/8
\]
by \eqref{def:small}. The result follows from the definition~\eqref{def:gamma}
of $\gamma(z)$ and the fact that $\cT^*(z)$ implies $\cT_2(z)$, so $|s^*_2(z)|\le1$.
\end{proof}

In the proof below, and also several times later in the section, we will need the fact that $m(z) = \Theta(z)$ uniformly in $z \in [z_-,z_+]$, which follows from $\cM(z)$ (and hence from $\cT^*(z)$) by Observation~\ref{obs:m:bounded}, and by~\eqref{eq:Lambda:delta:plus:littleoone} and~\eqref{eq:lowerbound:m:using:M}. 

\begin{proof}[Proof of \Lm{branching:critical}]
Note first that $\cT^*(z)$ implies that $\cK(z)$ and $\cM(z)$ hold, by \Lm{Kcritical} and \Df{Q}.
It follows that $\Prb(d(z)=1\mid\cF_z)=(1+o(1))\frac{m(z)}{z}e^{-m(z)/z}=\Theta(1)$, by \Lm{dz}
and since $m(z) = \Theta(z)$. Recall also that $D(z)=m(z)-m(z-1)=0$ if $d(z)\ne 1$. By \Th{branching},
\Lm{Kcritical}, and \eqref{eq:cDlikely} it follows that
\begin{align*}
 \E\big[D(z)\id_{\cD(z)}\mid\cF_z\big]
 &=\Big(\E\big[D(z)\mid\cF_z,\,d(z)=1\big] - O\big(m(z)z_0^{-20}\big)\Big)
 \Prb\big(d(z)=1\mid\cF_z\big)\\
 &=\bigg(1-\frac{2s_2(z)}{m(z)}+o(1)\bigg)^{-1}\frac{m(z)}{z}e^{-m(z)/z}.
\end{align*}
Now, $\cT^*(z)$ implies $\cD^*(z)$, so
\[
 \frac{2s_2(z)}{m(z)}=\big(1+\eps(2,z)s^*_2(z)\big)\frac{2\ts_2(z)}{m(z)}
 =\big(1+\eps(2,z)s^*_2(z)\big)\big(1-e^{-m(z)/z}\big),
\]
by \eqref{eq:ts2} and~\eqref{def:sstar:critical}. Thus we obtain
\begin{align*}
 \E\big[D(z)\id_{\cD(z)}\mid\cF_z\big]
 &=\Big(e^{m(z)/z}-\big(1+\eps(2,z)s^*_2(z)\big)\big(e^{m(z)/z}-1\big)
  +o(1)\Big)^{-1}\frac{m(z)}{z}\\
 &=\Big(1-\eps(2,z)s^*_2(z)\big(e^{m(z)/z}-1\big)+o(1)\Big)^{-1}\frac{m(z)}{z}\\
 &=\big(1+\gamma(z)+o(1)\big)\frac{m(z)}{z},
\end{align*}
and similarly for $\E\big[D(z)\mid\cF_z\big]$, as required.
\end{proof}

\subsection{The expected change in $\ts_k(z)$}

The next step is to bound the expected change of~$\ts_k(z)$. We shall use \Lm{branching:critical}
to bound the first moment of $D(z)$, and \Th{branching} to bound its second moment. To simplify
the statement, let us define
\begin{equation}\label{def:gk}
 g_k(z):=\frac{\ts_k(z)}{z}+\frac{e^{-m(z)/z}}{k!}\bigg(\frac{m(z)}{z}\bigg)^k
\end{equation}
for each $k\ge 2$. Note that $g_k(z)=O(1)$ if the event $\cT^*(z)$ holds, since $\ts_k(z)\le m(z)$
and $m(z)=O(z)$, by \Ob{m:bounded}. We shall prove the following lemma.

\begin{lemma}\label{lem:deltastilde}
 Let\/ $z\in [z_-, z_+]$. If\/ $\cT^*(z)$ holds, then
 \begin{equation}\label{eq:deltastilde}
  \E\big[ \big( \ts_k(z-1)-\ts_k(z) \big)\id_{\cD(z)}\mid\cF_z\big]
  =-\frac{\ts_k(z)}{z}-\big(\gamma(z)+o(1)\big)g_k(z)
 \end{equation}
 for every\/ $2\le k\le 4u_0$.
\end{lemma}

Note that, since $g_k(z)=O(1)$, the error term is $o(1)$. However, it will be important in the
proof of \Lm{deltaSstar} that $g_k(z)$ is significantly smaller than this when $k\to\infty$.
The first step in the proof of \Lm{deltastilde} is to obtain deterministic bounds on
$\ts_k(z-1)-\ts_k(z)$, which follow via some easy algebra. We give the details for completeness.

\begin{lemma}\label{lem:deterministic:deltastilde}
 Let\/ $z\in [z_-, z_+]$. If\/ $\cT^*(z)$ and\/ $\cD(z)$ hold, then
 \begin{equation}\label{eq:deterministic:deltastilde}
  \ts_k(z)-\ts_k(z-1)=\frac{\ts_k(z)}{z}+\frac{zg_k(z)}{m(z)}
  \bigg(D(z)-\frac{m(z)}{z}+\frac{O\big(k(D(z)^2+1)\big)}{z}\bigg).
 \end{equation}
 for\/ $2\le k\le 4u_0$.
\end{lemma}

\begin{proof}
For each $k\ge 2$ and $w\ge0$, set
\[
 f_k(w):=\frac{we^{-w}}{k(k-1)}\sum_{\ell=k-1}^\infty\frac{w^\ell}{\ell!},
\]
so that $\ts_k(z)/z=f_k(m(z)/z)$. Observe that
\[
 f'_k(w)=\frac{(1-w)e^{-w}}{k(k-1)}\sum_{\ell=k-1}^\infty\frac{w^\ell}{\ell!}
 +\frac{we^{-w}}{k(k-1)}\sum_{\ell=k-2}^\infty\frac{w^\ell}{\ell!}
 =\frac{1}{w}\bigg(f_k(w)+\frac{w^k e^{-w}}{k!}\bigg)
\]
so that $f'_k(m(z)/z)=zg_k(z)/m(z)$. Observe also that
\begin{equation}\label{eq:fpp}
 f''_k(w)=\frac{1}{w}\frac{d}{dw}\big(wf'_k(w)-f_k(w)\big)
 =\frac{1}{w}\frac{d}{dw}\frac{w^k e^{-w}}{k!}
 =\frac{w^{k-2}e^{-w}}{(k-1)!}-\frac{w^{k-1}e^{-w}}{k!}.
\end{equation}
Thus, if $w$ is bounded away from~0, $f''_k(w)=O(kf'_k(w))$.
Moreover, if $|w'-w|=O(u_0^2/z)=o(1/k)$ then we still have
$f''_k(w'')=O(kf'_k(w))$ for all $w''\in[w,w']$ as neither term
in \eqref{eq:fpp} changes by more than a constant factor. Thus, by Taylor's Theorem,
\begin{equation}\label{eq:df}
 f_k(w)-f_k(w')=f'_k(w)\Big( w-w'+O\big( k(w-w')^2 \big) \Big).
\end{equation}
Writing $w=m(z)/z$ and $w'=m(z-1)/(z-1)$ we have
\[
 w-w'=\frac{m(z)}{z}-\frac{m(z)-D(z)}{z-1}=\frac{1}{z-1}\bigg(D(z)-\frac{m(z)}{z}\bigg).
\]
If $\cT^*(z)$ holds then $w$ is indeed bounded away from zero as $\cT^*(z)$ implies
$\cQ(z)$, which implies $\cM(z)$.
Also $\cD(z)$ implies $D(z)\le u_0^2$, so $|w'-w|=O(u_0^2/z)$. Thus substituting
these values of $w$ and $w'$ into \eqref{eq:df} gives
\[
 \frac{\ts_k(z)}{z}-\frac{\ts_k(z-1)}{z-1}=
 \frac{zg_k(z)}{(z-1)m(z)}\bigg(D(z)-\frac{m(z)}{z}+\frac{O(k(D(z)^2+1))}{z}\bigg).
\]
Multiplying by $z-1$ then gives \eqref{eq:deterministic:deltastilde}.
\end{proof}

We note here, for future reference, the following identity, which follows immediately from the
definition~\eqref{eq:tsk} of $\ts_k(z)$.

\begin{obs}\label{obs:ksk}
 For each\/ $k\ge 2$, and every\/ $z\in [\pi(x)]$,
 \[
  (k-1)\ts_k(z)-(k+1)\ts_{k+1}(z)=\frac{e^{-m(z)/z}}{k!}\frac{m(z)^{k}}{z^{k-1}}.
 \]
\dispqed
\end{obs}

Note that it follows, as an immediate corollary of this observation, that
\begin{equation}\label{eq:g:bound:by:tsk}
 g_k(z)z=k\ts_k(z)-(k+1)\ts_{k+1}(z)\le k\ts_k(z).
\end{equation}
Moreover, it follows from Lemma~\ref{lem:deterministic:deltastilde} and~\eqref{eq:g:bound:by:tsk} that if the events $\cT^*(z)$ and $\cD(z)$ hold for some $z \in [z_-, z_+]$, then
\begin{equation}\label{eq:adjacent:shats:theta}
\frac{\ts_k(z-1)}{\ts_k(z)} = 1 + O\big( z^{- 1 + o(1)} \big),
\end{equation}
since $\cD(z)$ implies $D(z)\le u_0^2$ and $\cM(z)$ implies $m(z) = \Theta(z)$, as noted above. 

To deduce \Lm{deltastilde} from Lemma~\ref{lem:deterministic:deltastilde}, we just need to take expectations of both sides and apply \Th{branching} and \Lm{branching:critical}.

\begin{proof}[Proof of Lemma~\ref{lem:deltastilde}]
Recall that
\[
 \E\big[D(z)\id_{\cD(z)}\mid\cF_z\big]=\big(1+\gamma(z)+o(1)\big)\frac{m(z)}{z},
\]
by \Lm{branching:critical}, and that $\E[D(z)^2\mid\cF_z]=O(1)$ by \Th{branching}
and \Lm{Kcritical}. It therefore follows from \Lm{deterministic:deltastilde} and \eqref{eq:cDlikely} that
\begin{align*}
 \E\big[ \big( \ts_k(z)-\ts_k(z-1) \big)\id_{\cD(z)}\mid\cF_z\big]
 &=\frac{\ts_k(z)}{z}+\bigg(\gamma(z)+o(1)+\frac{O(k)}{m(z)}\bigg)g_k(z) + O\big( z_0^{-20} \big), 
\end{align*}
since $\ts_k(z) / z \le g_k(z) = O(1)$. The result follows since $k\le 4u_0=o(m(z))$.
\end{proof}

\subsection{The expected change in $s_k(z)$}

We now arrive at the main calculation: that of the expected change in the number of edges
of size~$k$.

\begin{lemma}\label{lem:deltas}
 Let\/ $z\in [z_-,z_+]$. If\/ $\cT^*(z)$ holds, then
 \[
  \E\big[ \big( s_k(z-1)-s_k(z) \big) \id_{\cD(z)}\mid\cF_z\big]\in-\frac{\ts_k(z)}{z}-\gamma(z)g_k(z)
  -\frac{k\cdot\eps(k,z)\ts_k(z)}{z}\bigg(s^*_k(z)\pm\frac{1}{6}\bigg)
 \]
 for every\/ $2\le k\le 4u_0$.
\end{lemma}

We shall first prove the following deterministic lemma.

\begin{lemma}\label{lem:deltas:deterministic}
 Let\/ $z\in [z_-,z_+]$. If\/ $\cT^*(z)$ holds, then
 \[
  ks_{k}(z)-(k+1)s_{k+1}(z)
  \in g_k(z)z+k\cdot\eps(k,z)\ts_k(z)\bigg(s^*_k(z)\pm\frac{1}{8}\bigg)
 \]
 for every\/ $k\ge 2$.
\end{lemma}

\begin{proof}
Observe first that, since the event $\cT^*(z)$ holds, we have
\begin{align*}
 ks_{k}(z)-(k+1)s_{k+1}(z)
 &\in k\big(1+\eps(k,z)s^*_k(z)\big)\ts_k(z)-(k+1)\big(1\pm\eps(k+1,z)\big)\ts_{k+1}(z)\\
 &=g_k(z)z+k\cdot\eps(k,z)s^*_k(z)\ts_k(z)\pm (k+1)\eps(k+1,z)\ts_{k+1}(z),
\end{align*}
by \eqref{eq:g:bound:by:tsk}. But by \Df{eps} and Observations \ref{obs:m:bounded}
and~\ref{obs:stildes:ineq}, we have
\begin{equation}\label{eq:epstsk:change}
 (k+1)\eps(k+1,z)\ts_{k+1}(z)\le\frac{m(z)}{z}\eps(k+1,z)\ts_k(z)
 \le\eps_1 C_0(k+1)\cdot\eps(k,z)\ts_k(z),
\end{equation}
so since $\eps_1 C_0(k+1)\le 2k\eps_1e^{C_0}<k/8$, by \eqref{def:small},
the claimed bounds follow.
\end{proof}

We can now use \Th{branching} to bound the expected change in $s_k(z)$.

\begin{proof}[Proof of \Lm{deltas}]
Let $z\in [z_-, z_+]$, and suppose that $\cT^*(z)$ holds. Recall that
\begin{equation}\label{eq:recall:deltas}
 s_k(z-1)-s_k(z)\in \id_{\{d(z)=k\}}
 +\big(-\Delta_k(z)+\Delta_{k+1}(z)\pm\Delta'(z)\big)\id_{\{d(z)=1\}}
\end{equation}
for each $k\ge 2$, by \Ob{change:ins:using:deltas}. Also,
\begin{equation}
 \E\big[\Delta_k(z)\id_{\cD(z)}\mid\cF_z,\,d(z)=1\big]
 =\frac{ks_k(z)}{m(z)}\Big(\E\big[D(z)\mid\cF_z,\,d(z)=1\big]+o(1)\Big)+O\big( m(z)^{-1/3} \big)
\end{equation}
for each $k\ge 2$, and
\begin{equation}
 \E\big[\Delta'(z)\mid\cF_z,\,d(z)=1\big]=O\big(m(z)^{-1/3}\big),
\end{equation}
by \Th{branching} and \Lm{Kcritical}, where we have used the fact that
$\Delta_k(z)\le s_k(z)$ deterministically and \eqref{eq:cDlikely} to
bound the error when $\cD(z)$ fails. Moreover,
\begin{equation}\label{eq:recall:last}
 \E\big[D(z)\mid\cF_z,\,d(z)=1\big]\,\Prb\big(d(z)=1 \mid\cF_z\big)
 =\big(1+\gamma(z)+o(1)\big)\frac{m(z)}{z},
\end{equation}
by \Lm{branching:critical}. We claim that 
\begin{align}
 &\E\big[\big(s_k(z-1)-s_k(z)\big)\id_{\cD(z)}\mid\cF_z,\,d(z)=1\big]
 \,\Prb\big(d(z)=1\mid\cF_z\big)\notag\\
 &\hspace{3cm}=\frac{1+\gamma(z)}{z}\big((k+1)s_{k+1}(z)-k s_{k}(z)\big)
 +\frac{o\big(k\eps(k,z)\ts_k(z)\big)}{z}.\label{eq:change:sk:claim}
\end{align}
Indeed, in order to deduce this from~\eqref{eq:recall:deltas}--\eqref{eq:recall:last}, we just need 
to show that the various error terms are all $o\big( k\eps(k,z)\ts_k(z)/z \big)$. To see this, note first
\begin{equation}\label{eq:epskz:boundedbelow}
 \eps(k,z)\ge \eps_1^kk!\ge \eps_1^{1/\eps_1}(1/\eps_1)!
\end{equation}
is bounded below by a positive constant, since $\Lambda(z) \le 1$ for every $z \in [z_-,z_+]$, and recall that $\eps(k,z)\ts_k(z)=z_0^{1+o(1)}$, by \Ob{eps:tsk:rough} and our assumption that $k\le 4u_0$, that $\eps(k+1,z)\ts_{k+1}(z) = O(\eps(k,z)\ts_k(z))$ by~\eqref{eq:epstsk:change}, that $|s^*_k(z)| + |s^*_{k+1}(z)| = O(1)$, by $\cT^*(z)$, and that $m(z) = \Theta(z)$. The error terms are therefore at most
$$\frac{o\big(k s_{k}(z) + (k+1)s_{k+1}(z) \big)}{z} + O\big(m(z)^{-1/3}\big) = \frac{o\big(k\eps(k,z)\ts_k(z)\big)}{z},$$
as claimed, so we have proved~\eqref{eq:change:sk:claim}. 


To deal with the case when $d(z)>1$, observe also that, by Lemmas \ref{lem:dz}
and~\ref{lem:Kcritical}, we have
\[
 \Prb\big(d(z)=k\mid\cF_z\big)
 = \big(1+o(1)\big)^k \frac{e^{-m(z)/z}}{k!} \bigg(\frac{m(z)}{z}\bigg)^k+\frac{O(1)}{z}.
\]
We claim that in fact
\begin{equation}\label{eq:prob:dzk:recalled}
 \Prb\big(d(z)=k\mid\cF_z\big)=\frac{e^{-m(z)/z}}{k!}\bigg(\frac{m(z)}{z}\bigg)^k
 +\frac{o\big(\eps(k,z)\ts_k(z)\big)}{z}.
\end{equation}
To see this, note that if $k=O(1)$ then $(1+o(1))^k=1+o(1)$ and
$\frac{e^{-m(z)/z}}{k!}\big(\frac{m(z)}{z}\big)^k=O(\ts_k(z)/z)$, by~\eqref{eq:tsk}.
On the other hand, if $k=\omega(1)$ then $\Prb(d(z)=k\mid\cF_z)\le (2C_0)^k/k!+O(1/z)$ decreases
super-exponentially with~$k$, while $\eps(k,z)\ts_k(z)=e^{O(k)}z$, by~\eqref{eq:eps:tsk:rough}
and our assumption that $k\le 4u_0$.

Now, noting that $d(z)=k>1$ implies $\cD(z)$, it follows from~\eqref{eq:recall:deltas},
\eqref{eq:change:sk:claim} and~\eqref{eq:prob:dzk:recalled} that
\begin{align*}
 \E\big[\big(s_k(z-1)-s_k(z)\big)\id_{\cD(z)}\mid\cF_z\big]
 &=\frac{1+\gamma(z)}{z}\big((k+1)s_{k+1}(z)-k s_{k}(z)\big)\\
 &\hspace{2.5cm}+\frac{e^{-m(z)/z}}{k!}\frac{m(z)^k}{z^k}+\frac{o\big(k\eps(k,z)\ts_k(z)\big)}{z}.
\end{align*}
By \Lm{deltas:deterministic}, this is contained in
\[
 \frac{e^{-m(z)/z}}{k!}\frac{m(z)^k}{z^k}-\frac{1+\gamma(z)}{z}
 \bigg(g_k(z)z+k\cdot\eps(k,z)\ts_k(z)\bigg(s^*_k(z)\pm\frac{1}{7}\bigg)\bigg)
\]
which is equal to
\[
 -\frac{\ts_k(z)}{z}-\gamma(z)g_k(z)-\frac{1+\gamma(z)}{z}
 \bigg(k\cdot\eps(k,z)\ts_k(z)\bigg(s^*_k(z)\pm\frac{1}{7}\bigg)\bigg).
\]
Since $|\gamma(z)|\le\eps_1$ and $|s^*_k(z)| \le 1$, by \Ob{gamma:small} and $\cT^*(z)$, the lemma follows.
\end{proof}

\subsection{The proof of Lemmas~\ref{lem:deltaSstar} and~\ref{lem:maxstep}}

We're finally ready to prove the two main lemmas of the section. We'll prove \Lm{maxstep} first,
since we shall need (a weak form of) it in the proof of \Lm{deltaSstar}.


\begin{proof}[Proof of \Lm{maxstep}]
If $\cD(z)$ fails to hold then $s^*_k(z-1) = s^*_k(z)$ by definition, so we may assume that 
$|m(z)-m(z-1)|\le u_0^2$, and hence $|s_k(z-1)-s_k(z)|\le 2u_0^3$, by~\eqref{eq:schange:vs:mchange}. We claim first that 
\begin{align*}
 s^*_k(z-1)-s^*_k(z)
 &=\frac{s_k(z-1)-\ts_k(z-1)}{\eps(k,z-1)\ts_k(z-1)}-\frac{s_k(z)-\ts_k(z)}{\eps(k,z)\ts_k(z)}\\
 &=\frac{s_k(z-1)-\ts_k(z-1)-s_k(z)+\ts_k(z)}{\eps(k,z)\ts_k(z)}+ O\bigg(\frac{s_k^*(z-1)}{z^{1+o(1)}} \bigg).
\end{align*}
Indeed, to see this simply recall that, by \Df{eps}, \Lm{Lambda:onestep} and~\eqref{eq:adjacent:shats:theta}, the events $\cT^*(z)$ and $\cD(z)$ imply that 
$$\frac{\eps(k,z-1)\ts_k(z-1)}{\eps(k,z)\ts_k(z)} = \frac{\Lambda(z)}{\Lambda(z-1)} \cdot \frac{\ts_k(z-1)}{\ts_k(z)} = 1 + O\big( z^{-1+o(1)} \big).$$
Now, note that $\ts_k(z-1)-\ts_k(z)=O(u_0^2)$, by \Lm{deterministic:deltastilde}, since $m(z)=\Theta(z)$, 
$\ts_k(z)/z\le g_k(z)=O(1)$ and $u_0=z_0^{o(1)}$. Thus
\[
 |s^*_k(z-1)-s^*_k(z)|=\frac{O(u_0^3)}{\eps(k,z)\ts_k(z)} + O\big( z^{-1+o(1)} \big)
\]
as $\cT^*(z)$ implies $|s_k^*(z)|\le 1$.
Now $z=z_0^{1+o(1)}$ for every $z\in [z_-,z_+]$, by \Lm{Lambda:precise}, and
$\eps(k,z)\ts_k(z)=z_0^{1+o(1)}$ for all $k\le 4u_0$, by \Ob{eps:tsk:rough}. Hence
\[
 |s^*_k(z-1)-s^*_k(z)|\le z_0^{-1+o(1)},
\]
as required.
\end{proof}

To finish the section, we shall deduce \Lm{deltaSstar} from
Lemmas~\ref{lem:maxstep},~\ref{lem:deltastilde} and~\ref{lem:deltas}.

\begin{proof}[Proof of Lemma~\ref{lem:deltaSstar}]
Let $z\in [z_-,z_+]$ and suppose that $\cT^*(z)$ holds. We shall break
\[
 s^*_k(z-1)-s^*_k(z)=\bigg(\frac{s_k(z-1)-\ts_k(z-1)}{\eps(k,z-1)\ts_k(z-1)}
 -\frac{s_k(z)-\ts_k(z)}{\eps(k,z)\ts_k(z)}\bigg)\id_{\cD(z)}
\]
into two pieces (see~\eqref{eq:notquitetrue} and~\eqref{eq:deltaSstar:error:term}, below) and
bound the expected size of each of them in turn. We note that, by
Lemmas~\ref{lem:deltastilde} and~\ref{lem:deltas}, we have
\[
 \E\Big[\big(s_k(z-1)-\ts_k(z-1)-s_k(z)+\ts_k(z)\big)\id_{\cD(z)}\bmid\cF_z\Big]
 \in-\frac{k\eps(k,z)\ts_k(z)}{z}\bigg(s^*_k(z)\pm\frac{1}{6}\bigg)+o(g_k(z)).
\]
But $g_k(z)\le k\ts_k(z)/z=O(k\eps(k,z)\ts_k(z)/z)$ by \eqref{eq:g:bound:by:tsk}
and the fact that $\eps(k,z)$ is bounded away from~0. Thus
\begin{equation}\label{eq:notquitetrue}
 \E\bigg[\frac{s_k(z-1)-\ts_k(z-1)-s_k(z)+\ts_k(z)}{\eps(k,z)\ts_k(z)}\id_{\cD(z)}\bmid\cF_z\bigg]
 \in-\frac{k}{z}\bigg(s^*_k(z)\pm\frac{1}{5}\bigg).
\end{equation}

We claim that
\begin{equation}\label{eq:deltaSstar:error:term}
 \E\bigg[\bigg(\frac{\eps(k,z)\ts_k(z)-\eps(k,z-1)\ts_k(z-1)}{\eps(k,z)\ts_k(z)}\bigg)
 s_k^*(z-1)\id_{\cD(z)}\bmid\cF_z\bigg]\in\frac{s^*_k(z)\pm 3\eps_1 k}{z},
\end{equation}
To prove~\eqref{eq:deltaSstar:error:term}, observe first that if $\cD(z)$ holds then
\[
 \eps(k,z)\ts_k(z)-\eps(k,z-1)\ts_k(z-1)
  =\eps(k,z)\big(\ts_k(z)-\ts_k(z-1)\big)+o\bigg(\frac{\eps(k,z)\ts_k(z)}{z}\bigg),
\]
since $\eps(k,z-1) / \eps(k,z) = 1 + o(1/z)$, by \Df{eps} and \Lm{Lambda:onestep}, and $\ts_k(z-1) = \Theta\big( \ts_k(z) \big)$ if $\cT^*(z)$ and $\cD(z)$ hold, by~\eqref{eq:adjacent:shats:theta}. Therefore, by \Lm{deltastilde},
\begin{align*}
 &\E\Big[\big(\eps(k,z)\ts_k(z)-\eps(k,z-1)\ts_k(z-1)\big)\id_{\cD(z)}\mid\cF_z\Big]\\
 &\hspace{3cm}=\frac{\eps(k,z)\ts_k(z)}{z}\bigg(1+\frac{(\gamma(z)+o(1))g_k(z)z}{\ts_k(z)}+o(1)\bigg)\\
 &\hspace{3cm}\in\big(1\pm 2\eps_1 k\big)\frac{\eps(k,z)\ts_k(z)}{z},
\end{align*}
since $g_k(z)\le k\ts_k(z)/z$, by~\eqref{eq:g:bound:by:tsk}, and $|\gamma(z)|\le\eps_1$,
by \Ob{gamma:small}. Since $|s^*_k(z-1)-s^*_k(z)|\le\eps_1$ by \Lm{maxstep}, it follows that
\begin{align*}
 \E\bigg[\bigg(\frac{\eps(k,z)\ts_k(z)-\eps(k,z-1)\ts_k(z-1)}{\eps(k,z)\ts_k(z)}\bigg)
 s_k^*(z-1)\id_{\cD(z)}\bmid\cF_z\bigg]&\in\frac{(1\pm 2\eps_1 k)(s_k^*(z)\pm\eps_1)}{z}\\
 &\in\frac{s^*_k(z)\pm 3\eps_1 k}{z},
\end{align*}
as claimed. As noted above, adding \eqref{eq:notquitetrue} and~\eqref{eq:deltaSstar:error:term}
we obtain
\[
 \E\big[s^*_k(z-1)- s^*_k(z)\mid\cF_z\big]
 \in-\frac{k}{z}\bigg(s^*_k(z)\pm\frac{1}{5}\bigg)+\frac{s^*_k(z)\pm 3\eps_1 k}{z}
 \subseteq -\frac{k-1}{z}\bigg(s^*_k(z)\pm\frac{1}{2}\bigg)
\]
for every $k\ge 2$, which completes the proof of the lemma.
\end{proof}

\section{The proof of Theorems~\ref{thm:track:m} and~\ref{thm:track:s}}\label{sec:proof:tracking}

In this section we shall use the method of self-correcting martingales to deduce
Theorems~\ref{thm:track:m} and~\ref{thm:track:s} from Lemmas~\ref{lem:deltaSstar}
and~\ref{lem:maxstep}. Recall that \Th{track:m} states that the event $\cM^*(z_-)$ holds
with high probability, and \Th{track:s} states that with high probability the event $\cT_k(z)$
holds for every $k\ge 2$ and every $z\in [z_-,z_+]$. Since we already know that $\cM^*(z_+)$
holds with high probability, by \Pp{largez}, it will suffice to prove that the event
$\cT^*(z_-)$ holds with high probability, since $\cT^*(z_-)$ implies that $\cM(z)$ holds for every
$z\in [z_-,z_+]$, and that $\cT_k(z)$ holds for every $k\ge 2$ and every $z\in [z_-,z_+]$.

As in \Sc{big:z}, the rough idea is to bound the probability, for each $a\in [z_-,z_+]$, that $a$
is maximal such that $\cT^*(a)$ fails to hold. Let us begin by proving the base case.

\begin{lemma}\label{lem:Tzplus:holds}
 $\cT^*(z_+)$ holds with high probability.
\end{lemma}
\begin{proof}
Recall that $\cQ(z_+)\cup\cM^*(z_+)^c$ holds with high probability, by \Lm{easy}; and that
$\cD^*(z_+)\cap\cM^*(z_+)\cap\bigcap_{k=2}^{4u_0}\cT_k(z_+)$ holds with high probability, by \Pp{largez} and
\Co{zp}. Thus
\[
\cT^*(z_+)=\cD^*(z_+)\cap\cQ(z_+)\cap\bigcap_{k=2}^{4u_0}\cT_k(z_+)
\]
holds with high probability, as claimed.
\end{proof}

By \Lm{Tzplus:holds}, we can assume that $z_-\le a< z_+$ and that $\cT^*(a+1)$ holds. For each
pair $z_-\le a<b\le z_+$ and each $k\ge 2$, let $\cW_k(a,b)$ denote the event that the following all occur:
\begin{itemize}
\item[$(a)$] $s^*_k(a)>1$,
\item[$(b)$] $s^*_k(z)>3/4$ for every $a<z<b$,
\item[$(c)$] $s^*_k(b)\le 3/4$,
\item[$(d)$] $\cT^*(a+1)$ holds.
\end{itemize}
Note that if $\cT^*(a+1)\cap\big\{s^*_k(z_+)\le 3/4\big\}$ holds and $s^*_k(a)>1$, then the event
$\cW_k(a,b)$ holds for some $a<b\le z_+$. We shall prove the following lemma.

\begin{lemma}\label{lem:Wab}
 If\/ $z_-\le a<b\le z_+$, then
\[
 \Prb\big(\cW_k(a,b)\big)\le z_0^{-20}
\]
 for every\/ $2\le k\le 4u_0$.
\end{lemma}
\begin{proof}
We may assume $s^*_k(b)\ge 3/4-\eps_1$ as otherwise $s^*_k(b-1)<3/4$ by \Lm{maxstep}, and so
$\cW_k(a,b)$ fails to hold. Set $\ell=b-a$, and for each $0\le t\le\ell$, define
\[
 X_t:=\begin{cases}
  s^*_k(b-t)-s^*_k(b),&\text{if }X_{t-1}\ge0\text{ or }t=0;\\
  X_{t-1},&\text{otherwise.}
 \end{cases}
\]
We claim that $X_t$ is a super-martingale with respect to the filtration $(\cF_{b-t})_{t = 0}^\ell$.
Indeed, since $\cT^*(a+1)$ holds, we have either $X_t<0$, in which case $X_{t+1}=X_t$,
or $X_t\ge0$, in which case
\[
\E\big[X_{t+1}-X_t\mid\cF_{b-t}\big]\le\frac{k-1}{b-t}\bigg(-s^*_k(b-t)+\frac{1}{2}\bigg)
\]
by \Lm{deltaSstar}. But if $X_t\ge0$ then $s^*_k(b-t)\ge s^*_k(b)\ge 3/4-\eps_1>1/2$,
so in all cases $\E[X_{t+1}-X_t\mid\cF_{b-t}]\le0$. Recalling that $\ell\le z_+=z_0^{1+o(1)}$,
it follows by the Azuma--Hoeffding inequality that
\[
 \Prb\big(\cW_k(a,b)\big)\le\exp\big(-z_0^{1-3\eps_1}\big)\le z_0^{-20},
\]
as claimed.
\end{proof}

It is easy to see that one can deal with the case $s^*_k(a)<-1$ in exactly the same way, so we
leave the details to the reader. Using the union bound over all choices of~$b$, it follows that
\[
\Prb\Big(\big\{|s^*_k(a)|>1\big\}\cap\cT^*(a+1)\cap\big\{s^*_k(z_+)\le 3/4\big\}\Big)\le z_0^{-18}
\]
for every $a\in [z_-,z_+]$ and $2\le k\le 4u_0$. Moreover, it follows from \Co{zp}
that, with high probability, $s^*_k(z_+)\le 1/2$ for every $k\ge 2$.

It therefore only remains to bound the probability that $\cD(z+1)^c\cup\cQ(z)^c\cup\cT^*(z+1)$
holds for some $z\in [z_-,z_+]$. By~\eqref{eq:cDlikely} and Lemmas~\ref{lem:easy} and~\ref{lem:Kcritical}, 
it will therefore suffice
to show that $m(z)$ is unlikely to be the first variable to go off track. Unfortunately,
unlike with $s_k(z)$, $m(z)$ is not self-correcting, and so a simple martingale approach
will not work. Instead we shall prove this using a two stage approach: we shall show, using
super-martingales, that $m(z)$ can only drift off track slowly, but every
so often (and well before it drifts so far as to cause $\cM(z)$ to fail) we shall use the
following lemma to put it firmly back on track. The following lemma, which follows from \Lm{m0} and
\Th{compare}, ensures that $m(z)$ is far closer to its target value than required by
$\cM(z)$, but unfortunately has a relatively large failure probability. Thus we
cannot use it for very many values of~$z$.

\begin{lemma}\label{lem:empty}
 For every\/ $z\in [z_-, z_+]$,
 \[
  \Prb\bigg(\cK(z)\cap\bigg\{m(z)\exp\bigg(-\sum_{k\ge 2}\frac{k s_k(z)}{m(z)}\bigg)
  \notin (1\pm 3\eps_1)\eta\Lambda(z)z\bigg\}\bigg)=\frac{O(1)}{u_0}.
 \]
\end{lemma}
\begin{proof}
Recall that, by \Lm{m0}, the number of isolated vertices $m_0(z)$ in the hypergraph $\cS_A(z)$
satisfies
\begin{equation}\label{eq:isolatedvertices:firstestimate}
 m_0(z)\in\big(1\pm\eps_1\big)\eta\Lambda(z)z
\end{equation}
with probability at least $1-1/x^2$. We shall use \Th{compare} to give another way of
approximating~$m_0(z)$, and together these estimate will imply the lemma.

Recall that a vertex $i$ in $\cS_A(z)$ is isolated if the submatrix
$A\big[\{i\}\times[z+1,\pi(x)]\big]$ is the all zero vector. We bound the number $m_0(z)$ of
isolated vertices using the second moment method, using \Th{compare} to estimate both the
mean and variance of $m_0(z)$. First, let $\cE\in\cF^+_z$ be an event of the
form~\eqref{def:eventE} for which $\cK(z)$ holds and $d(z) \le 4u_0$, and recall from \Df{tilde:variables} that the
random matrix $\tA_\cE$ is obtained by choosing each column uniformly at random from all
$\binom{m(z)}{d_j}$ possible choices, independently in each column. Thus for $j>z$,
\[
 \Prb\big(\tA_\cE[\{i\}\times\{j\}]=0\big)=1-\frac{d_j}{m(z)}
 =\exp\bigg(-\frac{d_j}{m(z)}+\frac{O(d_j^2)}{m(z)^2}\bigg).
\]
As the columns of $\tA_\cE$ are independent,
\[
 \Prb\big(\tA_\cE\big[\{i\}\times[z+1,\pi(x)]\big]=0\big)
 =\exp\bigg(-\frac{1}{m(z)}\sum_{j>z}d_j+\frac{O(1)}{m(z)^2}\sum_{j>z}d_j^2\bigg).
\]
Now $\sum_{j>z}d_j=\sum_{k\ge2}ks_k(z)$ and $\sum_{j>z}d_j^2=\sum_{k\ge2}k^2s_k(z)=O(m(z))$
by condition $\cK(z)$. Thus
\[
 \Prb\big(\tA_\cE\big[\{i\}\times[z+1,\pi(x)]\big]=0\big)
 =\exp\bigg(-\sum_{k\ge2}\frac{ks_k(z)}{m(z)}+\frac{O(1)}{m(z)}\bigg).
\]
Applying \Th{compare} with $I=\{i\}$, $C=[z+1,\pi(x)]$, and $R=0$, gives
\begin{align*}
 \Prb\big[i\text{ is isolated in }\cS_A(z)\mid\cE\big]
 &=\Prb\big(A\big[\{i\}\times[z+1,\pi(x)]\big]=0\mid\cE\big)\\
 &=\exp\bigg(-\sum_{k\ge2}\frac{ks_k(z)}{m(z)}+\frac{O(1)}{u_0}\bigg).
\end{align*}
Thus, summing over $i\in M(z)$,
\begin{equation}\label{eq:m0:first}
 \mu:=\E\big[m_0(z)\mid\cE\big]=m(z)
 \exp\bigg(-\sum_{k\ge2}\frac{ks_k(z)}{m(z)}+\frac{O(1)}{u_0}\bigg).
\end{equation}
For the second moment of $m_0(z)$ we consider the probability that two distinct vertices
$i_1,i_2$ are both isolated. In the independent model we have
\[
 \Prb\big(\tA_\cE[\{i_1,i_2\}\times\{j\}]=0\big)
 =\bigg(1-\frac{d_j}{m(z)}\bigg)\bigg(1-\frac{d_j}{m(z)-1}\bigg)
 =\exp\bigg(-\frac{2d_j}{m(z)}+\frac{O(d_j^2)}{m(z)^2}\bigg).
\]
As the columns of $\tA_\cE$ are independent, a similar argument to the above yields
\[
 \Prb\big(\tA_\cE\big[\{i_1,i_2\}\times[z+1,\pi(x)]\big]=0\big)
 =\exp\bigg(-\sum_{k\ge2}\frac{2ks_k(z)}{m(z)}+\frac{O(1)}{m(z)}\bigg).
\]
Applying \Th{compare} with $I=\{i_1,i_2\}$, $C=[z+1,\pi(x)]$, and $R=0$, gives
\[
 \Prb\big(A\big[\{i_1,i_2\}\times[z+1,\pi(x)]\big]=0\mid\cE\big)
 =\exp\bigg(-\sum_{k\ge2}\frac{2ks_k(z)}{m(z)}+\frac{O(1)}{u_0}\bigg).
\]
Summing over all ordered pairs $(i_1,i_2)$ we obtain
\begin{align}
 \E\big[m_0(z)(m_0(z)-1)\mid\cE\big]&=m(z)(m(z)-1)
 \exp\bigg(-\sum_{k\ge2}\frac{2ks_k(z)}{m(z)}+\frac{O(1)}{u_0}\bigg)\nonumber\\
 &=\bigg(1+\frac{O(1)}{u_0}\bigg)\mu^2.\label{eq:m0:second}
\end{align}
Combining \eqref{eq:m0:first} and \eqref{eq:m0:second} we have
\[
 \Var\big(m_0(z)\mid\cE\big)=\E\big[m_0(z)(m_0(z)-1)\mid\cE]+\mu-\mu^2=\mu+O(1/u_0)\mu^2.
\]
Now $\sum ks_k(z)=O(m(z))$ by $\cK(z)$ and so $\mu=\Theta(m(z))$. Thus
$\Var\big(m_0(z)\mid\cE\big)=O(\mu^2/u_0)$ and hence, by Chebychev's inequality,
\[
 \Prb\big( m_0(z)\notin(1\pm\eps_1)\mu\mid\cE \big)
 \le\frac{\Var\big( m_0(z)\mid\cE\big)}{\eps_1^2\mu^2}
 =\frac{O(1)}{u_0}.
\]
Since the event $\cE$ was chosen arbitrarily amongst those consistent with~$\cK(z)$ and satisfying $d(z) \le 4u_0$, and using~\eqref{eq:nottoomanyinacolumn} to bound the probability that $d(z) > 4u_0$, it follows that
\[
 \Prb\big( m_0(z)\notin(1\pm\eps_1)\mu\mid\cK(z)\big) =\frac{O(1)}{u_0}.
\]
Combining this with~\eqref{eq:isolatedvertices:firstestimate} and~\eqref{eq:m0:first} 
completes the proof of the lemma.
\end{proof}

We shall also need the following simple identity.

\begin{lemma}\label{lem:sksk}
 For every\/ $z\in [\pi(x)]$,
 \[
  \sum_{k\ge 2} k\ts_k(z)=m(z)\Ein\bigg(\frac{m(z)}{z}\bigg).
 \]
\end{lemma}
\begin{proof}
Define
\[
 f(w):=\sum_{k\ge 2}\frac{1}{k-1}\sum_{\ell=k- 1}^\infty\frac{e^{-w}w^\ell}{\ell!},
\]
and note that
\[
 \sum_{k\ge 2}k\ts_k(z)=\sum_{k\ge 2}\frac{m(z)}{k-1}e^{-m(z)/z}
 \sum_{\ell=k-1}^\infty\frac{1}{\ell!}\bigg(\frac{m(z)}{z}\bigg)^\ell
 =m(z)f\big(m(z)/z\big).
\]
Now, we have $f(0)=0$ and
\begin{align*}
 f'(w)
 &=\sum_{k\ge 2}\frac{e^{-w}}{k-1}\bigg(\sum_{\ell=k-1}^\infty\frac{w^{\ell-1}}{(\ell-1)!}
  -\sum_{\ell=k-1}^\infty\frac{w^{\ell}}{\ell!}\bigg)\\
 &=e^{-w}\sum_{k\ge 2}\frac{w^{k-2}}{(k-1)!}
 =e^{-w}\bigg(\frac{e^w-1}{w}\bigg)=\frac{1-e^{-w}}{w}.
\end{align*}
It follows that $f(w)=\ds\int_0^w\frac{1-e^{-t}}{t}\,dt=\Ein(w)$, as required.
\end{proof}

We can now complete the proof of our main auxiliary results.

\begin{proof}[Proof of Theorems~\ref{thm:track:m} and~\ref{thm:track:s}]
As noted above, it will suffice to prove that the event $\cT^*(z_-)$
holds with high probability. Recall that $\cT^*(z_+)$ holds with high probability, by
\Lm{Tzplus:holds}, and let $a\in [z_-,z_+]$ be maximal such that $\cT^*(a)$ fails to hold.
As $\cT^*(a+1)$ holds, one of the events $\cD^*(a)$, $\cQ(a)$, or $\cT_k(a)$ for some
$2\le k\le 4u_0$, must fail. By \Lm{Kcritical} and \eqref{eq:cDlikely}, with high probability
there is no $a$ such that $\cT^*(a+1)$ holds but $\cD(a+1)$ fails, so we may
assume that $\cD(a+1)$, and hence $\cD^*(a)$, holds. By \Lm{Wab},
with high probability there is no $a$ such that $\cT^*(a+1)$ and $\cD^*(a)$ hold,
but $\cT_k(a)$ fails for some $k\ge2$. By \Lm{easy}, with high probability there is
no $a$ such that $\cM^*(a)$ holds but $\cQ(a)$ fails. As $\cT^*(a+1)$ implies
$\cM^*(a+1)$, we deduce that with high probability there is no $a$ such that
$\cT^*(a+1)$ and $\cM(a)$ hold, but $\cT^*(a)$ fails.
It will therefore suffice to bound the probability that $\cM(a)$ fails to hold, assuming
that $\cD^*(a)$ holds, and that $\cT_k(a)$ holds for every $2 \le k \le 4u_0$.

To do so, let us choose a set $W=\big\{w_0,w_1,\dots,w_\ell\big\}$, where
$z_-=w_0<w_1<\dots<w_\ell=z_+$, such that $\ell=O(\log(z_+/z_-))$ and $w_i\le 2w_{i-1}$
for each $i\in [\ell]$. Since $\log(z_+/z_-)=\Theta(\sqrt{\log z_0})=o(u_0)$, by
\Lm{Lambda:precise} and \eqref{eq:ulogu}, it follows from \Lm{empty} that with high
probability either
\begin{equation}\label{eq:empty:application}
 m(w)\exp\bigg(-\sum_{k\ge 2}\frac{k s_k(w)}{m(w)}\bigg)\in (1\pm 3\eps_1)\eta\Lambda(w)w,
\end{equation}
or $\cK(w)$ fails to hold, for every $w\in W$. Since $\cT^*(z)$ implies $\cK(z)$,
by \Lm{track:K}, it will suffice to bound the probability that~\eqref{eq:empty:application}
holds for $w=w_i$, say, and
\begin{equation}\label{eq:trackingproofs:finalbound}
 \cM(a)^c\cap\cD^*(a)\cap\cT^*(a+1)\cap\bigcap_{k=2}^{4u_0}\cT_k(a)
\end{equation}
holds for some $w_{i-1}\le a<w_i$. We shall show that this has probability at
most $z_0^{-15}$.

To bound the probability of the event~\eqref{eq:trackingproofs:finalbound}, we shall use a
martingale approach to control $m(z)$ in the interval $a\le z\le w$. Define $X_0:=0$ and
for $0\le t:=w-z<w-a$,
\[
 X_{t+1}:=X_t+\bigg(\frac{m(w-t-1)}{w-t-1}-\frac{m(w-t)}{w-t}\bigg)
 \id_{\cD(w-t)\cap\cT^*(w-t)}-5\eps_1\frac{m(w)}{w^2}.
\]
We claim that $X_t$ is a super-martingale with respect to the filtration $(\cF_{w-t})_{t = 0}^{w-a}$.
Indeed, $\cT^*(w-t)$ is $\cF_{w-t}$-measurable and clearly $X_{t+1}\le X_t$ when
$\cT^*(w-t)$ fails. Assuming $\cT^*(w-t)=\cT^*(z)$ holds, we have
\[
 \E\big[m(z)-m(z-1)\mid\cF_z\big]=\big(1+\gamma(z)+o(1)\big)\frac{m(z)}{z}
\]
by \Lm{branching:critical}. Recalling that $\cT^*(z)$ implies that $|\gamma(z)|\le\eps_1$,
by \Ob{gamma:small}, it follows (using~\eqref{eq:cDlikely}) that
\begin{align*}
 \E\big[X_{t+1}-X_t\mid\cF_z\big]
 &=\E\bigg[\frac{m(z-1)-m(z)}{z-1}+\frac{m(z)}{z(z-1)}-\frac{5\eps_1 m(w)}{w^2}\Bmid\cF_z\bigg]
 +O\big(z_0^{-20}\big)\\
 &=-\big(\gamma(z)+o(1)\big)\frac{m(z)}{z(z-1)}-\frac{5\eps_1 m(w)}{w^2}\le 0
\end{align*}
for every $a<z\le w$, as claimed, since $w\le 2a$ and $m(z)\le m(w)$.

Now, $|X_{t+1}-X_t|\le z_0^{-1+\eps_1}$ for every $a<z\le w$, since $m(z)=O(z)$ (by \Ob{m:bounded}) and  
$z=z_0^{1+o(1)}$ (by~\eqref{eq:zpm:rough}), and since $\cD(z)$ implies that $|m(z)-m(z-1)|\le u_0^2=z_0^{o(1)}$. 
Thus, noting that $w-a\le z_+=z_0^{1+o(1)}$,
by the Azuma--Hoeffding inequality we obtain
\[
 \Prb\Big(\big\{X_{w-a}>z_0^{-\eps_1}\big\}\cap\cT^*(a+1)\Big)
 \le\exp\big(-z_0^{1-5\eps_1}\big)\le z_0^{-20}.
\]

Observe that if $\cD^*(a)\cap\cT^*(a+1)$ holds, then $X_{w-a}\le z_0^{-\eps_1}$ implies that
\begin{equation}\label{eq:upperbound:ma}
 \frac{m(a)}{a}\le\big(1+5\eps_1\big)\frac{m(w)}{w}+z_0^{-\eps_1}.
\end{equation}
Now, to finish the proof, we shall show that if $\cT^*(w)$ and~\eqref{eq:empty:application}
hold, then
\begin{equation}\label{eq:upperbound:mw}
 \frac{m(w)}{w}\le\alpha\big((1+5\eps_1)\eta\Lambda(w)\big).
\end{equation}
To prove this, observe first that if $\cT^*(w)$ holds then
\[
 \sum_{k\ge 2}ks_k(w)\le\sum_{k\ge2}\big(1+\eps(k,z)\big)k\ts_k(w)
 \le\bigg(\Ein\bigg(\frac{m(w)}{w}\bigg)+\eps_1\bigg)m(w),
\]
by Lemmas~\ref{lem:sum:eps:tsk} and~\ref{lem:sksk}. Thus~\eqref{eq:empty:application} implies that
\[
 \frac{m(w)}{w}\exp\bigg(-\Ein\bigg(\frac{m(w)}{w}\bigg)\bigg)
 \le e^{\eps_1}\big(1+3\eps_1\big)\eta\Lambda(w),
\]
which implies~\eqref{eq:upperbound:mw}, since $\alpha$ is increasing and
$e^{\eps_1}(1+3\eps_1)\le 1+5\eps_1$. Combining~\eqref{eq:upperbound:ma}
and~\eqref{eq:upperbound:mw}, and recalling that $\Lambda(a)=(1+o(1))\Lambda(w)\ge\delta+o(1)$
by \Lm{Lambda:onestep} and~\eqref{eq:Lambda:delta:plus:littleoone}, and that $\alpha(t)/t$ is
increasing, by~\eqref{def:alpha}, it follows that,
\[
 \frac{m(a)}{a}\le\alpha\big((1+11\eps_1)\eta\Lambda(a)\big).
\]
By the definition of $\alpha$ (and again using the fact that $\alpha$ is increasing), this implies that
\[
 \frac{m(a)}{a}\exp\bigg(-\Ein\bigg(\frac{m(a)}{a}\bigg)\bigg)
 \le\big(1+11\eps_1\big)\eta\Lambda(a).
\]
A corresponding lower bound can be proved similarly, and thus
\[
 \Prb\bigg(\cM(a)^c\cap\cD^*(a)\cap\cT^*(a+1)\cap\bigcap_{k=2}^{4u_0}\cT_k(a)\bigg)\le z_0^{-15},
\]
as claimed. As explained above, this implies that $\cT^*(z_-)$ holds with high probability, and
hence this completes the proof of Theorems~\ref{thm:track:m} and~\ref{thm:track:s}.
\end{proof}

\section{The proof of \Th{squares:sharp}}\label{sec:squares:proof}

Once we have \Th{track:m}, it is straightforward to deduce \Th{squares:sharp}. Indeed, the
deduction of the lower bound follows from the results of~\cite{CGPT}, the extra ingredient
provided by \Th{track:m} being that any linear relation can only involve at most
$m(z_-)\approx\eta\Lambda(z_-)z_-$ rows. We shall use the following result, which was proved
in~\cite{CGPT}.

\begin{prop}[Croot, Granville, Pemantle and Tetali]\label{prop:CGPT:lower}
 There exists\/ $c>0$ such that if\/ $N\le e^{-\gamma}J(x)$, then with high probability
 there does not exist a set\/ $I\subseteq [N]$ with
\[
  0<|I|\le z_0\exp\Big(-c\sqrt{\log z_0}\Big)
\]
 such that\/ $\prod_{i\in I}a_i$ is a square.
\end{prop}

We remark that \Pp{CGPT:lower} follows from the proof of~\cite[Theorem~1.3]{CGPT},
see~\cite[Section~3.5]{CGPT}, for any constant $c>\sqrt{2-\log 2}$. However, we shall only use
the fact that $c=O(1)$.

As noted in the introduction, we shall also give a new proof of the upper bound, which was
originally proved in~\cite{CGPT}. For the upper bound we shall need the following simple identity.

\begin{lemma}\label{lem:number:of:columns}
 For every\/ $z\in [\pi(x)]$,
\[
 \sum_{k\ge 2}\ts_k(z)=m(z)-z\big(1-e^{-m(z)/z}\big).
\]
\end{lemma}
\begin{proof}
Using the fact that $\frac{1}{k(k-1)}=\frac{1}{k-1}-\frac{1}{k}$, we have
\begin{align*}
 \sum_{k\ge 2}\ts_k(z)
 &=m(z)e^{-m(z)/z}\sum_{k\ge 2}\frac{1}{k(k-1)}
 \sum_{\ell=k-1}^\infty\frac{1}{\ell!}\bigg(\frac{m(z)}{z}\bigg)^\ell\\
 &=m(z)e^{-m(z)/z}\sum_{\ell=1}^\infty\frac{1}{\ell!}\bigg(\frac{m(z)}{z}\bigg)^\ell
 \sum_{k=2}^{\ell+1}\bigg(\frac{1}{k-1}-\frac{1}{k}\bigg)\\
 &=m(z)e^{-m(z)/z}\sum_{\ell=1}^\infty\frac{1}{\ell!}\bigg(\frac{m(z)}{z}\bigg)^\ell
 \bigg(1-\frac{1}{\ell+1}\bigg)\\
 &=m(z)e^{-m(z)/z}\sum_{\ell=1}^\infty\frac{1}{\ell!}\bigg(\frac{m(z)}{z}\bigg)^\ell
 -ze^{-m(z)/z}\sum_{\ell=1}^\infty\frac{1}{(\ell+1)!}\bigg(\frac{m(z)}{z}\bigg)^{\ell+1}\\
 &=m(z)e^{-m(z)/z}\big(e^{m(z)/z}-1\big)-z e^{-m(z)/z}\big(e^{m(z)/z}-1-m(z)/z\big)\\
 &=m(z)-z\big(1-e^{-m(z)/z}\big),
\end{align*}
as claimed.
\end{proof}

\begin{proof}[Proof of Theorem~\ref{thm:squares:sharp}]
To prove the lower bound, it is enough to show that for $\eta<e^{-\gamma}$ there is, with high
probability, no linear relation between the rows of~$A$. Any such relation would correspond to
an even sub-hypergraph of $\cH_A(z_-)$ without isolated vertices, and so must lie in the 2-core
$\cC_A(z_-)$. In particular, the number of rows involved is at most $m(z_-)$, which satisfies
\begin{equation}\label{eq:mtrack:application}
 m(z_-)\le\alpha\big((1+\eps_0)\Lambda(z_-)\eta\big)z_-
\le\alpha\big(2\delta e^{-\gamma}\big)z_-\le z_-
\end{equation}
with high probability, by \Th{track:m}. We may therefore assume that there is no linear
relation involving more than $z_-$ rows. However, by \Lm{Lambda:precise} we have
\[
 z_-=z_0\exp\Big(- \big(1+o(1)\big)\sqrt{\log(1/\delta) \log z_0}\Big),
\]
and so by \Pp{CGPT:lower} there is with high probability no linear relation involving at most
$z_-$ rows if $\delta$ is taken sufficiently small.
Hence there is, with high probability, no linear relation between the rows of~$A$,
as required.

To prove the upper bound, assume that we have $\eta' J(x)$ numbers with $\eta'=e^{-\gamma}+\nu$,
$\nu>0$. Pick $\eta<e^{-\gamma}$ and construct the 2-core $\cC_A(z_0)$ starting with just the
first $N=\eta J(x)$ numbers. Observe that, by Theorems~\ref{thm:track:m} and~\ref{thm:track:s},
and Lemmas~\ref{lem:sum:eps:tsk} and~\ref{lem:number:of:columns}, the number
of columns of $A$ that either have a non-zero entry in one of the rows of $M(z_0)$, or are to
the left of $z_0$, is with high probability at most
\begin{align}
 z_0+\sum_{k\ge2}s_k(z_0)
 &\le z_0+\sum_{k\ge2}\big(1+\eps(k,z_0)\big)\ts_k(z_0)\notag\\
 &\le (1+\eps_1)m(z_0)+z_0 e^{-m(z_0)/z_0}.\label{eq:active:columns}
\end{align}
On the other hand, we have at least $\nu J(x)$ remaining unused numbers, and among these
there are, with high probability, at least 
$$\frac{\nu J(x)}{2} \cdot \frac{\Psi(x,y_0)}{x} = \frac{\nu z_0}{2}$$ 
$y_0$-smooth numbers. Thus we have a total of at least 
$m(z_0)+ \nu z_0 / 2$ rows of $A$, all of whose non-zero entries lie in a set of columns
of size at most $(1+\eps_1)m(z_0)+z_0 e^{-m(z_0)/z_0}$. Hence, if
\[
 \frac{\nu}{2}>\eps_1 m(z_0)/z_0+ e^{-m(z_0)/z_0}
\]
then we obtain a linear relation between the rows. Now, recall that
$m(z_0)/z_0\ge\alpha((1-\eps_0)\eta)$ with high probability, by
\Th{track:m}, and that $\alpha(w)\to\infty$ as $w\to e^{-\gamma}$. Hence, by choosing $\eta$
sufficiently close to $e^{-\gamma}$, and $\eps_0$ sufficiently small, we can make $m(z_0)/z_0$
arbitrarily large. In particular we can force
$e^{-m(z_0)/z_0}<\nu/4$. Since $m(z_0)/z_0\le C_0$ with high probability, 
with $C_0 = C_0(\eta)$ fixed, the result follows by taking $\eps_1$ sufficiently small.
\end{proof}

The proof of the upper bound in \Th{squares:sharp} can be modified to show that the
expected number of linear relations between the rows of $A$ blows up at some $\eta_0 J(x)$,
with $\eta_0<e^{-\gamma}$, thus demonstrating that a straightforward application of the first
moment method cannot give a sharp lower bound on $T(x)$. To see this, let $\eta<e^{-\gamma}$
and consider $N=\eta J(z)$ integers~$a_i$. The number $m_0(z_0)$ of $y_0$-smooth numbers
is binomially distributed with mean $\eta z_0$, but can be much higher. Indeed,
\[
 \Prb\big(m_0(z_0)=2\eta z_0\big)\approx \frac{(\eta z_0)^{2\eta z_0}}{(2\eta z_0)!}e^{-\eta z_0}
 \approx (e/4)^{(1+o(1))\eta z_0}.
\]
However, the remaining numbers are still uniformly distributed over non-smooth numbers, and smooth
numbers have no effect on the algorithm determining the 2-core. Thus if we remove about $\eta z_0$
smooth numbers, the distribution of the 2-core $\cC_A(z_0)$ of the remaining numbers has
approximately the same distribution as if we had started with $N-\eta z_0=(1+o(1))N$ numbers
initially. Thus with probability $(e/4)^{(1+o(1))\eta z_0}$ we have a submatrix of $A$ with
$m(z_0)+\eta z_0$ rows and $(1+\eps_1)m(z_0)+z_0e^{-m(z_0)/z_0}$ non-zero columns. Taking $\eta$
sufficiently close to $e^{-\gamma}$ and $\eps_0,\eps_1$ sufficiently small, we obtain a submatrix
of $A$ with $(\eta-\eps)z_0$ more rows than non-zero columns. This results in at least
$2^{(\eta-\eps)z_0}-1$ non-trivial linear relations between the rows. As this occurs with
probability $(e/4)^{(1+o(1))\eta z_0}$ and $e/4>1/2$, the expected number of linear relations
grows exponentially with~$z_0$, even though we are below the threshold.

Finally we give a proof of \Co{expectation}.

\begin{proof}[Proof of \Co{expectation}.]
Since finding a square product among $\{a_1,\dots,a_t\}$ is independent of finding one among
$\{a_{t+1},\dots,a_{2t}\}$ we have that $\Prb\big( T(x) \ge 2t \big)\le\Prb\big( T(x) \ge t \big)^2$, and more
generally $\Prb\big( T(x) \ge kt \big) \le \Prb\big( T(x) \ge t \big)^k$ for every $k\in\N$.

Setting $t=\big( e^{-\gamma}+\eps \big)J(x)$ and $\theta=\Prb\big( T(x) \ge t \big)$ we have
\begin{align*}
 \E[T(x)]&\le t\,\Prb\big( T(x) \in[0,t) \big) + 2t\,\Prb\big( T(x) \in[t,2t) \big)+3t\,\Prb\big( T(x) \in[2t,3t) \big) + \dots\\
 & = t + t \,\Prb\big( T(x) \ge t \big) + t \,\Prb\big( T(x) \ge 2t \big) + \dots\\
 & \le \big( 1 + \theta+\theta^2+\dots \big)t = \frac{e^{-\gamma}+\eps}{1-\theta}\cdot J(x).
\end{align*}
Since $\eps>0$ is arbitrary and $\theta\to0$ as $x\to\infty$ for any $\eps>0$,
$\E[T(x)] \le \big( e^{-\gamma}+o(1) \big) J(x)$. On the other hand, taking $t= \big( e^{-\gamma} - \eps \big)J(x)$ we have
\[
 \E[T(x)]\ge t\,\Prb\big( T(x) \ge t \big) = \big( 1 + o(1) \big) t = \big( e^{-\gamma}-\eps-o(1)\big)J(x)
\]
as $x\to\infty$, so, since $\eps>0$ was arbitrary, $\E[T(x)]\ge \big( e^{-\gamma}+o(1) \big)J(x)$, as required.
\end{proof}

 \SkipTocEntry\section*{Acknowledgement}

This research was begun while the authors were visiting IMT, Lucca, and partly carried out while
the first and third authors were Visiting Fellow Commoners of Trinity College, Cambridge.
We would like to thank both institutions for providing a wonderful working environment.

\end{document}